\documentclass{article}
\usepackage{amssymb}
\usepackage{amsfonts}
\usepackage{amsmath}
\usepackage{graphicx}
\usepackage[a4paper]{geometry}

\newtheorem{theorem}{Theorem}

\newtheorem{lemma}[theorem]{Lemma}

\newtheorem{proposition}[theorem]{Proposition}
\newtheorem{remark}[theorem]{Remark}

\newenvironment{proof}[1][Proof]{\noindent\textbf{#1.} }{\ \rule{0.5em}{0.5em}}
\input{tcilatex}

\begin{document}

\title{Distributional Results for Thresholding Estimators in
High-Dimensional Gaussian Regression Models \thanks{%
We would like to thank Hannes Leeb, a referee, and an associate editor for
comments on a previous version of the paper.}}
\author{Benedikt M. P\"{o}tscher \and Ulrike Schneider}
\date{June 2011\\
Revised November 2011}
\maketitle

\begin{abstract}
We study the distribution of hard-, soft-, and adaptive soft-thresholding
estimators within a linear regression model where the number of parameters $%
k $ can depend on sample size $n$ and may diverge with $n$. In addition to
the case of known error-variance, we define and study versions of the
estimators when the error-variance is unknown. We derive the finite-sample
distribution of each estimator and study its behavior in the large-sample
limit, also investigating the effects of having to estimate the variance
when the degrees of freedom $n-k$ does not tend to infinity or tends to
infinity very slowly. Our analysis encompasses both the case where the
estimators are tuned to perform consistent variable selection and the case
where the estimators are tuned to perform conservative variable selection.
Furthermore, we discuss consistency, uniform consistency and derive the
uniform convergence rate under either type of tuning.


\end{abstract}

\section{Introduction}

We study the distribution of thresholding estimators such as
hard-thresholding, soft-thresholding, and adaptive soft-thresholding in a
linear regression model when the number of regressors can be large. These
estimators can be viewed as penalized least-squares estimators in the case
of an orthogonal design matrix, with soft-thresholding then coinciding with
the Lasso (introduced by Frank and Friedman (1993), Alliney and Ruzinsky
(1994), and Tibshirani (1996)) and with adaptive soft-thresholding
coinciding with the adaptive Lasso (introduced by Zou (2006)). Thresholding
estimators have of course been discussed earlier in the context of model
selection (see Bauer, P\"{o}tscher and Hackl (1988)) and in the context of
wavelets (see, e.g., Donoho, Johnstone, Kerkyacharian, Picard (1995)).
Contributions concerning distributional properties of thresholding and
penalized least-squares estimators are as follows: Knight and Fu (2000)
study the asymptotic distribution of the Lasso estimator when it is tuned to
act as a conservative variable selection procedure, whereas Zou (2006)
studies the asymptotic distribution of the Lasso and the adaptive Lasso
estimators when they are tuned to act as consistent variable selection
procedures. Fan and Li (2001) and Fan and Peng (2004) study the asymptotic
distribution of the so-called smoothly clipped absolute deviation (SCAD)
estimator when it is tuned to act as a consistent variable selection
procedure. In the wake of Fan and Li (2001) and Fan and Peng (2004) a large
number of papers have been published that derive the asymptotic distribution
of various penalized maximum likelihood estimators under consistent tuning;
see the introduction in P\"{o}tscher and Schneider (2009) for a partial
list. Except for Knight and Fu (2000), all these papers derive the
asymptotic distribution in a fixed-parameter framework. As pointed out in
Leeb and P\"{o}tscher (2005), such a fixed-parameter framework is often
highly misleading in the context of variable selection procedures and
penalized maximum likelihood estimators. For that reason, P\"{o}tscher and
Leeb (2009) and P\"{o}tscher and Schneider (2009) have conducted a detailed
study of the finite-sample as well as large-sample distribution of various
penalized least-squares estimators, adopting a moving-parameter framework
for the asymptotic results. [Related results for so-called
post-model-selection estimators can be found in Leeb and P\"{o}tscher (2003,
2005) and for model averaging estimators in P\"{o}tscher (2006); see also
Sen (1979) and P\"{o}tscher (1991).] The papers by P\"{o}tscher and Leeb
(2009) and P\"{o}tscher and Schneider (2009) are set in the framework of an
orthogonal linear regression model with a fixed number of parameters and
with the error-variance being known.

In the present paper we build on the just mentioned papers P\"{o}tscher and
Leeb (2009) and P\"{o}tscher and Schneider (2009). In contrast to these
papers, we do not assume the number of regressors $k$ to be fixed, but let
it depend on sample size -- thus allowing for high-dimensional models. We
also consider the case where the error-variance is unknown, which in case of
a high-dimensional model creates non-trivial complications as then
estimators for the error-variance will typically not be consistent.
Considering thresholding estimators from the outset in the present paper
allows us also to cover non-orthogonal design. While the asymptotic
distributional results in the known-variance case do not differ in substance
from the results in P\"{o}tscher and Leeb (2009) and P\"{o}tscher and
Schneider (2009), not unexpectedly we observe different asymptotic behavior
in the unknown-variance case if the number of degrees of freedom $n-k$ is
constant, the difference resulting from the non-vanishing variability of the
error-variance estimator in the limit. Less expected is the result that --
under consistent tuning -- for the variable selection probabilities (implied
by all the estimators considered) as well as for the distribution of the
hard-thresholding estimator, estimation of the error-variance still has an
effect asymptotically even if $n-k$ diverges, but does so only slowly.

To give some idea of the theoretical results obtained in the paper we next
present a rough summary of some of these results. For simplicity of
exposition assume for the moment that the $n\times k$ design matrix $X$ is
such that the diagonal elements of $(X^{\prime }X/n)^{-1}$ are equal to $1$,
and that the error-variance $\sigma ^{2}$ is equal to $1$. Let $\tilde{\theta%
}_{H,i}$ denote the hard-thresholding estimator for the $i$-th component $%
\theta _{i}$ of the regression parameter, the threshold being given by $\hat{%
\sigma}\eta _{i,n}$, with $\hat{\sigma}^{2}$ denoting the usual
error-variance estimator and with $\eta _{i,n}$ denoting a tuning parameter.
An infeasible version of the estimator, denoted by $\hat{\theta}_{H,i}$,
which uses $\sigma $ instead of $\hat{\sigma}$, is also considered
(known-variance case). We then show that the uniform rate of convergence of
the hard-thresholding estimator is $n^{-1/2}$ if the threshold satisfies $%
\eta _{i,n}\rightarrow 0$ and $n^{1/2}\eta _{i,n}\rightarrow e_{i}<\infty $
("conservative tuning"), but that the uniform rate is only $\eta _{i,n}$ if
the threshold satisfies $\eta _{i,n}\rightarrow 0$ and $n^{1/2}\eta
_{i,n}\rightarrow \infty $ ("consistent tuning"). The same result also holds
for the soft-thresholding estimator $\tilde{\theta}_{S,i}$ and the adaptive
soft-thresholding estimator $\tilde{\theta}_{AS,i}$, as well as for
infeasible variants of the estimators that use knowledge of $\sigma $
(known-variance case). Furthermore, all possible limits of the centered and
scaled distribution of the hard-thresholding estimator $\tilde{\theta}_{H,i}$
(as well as of the soft- and the adaptive soft-thresholding estimators $%
\tilde{\theta}_{S,i}$ and $\tilde{\theta}_{AS,i}$) under a moving parameter
framework are obtained. Consider first the case of conservative tuning: then
all possible limiting forms of the distribution of $n^{1/2}\left( \tilde{%
\theta}_{H,i}-\theta _{i,n}\right) $ as well as of $n^{1/2}\left( \hat{\theta%
}_{H,i}-\theta _{i,n}\right) $ for arbitrary parameter sequences $\theta
_{i,n}$ are determined. It turns out that -- in the known-variance case --
these limits are of the same functional form as the finite-sample
distribution, i.e., they are a convex combination of a pointmass and an
absolutely continuous distribution that is an excised version of a normal
distribution. In the unknown-variance case, when the number of degrees of
freedom $n-k$ goes to infinity, exactly the same limits arise. However, if $%
n-k$ is constant, the limits are "averaged" versions of the limits in the
known-variance case, the averaging being with respect to the distribution of
the variance estimator $\hat{\sigma}^{2}$. Again these limits have the same
functional form as the corresponding finite-sample distributions. Consider
next the case of consistent tuning: Here the possible limits of $\eta
_{i,n}^{-1}\left( \tilde{\theta}_{H,i}-\theta _{i,n}\right) $ as well as of $%
\eta _{i,n}^{-1}\left( \hat{\theta}_{H,i}-\theta _{i,n}\right) $ have to be
considered, as $\eta _{i,n}$ is the uniform convergence rate. In the
known-variance case the limits are convex combinations of (at most) two
pointmasses, the location of the pointmasses as well as the weights
depending on $\theta _{i,n}$ and $\eta _{i,n}$. In the unknown-variance case
exactly the same limits arise if $n-k$ diverges to infinity sufficiently
fast; however, if $n-k$ is constant or diverges to infinity sufficiently
slowly, the limits are again convex combinations of the same pointmasses,
but with weights that are typically different. The picture for
soft-thresholding and adaptive soft-thresholding is somewhat different: in
the known-variance case, as well as in the unknown-variance case when $n-k$
diverges to infinity, the limits are (single) pointmasses. However, in the
unknown-variance case and if $n-k$ is constant, the limit distribution can
have an absolutely continuous component. It is furthermore useful to point
out that in case of consistent tuning the sequence of distributions of $%
n^{1/2}\left( \tilde{\theta}_{H,i}-\theta _{i,n}\right) $ is not
stochastically bounded in general (since $\eta _{i,n}$ is the uniform
convergence rate), and the same is true for soft-thresholding $\tilde{\theta}%
_{S,i}$ and adaptive soft-thresholding $\tilde{\theta}_{AS,i}$. This throws
a light on the fragility of the oracle-property, see Section \ref{oracle}
for more discussion.

While our theoretical results for the thresholding estimators immediately
apply to Lasso and adaptive Lasso in case of orthogonal design, this is not
so in the non-orthogonal case. In order to get some insight into the
finite-sample distribution of the latter estimators also in the
non-orthogonal case, we numerically compare the distribution of Lasso and
adaptive Lasso with their thresholding counterparts in a simulation study.

\bigskip

The main take-away messages of the paper can be summarized as follows:

\begin{itemize}
\item The finite-sample distributions of the various thresholding estimators
considered are highly non-normal, the distributions being in each case a
convex combination of pointmass and an absolutely continuous (non-normal)\
component.

\item The non-normality persists asymptotically in a moving parameter
framework.

\item Results in the unknown-variance case are obtained from the
corresponding results in the known-variance case by smoothing with respect
to the distribution of $\hat{\sigma}$. In line with this, one would expect
the limiting behavior in the unknown-variance case to coincide with the
limiting behavior in the known-variance whenever the degrees of freedom $n-k$
diverge to infinity. This indeed turns out to be so for some of the results,
but not for others where we see that the speed of divergence of $n-k$
matters.

\item In case of conservative tuning the estimators have the expected
uniform convergence rate, which is $n^{-1/2}$ under the simplified
assumptions of the above discussion, whereas under consistent tuning the
uniform rate is slower, namely $\eta _{i,n}$ under the simplified
assumptions of the above discussion. This is intimately connected with the
fact that the so-called `oracle property' paints a misleading picture of the
performance of the estimators.

\item The numerical study suggests that the results for the thresholding
estimators $\tilde{\theta}_{S,i}$ and $\tilde{\theta}_{AS,i}$ qualitatively
apply also to the (components of) the Lasso and the adaptive Lasso as long
as the design matrix is not too ill-conditioned.
\end{itemize}

The paper is organized as follows. We introduce the model and define the
estimators in Section \ref{model}. Section \ref{variable} treats the
variable selection probabilities implied by the estimators. Consistency,
uniform consistency, and uniform convergence rates are discussed in Section %
\ref{minimax}. We derive the finite-sample distribution of each estimator in
Section \ref{FS} and study the large-sample behavior of these in Section \ref%
{LS}. A numerical study of the finite-sample distribution of Lasso and
adaptive Lasso can be found in Section \ref{numstudy}. All proofs are
relegated to Section \ref{prfs}.

\section{The Model and the Estimators\label{model}}

Consider the linear regression model

\begin{equation*}
Y=X\theta +u
\end{equation*}%
with $Y$ an $n\times 1$ vector, $X$ a nonstochastic $n\times k$ matrix of
rank $k\geq 1$, and $u\sim N(0,\sigma ^{2}I_{n})$, $0<\sigma <\infty $. We
allow $k$, the number of columns of $X$, as well as the entries of $Y$, $X$,
and $u$ to depend on sample size $n$ (in fact, also the probability spaces
supporting $Y$ and $u$ may depend on $n$), although we shall almost always
suppress this dependence on $n$ in the notation. Note that this framework
allows for high-dimensional regression models, where the number of
regressors $k$ is large compared to sample size $n$, as well as for the more
classical situation where $k$ is much smaller than $n$. Furthermore, let $%
\xi _{i,n}$ denote the nonnegative square root of $((X^{\prime
}X/n)^{-1})_{ii}$, the $i$-th diagonal element of $(X^{\prime }X/n)^{-1}$.
Now let%
\begin{equation*}
\hat{\theta}_{LS}=\left( X^{\prime }X\right) ^{-1}X^{\prime }Y
\end{equation*}%
\begin{equation*}
\hat{\sigma}^{2}=(n-k)^{-1}(Y-X\hat{\theta}_{LS})^{\prime }(Y-X\hat{\theta}%
_{LS})
\end{equation*}%
denote the least-squares estimator for $\theta $ and the associated
estimator for $\sigma ^{2}$, the latter being defined only if $n>k$. The
hard-thresholding estimator $\tilde{\theta}_{H}$ is defined via its
components as follows 
\begin{equation*}
\tilde{\theta}_{H,i}=\tilde{\theta}_{H,i}(\eta _{i,n})=\hat{\theta}_{LS,i}%
\boldsymbol{1}\left( \left\vert \hat{\theta}_{LS,i}\right\vert >\hat{\sigma}%
\xi _{i,n}\eta _{i,n}\right) ,
\end{equation*}%
where the tuning parameters $\eta _{i,n}$ are positive real numbers and $%
\hat{\theta}_{LS,i}$ denotes the $i$-th component of the least-squares
estimator. We shall also need to consider its infeasible counterpart $\hat{%
\theta}_{H}$ given by 
\begin{equation*}
\hat{\theta}_{H,i}=\hat{\theta}_{H,i}(\eta _{i,n})=\hat{\theta}_{LS,i}%
\boldsymbol{1}\left( \left\vert \hat{\theta}_{LS,i}\right\vert >\sigma \xi
_{i,n}\eta _{i,n}\right) .
\end{equation*}%
The soft-thresholding estimator $\tilde{\theta}_{S}$ and its infeasible
counterpart $\hat{\theta}_{S}$\ are given by%
\begin{equation*}
\tilde{\theta}_{S,i}=\tilde{\theta}_{S,i}(\eta _{i,n})=\func{sign}(\hat{%
\theta}_{LS,i})\left( \left\vert \hat{\theta}_{LS,i}\right\vert -\hat{\sigma}%
\xi _{i,n}\eta _{i,n}\right) _{+}
\end{equation*}%
and 
\begin{equation*}
\hat{\theta}_{S,i}=\hat{\theta}_{S,i}(\eta _{i,n})=\func{sign}(\hat{\theta}%
_{LS,i})\left( \left\vert \hat{\theta}_{LS,i}\right\vert -\sigma \xi
_{i,n}\eta _{i,n}\right) _{+},
\end{equation*}%
where $\left( \cdot \right) _{+}=\max (\cdot ,0)$. Finally, the adaptive
soft-thresholding estimator $\tilde{\theta}_{AS}$ and its infeasible
counterpart $\hat{\theta}_{AS}$ are defined via 
\begin{eqnarray*}
\tilde{\theta}_{AS,i} &=&\tilde{\theta}_{AS,i}(\eta _{i,n})=\hat{\theta}%
_{LS,i}\left( 1-\hat{\sigma}^{2}\xi _{i,n}^{2}\eta _{i,n}^{2}/\hat{\theta}%
_{LS,i}^{2}\right) _{+} \\
&=&\left\{ 
\begin{array}{cc}
0 & \text{if }\left\vert \hat{\theta}_{LS,i}\right\vert \leq \hat{\sigma}\xi
_{i,n}\eta _{i,n} \\ 
\hat{\theta}_{LS,i}-\hat{\sigma}^{2}\xi _{i,n}^{2}\eta _{i,n}^{2}/\hat{\theta%
}_{LS,i} & \text{if }\left\vert \hat{\theta}_{LS,i}\right\vert >\hat{\sigma}%
\xi _{i,n}\eta _{i,n}%
\end{array}%
\right.
\end{eqnarray*}%
and 
\begin{eqnarray*}
\hat{\theta}_{AS,i} &=&\hat{\theta}_{AS,i}(\eta _{i,n})=\hat{\theta}%
_{LS,i}\left( 1-\sigma ^{2}\xi _{i,n}^{2}\eta _{i,n}^{2}/\hat{\theta}%
_{LS,i}^{2}\right) _{+} \\
&=&\left\{ 
\begin{array}{cc}
0 & \text{if }\left\vert \hat{\theta}_{LS,i}\right\vert \leq \sigma \xi
_{i,n}\eta _{i,n} \\ 
\hat{\theta}_{LS,i}-\sigma ^{2}\xi _{i,n}^{2}\eta _{i,n}^{2}/\hat{\theta}%
_{LS,i} & \text{if }\left\vert \hat{\theta}_{LS,i}\right\vert >\sigma \xi
_{i,n}\eta _{i,n}%
\end{array}%
\right. .
\end{eqnarray*}

Note that $\tilde{\theta}_{H}$, $\tilde{\theta}_{S}$, and $\tilde{\theta}%
_{AS}$ as well as their infeasible counterparts are equivariant under
scaling of the columns of $(Y:X)$ by non-zero column-specific scale factors.
We have chosen to let the thresholds $\hat{\sigma}\xi _{i,n}\eta _{i,n}$ ($%
\sigma \xi _{i,n}\eta _{i,n}$, respectively) depend explicitly on $\hat{%
\sigma}$ ($\sigma $, respectively) and $\xi _{i,n}$ in order to give $\eta
_{i,n}$ an interpretation independent of the values of $\sigma $ and $X$.
Furthermore, often $\eta _{i,n}$ will be chosen independently of $i$, i.e., $%
\eta _{i,n}=\eta _{n}$ where $\eta _{n}$ is a positive real number. Clearly,
for the feasible versions we always need to assume $n>k$, whereas for the
infeasible versions $n\geq k$ suffices.

We note the simple fact that%
\begin{equation}
0\leq \tilde{\theta}_{S,i}\leq \tilde{\theta}_{AS,i}\leq \tilde{\theta}%
_{H,i}\leq \hat{\theta}_{LS,i}  \label{ineq_1}
\end{equation}%
holds on the event that $\hat{\theta}_{LS,i}\geq 0$, and that%
\begin{equation}
\hat{\theta}_{LS,i}\leq \tilde{\theta}_{H,i}\leq \tilde{\theta}_{AS,i}\leq 
\tilde{\theta}_{S,i}\leq 0  \label{ineq_2}
\end{equation}%
holds on the event that $\hat{\theta}_{LS,i}\leq 0$. Analogous inequalities
hold for the infeasible versions of the estimators.

\begin{remark}
\label{LASSO}\normalfont\emph{(Lasso) }(i) Consider the objective function 
\begin{equation*}
(Y-X\theta )^{\prime }(Y-X\theta )+2n\hat{\sigma}\sum_{i=1}^{k}\eta
_{i,n}^{\prime }\left\vert \theta _{i}\right\vert ,
\end{equation*}%
where $\eta _{i,n}^{\prime }$ are positive real numbers. It is well-known
that a unique minimizer $\tilde{\theta}_{L}$ of this objective function
exists, the Lasso-estimator. It is easy to see that in case $X^{\prime }X$
is diagonal we have 
\begin{equation*}
\tilde{\theta}_{L,i}=\func{sign}(\hat{\theta}_{LS,i})\left( \left\vert \hat{%
\theta}_{LS,i}\right\vert -\hat{\sigma}\eta _{i,n}^{\prime }\xi
_{i,n}^{2}\right) _{+}.
\end{equation*}%
Hence, in the case of diagonal $X^{\prime }X$, the components $\tilde{\theta}%
_{L,i}$ of the Lasso reduce to soft-thresholding estimators with appropriate
thresholds; in particular, $\tilde{\theta}_{L,i}$ coincides with $\tilde{%
\theta}_{S,i}$ for the choice $\eta _{i,n}^{\prime }=\eta _{i,n}\xi
_{i,n}^{-1}$. Therefore all results derived below for soft-thresholding
immediately give corresponding results for the Lasso as well as for the
Dantzig-selector in the diagonal case. We shall abstain from spelling out
further details.

(ii) Sometimes $\eta _{i,n}^{\prime }$ in the definition of the Lasso is
chosen independently of $i$; more reasonable choices seem to be (a) $\eta
_{i,n}^{\prime }=\eta _{i,n}\psi _{i,n}$ (where $\psi _{i,n}$ denotes the
nonnegative square root of the $i$-th diagonal element of $(X^{\prime }X/n)$%
), and (b) $\eta _{i,n}^{\prime }=\eta _{i,n}\xi _{i,n}^{-1}$ where $\eta
_{i,n}$ are positive real numbers (not depending on the design matrix and
often not on $i$) as then $\eta _{i,n}$ again has an interpretation
independent of the values of $\sigma $ and $X$. Note that in case (a) or (b)
the solution of the optimization problem is equivariant under scaling of the
columns of $(Y:X)$ by non-zero column-specific scale factors.

(iii) Similar results obviously hold for the infeasible versions of the
estimators.
\end{remark}

\begin{remark}
\label{ALASSO}\normalfont\emph{(Adaptive Lasso)} Consider the objective
function 
\begin{equation*}
(Y-X\theta )^{\prime }(Y-X\theta )+2n\hat{\sigma}^{2}\sum_{i=1}^{k}(\eta
_{i,n}^{\prime })^{2}\left\vert \theta _{i}\right\vert /\left\vert \hat{%
\theta}_{LS,i}\right\vert ,
\end{equation*}%
where $\eta _{i,n}^{\prime }$ are positive real numbers. This is the
objective function of the adaptive Lasso (where often $\eta _{i,n}^{\prime
}=\eta _{n}^{\prime }$ is chosen independent of $i$). Again the minimizer $%
\tilde{\theta}_{AL}$ exists and is unique (at least on the event where $\hat{%
\theta}_{LS,i}\neq 0$ for all $i$). Clearly, $\tilde{\theta}_{AL}$ is
equivariant under scaling of the columns of $(Y:X)$ by non-zero
column-specific scale factors provided $\eta _{i,n}^{\prime }$ does not
depend on the design matrix. It is easy to see that in case $X^{\prime }X$
is diagonal we have 
\begin{equation*}
\tilde{\theta}_{AL,i}=\hat{\theta}_{LS,i}\left( 1-\hat{\sigma}^{2}\xi
_{i,n}^{2}\left( \eta _{i,n}^{\prime }\right) ^{2}/\hat{\theta}%
_{LS,i}^{2}\right) _{+}.
\end{equation*}%
Hence, in the case of diagonal $X^{\prime }X$, the components $\tilde{\theta}%
_{AL,i}$ of the adaptive Lasso reduce to the adaptive soft-thresholding
estimators $\tilde{\theta}_{AS,i}$ (for $\eta _{i,n}^{\prime }=\eta _{i,n}$%
). Therefore all results derived below for adaptive soft-thresholding
immediately give corresponding results for the adaptive Lasso in the
diagonal case. We shall again abstain from spelling out further details.
Similar results obviously hold for the infeasible versions of the estimators.
\end{remark}

\begin{remark}
\normalfont\emph{(Other estimators) }(i) The adaptive Lasso as defined in
Zou (2006) has an additional tuning parameter $\gamma $. We consider
adaptive soft-thresholding only for the case $\gamma =1$, since otherwise
the estimator is not equivariant in the sense described above. Nonetheless
an analysis for the case $\gamma \neq 1$, similar to the analysis in this
paper, is possible in principle.

(ii) An analysis of a SCAD-based thresholding estimator is given in P\"{o}%
tscher and Leeb (2009) in the known-variance case. [These results are given
in the orthogonal design case, but easily generalize to the non-orthogonal
case.] The results obtained there for SCAD-based thresholding are similar in
spirit to the results for the other thresholding estimators considered here.
The unknown-variance case could also be analyzed in principle, but we
refrain from doing so for the sake of brevity.

(iii) Zhang (2010) introduced the so-called minimax concave penalty (MCP)\
to be used for penalized least-squares estimation. Apart from the usual
tuning parameter, MCP also depends on a shape parameter $\gamma $. It turns
out that the thresholding estimator based on MCP coincides with
hard-thresholding in case $\gamma \leq 1$, and thus is covered by the
analysis of the present paper. In case $\gamma >1$, the MCP-based
thresholding estimator could similarly be analyzed, especially since the
functional form of the MCP-based thresholding estimator is relatively simple
(namely, a piecewise linear function of the least-squares estimator). We do
not provide such an analysis for brevity.
\end{remark}

\emph{For all asymptotic considerations in this paper we shall always assume
without further mentioning that }$\xi _{i,n}^{2}/n=((X^{\prime
}X)^{-1})_{ii} $\emph{\ satisfies}%
\begin{equation}
\sup_{n}\xi _{i,n}^{2}/n<\infty  \label{xi}
\end{equation}%
\emph{for every fixed }$i\geq 1$\emph{\ satisfying }$i\leq k(n)$\emph{\ for
large enough }$n$\emph{.} The case excluded by assumption (\ref{xi}) seems
to be rather uninteresting as unboundedness of $\xi _{i,n}^{2}/n$\ means
that the information contained in the regressors gets weaker with increasing
sample size (at least along a subsequence); in particular, this implies
(coordinate-wise) inconsistency of the least-squares estimator. [In fact, if 
$k$ as well as the elements of $X$ do not depend on $n$, this case is
actually impossible as $\xi _{i,n}^{2}/n$ is then necessarily monotonically
nonincreasing.]

The following notation will be used in the paper: Let $\mathbb{\bar{R}}$
denote the extended real line $\mathbb{R\cup }\left\{ -\infty ,\infty
\right\} $ endowed with the usual topology. On $\mathbb{N\cup }\left\{
\infty \right\} $ we shall consider the topology it inherits from $\mathbb{%
\bar{R}}$. Furthermore, $\Phi $ and $\phi $ denote the cumulative
distribution function (cdf) and the probability density function (pdf) of a
standard normal distribution, respectively. By $T_{m,c}$ we denote the cdf
of a non-central $T$-distribution with $m\in \mathbb{N}$ degrees of freedom
and non-centrality parameter $c\in \mathbb{R}$. In the central case, i.e., $%
c=0$, we simply write $T_{m}$. We use the convention $\Phi (\infty )=1$, $%
\Phi (-\infty )=0$ with a similar convention for $T_{m,c}$.

\section{Variable Selection Probabilities\label{variable}}

The estimators $\tilde{\theta}_{H}$, $\tilde{\theta}_{S}$, and $\tilde{\theta%
}_{AS}$ can be viewed as performing variable selection in the sense that
these estimators set components of $\theta $ exactly equal to zero with
positive probability. In this section we study the variable selection
probability $P_{n,\theta ,\sigma }\left( \tilde{\theta}_{i}\neq 0\right) $,
where $\tilde{\theta}_{i}$ stands for any of the estimators $\tilde{\theta}%
_{H,i}$, $\tilde{\theta}_{S,i}$, and $\tilde{\theta}_{AS,i}$. Since these
probabilities are the same for any of the three estimators considered we
shall drop the subscripts $H$, $S$, and $AS$ in this section. We use the
same convention also for the variable selection probabilities of the
infeasible versions.

\subsection{Known-Variance Case}

Since $P_{n,\theta ,\sigma }\left( \hat{\theta}_{i}\neq 0\right)
=1-P_{n,\theta ,\sigma }\left( \hat{\theta}_{i}=0\right) $ it suffices to
study the variable deletion probability%
\begin{equation}
P_{n,\theta ,\sigma }\left( \hat{\theta}_{i}=0\right) =\Phi \left(
n^{1/2}\left( -\theta _{i}/(\sigma \xi _{i,n})+\eta _{i,n}\right) \right)
-\Phi \left( n^{1/2}\left( -\theta _{i}/(\sigma \xi _{i,n})-\eta
_{i,n}\right) \right) .  \label{select_prob}
\end{equation}

As can be seen from the above formula, $P_{n,\theta ,\sigma }\left( \hat{%
\theta}_{i}=0\right) $ depends on $\theta $ only via $\theta _{i}$. We first
study the variable selection/deletion probabilities under a
"fixed-parameter" asymptotic framework.

\begin{proposition}
\label{select_prob_pointwise}Let $0<\sigma <\infty $ be given. For every $%
i\geq 1$ satisfying $i\leq k=k(n)$ for large enough $n$ we have:

(a) A necessary and sufficient condition for $P_{n,\theta ,\sigma }\left( 
\hat{\theta}_{i}=0\right) \rightarrow 0$ as $n\rightarrow \infty $ for all $%
\theta $ satisfying $\theta _{i}\neq 0$ ($\theta _{i}$ not depending on $n$)
is $\xi _{i,n}\eta _{i,n}\rightarrow 0$.

(b) A necessary and sufficient condition for $P_{n,\theta ,\sigma }\left( 
\hat{\theta}_{i}=0\right) \rightarrow 1$ as $n\rightarrow \infty $ for all $%
\theta $ satisfying $\theta _{i}=0$ is $n^{1/2}\eta _{i,n}\rightarrow \infty 
$.

(c) A necessary and sufficient condition for $P_{n,\theta ,\sigma }\left( 
\hat{\theta}_{i}=0\right) \rightarrow c_{i}<1$ as $n\rightarrow \infty $ for
all $\theta $ satisfying $\theta _{i}=0$ is $n^{1/2}\eta _{i,n}\rightarrow
e_{i}$, $0\leq e_{i}<\infty $. The constant $c_{i}$ is then given by $%
c_{i}=\Phi \left( e_{i}\right) -\Phi \left( -e_{i}\right) $.
\end{proposition}

Part (a) of the above proposition gives a necessary and sufficient condition
for the procedure to correctly detect nonzero coefficients with probability
converging to $1$. Part (b) gives a necessary and sufficient condition for
correctly detecting zero coefficients with probability converging to $1$.

\begin{remark}
\label{uninteresting}\normalfont If $\xi _{i,n}/n^{1/2}$ does not converge
to zero, the conditions on $\eta _{i,n}$ in Parts (a) and (b) are
incompatible; also the conditions in Parts (a) and (c) are then incompatible
(except when $e_{i}=0$). However, the case where $\xi _{i,n}/n^{1/2}$ does
not converge to zero is of little interest as the least-squares estimator $%
\hat{\theta}_{LS,i}$ is then not consistent.
\end{remark}

\begin{remark}
\normalfont\emph{(Speed of convergence in Proposition \ref%
{select_prob_pointwise}) }(i) The speed of convergence in (a) is $\xi
_{i,n}\eta _{i,n}$ in case $n^{1/2}\xi _{i,n}^{-1}$ is bounded (an
uninteresting case as noted above); if$\ n^{1/2}\xi _{i,n}^{-1}\rightarrow
\infty $, the speed of convergence in (a) is not slower than $\exp \left(
-cn\xi _{i,n}^{-2}\right) /\left( n^{1/2}\xi _{i,n}^{-1}\right) $ for some
suitable $c>0$ depending on $\theta _{i}/\sigma $.

(ii) The speed of convergence in (b) is $\exp \left( -0.5n\eta
_{i,n}^{2}\right) /\left( n^{1/2}\eta _{i,n}\right) $. In (c) the speed of
convergence is given by the rate at which $n^{1/2}\eta _{i,n}$ approaches $%
e_{i}$.

[For the above results we have made use of Lemma VII.1.2 in Feller (1957).]
\end{remark}

\begin{remark}
\normalfont For $\theta \in \mathbb{R}^{k(n)}$ let $A_{n}(\theta )=\left\{
i:1\leq i\leq k(n),\theta _{i}\neq 0\right\} $. Then (i) for every $i\in
A_{n}(\theta )$ 
\begin{equation*}
P_{n,\theta ,\sigma }\left( \hat{\theta}_{i}=0\right) \leq P_{n,\theta
,\sigma }\left( \bigcup_{j\in A_{n}(\theta )}\left\{ \hat{\theta}%
_{j}=0\right\} \right) \leq \sum_{j\in A_{n}(\theta )}P_{n,\theta ,\sigma
}\left( \hat{\theta}_{j}=0\right) .
\end{equation*}%
Suppose now that the entries of $\theta $ do not change with $n$ (although
the dimension of $\theta $ may depend on $n$).\footnote{%
More precisely, this means that $\theta $ is made up of the initial $k(n)$
elements of a fixed element of $\mathbb{R}^{\infty }$.} Then, given that $%
\limfunc{card}(A_{n}(\theta ))$ is bounded (this being in particular the
case if $k(n)$ is bounded), the probability of incorrect non-detection of at
least one nonzero coefficient converges to $0$ if and only if $\xi
_{i,n}\eta _{i,n}\rightarrow 0$ as $n\rightarrow \infty $ for every $i\in
A_{n}(\theta )$. [If $\limfunc{card}(A_{n}(\theta ))$ is unbounded then this
probability converges to $0$, e.g., if $\xi _{i,n}\eta _{i,n}\rightarrow 0$
and $n^{1/2}\xi _{i,n}^{-1}\rightarrow \infty $ as $n\rightarrow \infty $
for every $i\in A_{n}(\theta )$ and $\inf_{i\in A_{n}(\theta )}\left\vert
\theta _{i}\right\vert >0$ and $\sum_{i\in A_{n}(\theta )}\exp \left( -cn\xi
_{i,n}^{-2}\right) /\left( n^{1/2}\xi _{i,n}^{-1}\right) \rightarrow 0$ as $%
n\rightarrow \infty $ for a suitable $c$ that is determined by $\inf_{i\in
A_{n}(\theta )}\left\vert \theta _{i}\right\vert /\sigma $.]

(ii) For every $i\notin A_{n}(\theta )$ we have%
\begin{eqnarray*}
P_{n,\theta ,\sigma }\left( \hat{\theta}_{i}=0\right) &\geq &P_{n,\theta
,\sigma }\left( \bigcap_{j\notin A_{n}(\theta )}\left\{ \hat{\theta}%
_{j}=0\right\} \right) =1-P_{n,\theta ,\sigma }\left( \bigcup_{j\notin
A_{n}(\theta )}\left\{ \hat{\theta}_{j}\neq 0\right\} \right) \\
&\geq &1-\sum_{j\notin A_{n}(\theta )}\left[ 1-P_{n,\theta ,\sigma }\left( 
\hat{\theta}_{j}=0\right) \right] .
\end{eqnarray*}%
Suppose again that the entries of $\theta $ do not change with $n$. Then,
given that $\limfunc{card}(A_{n}^{c}(\theta ))$ is bounded (this being in
particular the case if $k(n)$ is bounded), the probability of incorrectly
classifying at least one zero parameter as a non-zero one converges to $0$
as $n\rightarrow \infty $ if and only if $n^{1/2}\eta _{i,n}\rightarrow
\infty $ for every $i\in A_{n}(\theta )$. [If $\limfunc{card}%
(A_{n}^{c}(\theta ))$ is unbounded then this probability converges to $0$,
e.g., if $\sum_{i\notin A_{n}(\theta )}\exp \left( -0.5n\eta
_{i,n}^{2}\right) /\left( n^{1/2}\eta _{i,n}\right) \rightarrow 0$ as $%
n\rightarrow \infty $.]

(iii) In case $X^{\prime }X$ is diagonal, the relevant probabilities $%
P_{n,\theta ,\sigma }\left( \bigcup_{i\in A_{n}(\theta )}\left\{ \hat{\theta}%
_{i}=0\right\} \right) $ as well as $P_{n,\theta ,\sigma }\left(
\bigcap_{i\notin A_{n}(\theta )}\left\{ \hat{\theta}_{i}=0\right\} \right) $
can be directly expressed in terms of products of $P_{n,\theta ,\sigma
}\left( \hat{\theta}_{i}=0\right) $ or $1-P_{n,\theta ,\sigma }\left( \hat{%
\theta}_{i}=0\right) $, and Proposition \ref{select_prob_pointwise} can then
be applied.
\end{remark}

Since the fixed-parameter asymptotic framework often gives a misleading
impression of the actual behavior of a variable selection procedure
(cf.~Leeb and P\"{o}tscher (2005), P\"{o}tscher and Leeb (2009)) we turn to
a "moving-parameter" framework next, i.e., we allow the elements of $\theta $
as well as $\sigma $ to depend on sample size $n$. In the proposition to
follow (and all subsequent large-sample results) we shall concentrate only
on the case where $\xi _{i,n}\eta _{i,n}\rightarrow 0$ as $n\rightarrow
\infty $, since otherwise the estimators $\hat{\theta}_{i}$ are not even
consistent for $\theta _{i}$ as a consequence of Proposition\emph{\ }\ref%
{select_prob_pointwise}, cf.~also Theorem \ref{thresh_consistency} below.
Given the condition $\xi _{i,n}\eta _{i,n}\rightarrow 0$, we shall then
distinguish between the case $n^{1/2}\eta _{i,n}\rightarrow e_{i}$, $0\leq
e_{i}<\infty $, and the case $n^{1/2}\eta _{i,n}\rightarrow \infty $, which
in light of Proposition \ref{select_prob_pointwise} we shall call the case
of "conservative tuning" and the case of "consistent tuning", respectively.%
\footnote{%
There is no loss of generality here in assuming convergence of $n^{1/2}\eta
_{i,n}$ to a (finite or infinite) limit, in the sense that this convergence
can, for any given sequence $n^{1/2}\eta _{i,n}$, be achieved along suitable
subsequences in light of compactness of the extended real line.}

\begin{proposition}
\label{select_prob_moving_par}Suppose that for given $i\geq 1$ satisfying $%
i\leq k=k(n)$ for large enough $n$ we have $\xi _{i,n}\eta _{i,n}\rightarrow
0$ and $n^{1/2}\eta _{i,n}\rightarrow e_{i}$ where $0\leq e_{i}\leq \infty $.

(a) Assume $e_{i}<\infty $. Suppose that the true parameters $\theta
^{(n)}=(\theta _{1,n},\dots ,\theta _{k_{n},n})\in \mathbb{R}^{k_{n}}$ and $%
\sigma _{n}\in (0,\infty )$ satisfy $n^{1/2}\theta _{i,n}/(\sigma _{n}\xi
_{i,n})\rightarrow \nu _{i}\in \mathbb{\bar{R}}$. Then 
\begin{equation*}
\lim_{n\rightarrow \infty }P_{n,\theta ^{(n)},\sigma _{n}}\left( \hat{\theta}%
_{i}=0\right) =\Phi \left( -\nu _{i}+e_{i}\right) -\Phi \left( -\nu
_{i}-e_{i}\right) .
\end{equation*}

(b) Assume $e_{i}=\infty $. Suppose that the true parameters $\theta
^{(n)}=(\theta _{1,n},\ldots ,\theta _{k_{n},n})\in \mathbb{R}^{k_{n}}$ and $%
\sigma _{n}\in (0,\infty )$ satisfy $\theta _{i,n}/(\sigma _{n}\xi
_{i,n}\eta _{i,n})\rightarrow \zeta _{i}\in \mathbb{\bar{R}}$. Then

\qquad 1. $\left\vert \zeta _{i}\right\vert <1$ implies $\lim_{n\rightarrow
\infty }P_{n,\theta ^{(n)},\sigma _{n}}\left( \hat{\theta}_{i}=0\right) =1$.

\qquad 2. $\left\vert \zeta _{i}\right\vert >1$ implies $\lim_{n\rightarrow
\infty }P_{n,\theta ^{(n)},\sigma _{n}}\left( \hat{\theta}_{i}=0\right) =0$.

\qquad 3. $\left\vert \zeta _{i}\right\vert =1$ and $r_{i,n}:=n^{1/2}\left(
\eta _{i,n}-\zeta _{i}\theta _{i,n}/(\sigma _{n}\xi _{i,n})\right)
\rightarrow r_{i}$, for some $r_{i}\in \mathbb{\bar{R}}$, imply 
\begin{equation*}
\lim_{n\rightarrow \infty }P_{n,\theta ^{(n)},\sigma _{n}}\left( \hat{\theta}%
_{i}=0\right) =\Phi (r_{i}).
\end{equation*}
\end{proposition}

In a fixed-parameter asymptotic analysis, which in Proposition \ref%
{select_prob_moving_par} corresponds to the case $\theta _{i,n}\equiv \theta
_{i}$ and $\sigma _{n}\equiv \sigma $, the limit of the probabilities $%
P_{n,\theta ,\sigma }\left( \hat{\theta}_{i}=0\right) $ is always $0$ in
case $\theta _{i}\neq 0$, and is $1$ in case $\theta _{i}=0$ and consistent
tuning (it is $\Phi \left( e_{i}\right) -\Phi \left( -e_{i}\right) $ in case 
$\theta _{i}=0$ and conservative tuning); this does clearly not properly
capture the finite-sample behavior of these probabilities. The
moving-parameter asymptotic analysis underlying Proposition \ref%
{select_prob_moving_par} better captures the finite-sample behavior and,
e.g., allows for limits other than $0$ and $1$ even in the case of
consistent tuning. In particular, Proposition \ref{select_prob_moving_par}
shows that the convergence of the variable selection/deletion probabilities
to their limits in a fixed-parameter asymptotic framework is not uniform in $%
\theta _{i}$, and this non-uniformity is local in the sense that it occurs
in an arbitrarily small neighborhood of $\theta _{i}=0$ (holding the value
of $\sigma >0$ fixed).\footnote{%
More generally, the non-uniformity arises for $\theta _{i}/\sigma $ in a
neighborhood of zero.} Furthermore, the above proposition entails that under
consistent tuning deviations from $\theta _{i}=0$ of larger order than under
conservative tuning go unnoticed asymptotically with probability 1 by the
variable selection procedure corresponding to $\hat{\theta}_{i}$. For more
discussion in a special case (which in its essence also applies here) see P%
\"{o}tscher and Leeb (2009).

\begin{remark}
\normalfont\emph{(Speed of convergence in Proposition \ref%
{select_prob_moving_par}) }(i) The speed of convergence in (a) is given by
the slower of the rate at which $n^{1/2}\eta _{i,n}$ approaches $e_{i}$ and $%
n^{1/2}\theta _{i,n}/(\sigma _{n}\xi _{i,n})$ approaches $\nu _{i}$ provided
that $\left\vert \nu _{i}\right\vert <\infty $; if $\left\vert \nu
_{i}\right\vert =\infty $, the speed of convergence is not slower than 
\begin{equation*}
\exp \left( -cn\theta _{i,n}^{2}/(\sigma _{n}^{2}\xi _{i,n}^{2})\right)
/\left\vert n^{1/2}\theta _{i,n}/(\sigma _{n}\xi _{i,n})\right\vert
\end{equation*}%
for any $c<1/2$.

(ii) The speed of convergence in (b1) is not slower than $\exp \left(
-cn\eta _{i,n}^{2}\right) /\left( n^{1/2}\eta _{i,n}\right) $ where $c$
depends on $\zeta _{i}$. The same is true in case (b2) provided $\left\vert
\zeta _{i}\right\vert <\infty $; if $\left\vert \zeta _{i}\right\vert
=\infty $, the speed of convergence is not slower than $\exp \left(
-cn\theta _{i,n}^{2}/(\sigma _{n}^{2}\xi _{i,n}^{2})\right) /\left\vert
n^{1/2}\theta _{i,n}/(\sigma _{n}\xi _{i,n})\right\vert $ for every $c<1/2$.
In case (b3) the speed of convergence is not slower than the speed of
convergence of%
\begin{equation*}
\max \left( \exp \left( -cn\eta _{i,n}^{2}\right) /\left( n^{1/2}\eta
_{i,n}\right) ,\left\vert r_{i,n}-r_{i}\right\vert \right)
\end{equation*}%
for any $c<2$ in case $\left\vert r_{i}\right\vert <\infty $; in case $%
\left\vert r_{i}\right\vert =\infty $ it is not slower than%
\begin{equation*}
\max \left( \exp \left( -cn\eta _{i,n}^{2}\right) /\left( n^{1/2}\eta
_{i,n}\right) ,\exp \left( -0.5r_{i,n}^{2}\right) /\left\vert
r_{i,n}\right\vert \right)
\end{equation*}%
for any $c<2$.
\end{remark}

The preceding remark corrects and clarifies the remarks at the end of
Section 3 in P\"{o}tscher and Leeb (2009) and Section 3.1 in P\"{o}tscher
and Schneider (2009).

\subsection{Unknown-Variance Case}

In the unknown-variance case the finite-sample variable selection/deletion
probabilities can be obtained as follows:%
\begin{align}
P_{n,\theta ,\sigma }\left( \tilde{\theta}_{i}=0\right) & =P_{n,\theta
,\sigma }\left( \left\vert \hat{\theta}_{LS,i}\right\vert \leq \hat{\sigma}%
\xi _{i,n}\eta _{i,n}\right)  \notag \\
& =\int_{0}^{\infty }P_{n,\theta ,\sigma }\left( \left\vert \hat{\theta}%
_{LS,i}\right\vert \leq \hat{\sigma}\xi _{i,n}\eta _{i,n}\mid \hat{\sigma}%
=s\sigma \right) \rho _{n-k}(s)ds  \notag \\
& =\int_{0}^{\infty }P_{n,\theta ,\sigma }\left( \hat{\theta}_{i}(s\eta
_{i,n})=0\right) \rho _{n-k}(s)ds  \notag \\
& =\int_{0}^{\infty }\left[ \Phi \left( n^{1/2}\left( -\theta _{i}/(\sigma
\xi _{i,n})+s\eta _{i,n}\right) \right) \right.  \notag \\
& \qquad \left. -\Phi \left( n^{1/2}\left( -\theta _{i}/(\sigma \xi
_{i,n})-s\eta _{i,n}\right) \right) \right] \rho _{n-k}(s)ds  \notag \\
& =T_{n-k,n^{1/2}\theta _{i}/(\sigma \xi _{i,n})}\left( n^{1/2}\eta
_{i,n}\right) -T_{n-k,n^{1/2}\theta _{i}/(\sigma \xi _{i,n})}\left(
-n^{1/2}\eta _{i,n}\right) .  \label{select_prob_unknown}
\end{align}%
Here we have used (\ref{select_prob}), and independence of $\hat{\sigma}$
and $\hat{\theta}_{LS,i}$ allowed us to replace $\hat{\sigma}$ by $s\sigma $
in the relevant formulae, cf.~Leeb and P\"{o}tscher (2003, p.~110). In the
above $\rho _{n-k}$ denotes the density of $(n-k)^{-1/2}$ times the square
root of a chi-square distributed random variable with $n-k$ degrees of
freedom. It will turn out to be convenient to set $\rho _{n-k}(s)=0$ for $%
s<0 $, making $\rho _{n-k}$ a bounded continuous function on $\mathbb{R}$.

We now have the following fixed-parameter asymptotic result for the variable
selection/deletion probabilities in the unknown-variance case that perfectly
parallels the corresponding result in the known-variance case, i.e.,
Proposition \ref{select_prob_pointwise}:

\begin{proposition}
\label{select_prob_pointwise_unknown}Let $0<\sigma <\infty $ be given. For
every $i\geq 1$ satisfying $i\leq k=k(n)$ for large enough $n$ we have:

(a) A necessary and sufficient condition for $P_{n,\theta ,\sigma }\left( 
\tilde{\theta}_{i}=0\right) \rightarrow 0$ as $n\rightarrow \infty $ for all 
$\theta $ satisfying $\theta _{i}\neq 0$ ($\theta _{i}$ not depending on $n$%
) is $\xi _{i,n}\eta _{i,n}\rightarrow 0$.

(b) A necessary and sufficient condition for $P_{n,\theta ,\sigma }\left( 
\tilde{\theta}_{i}=0\right) \rightarrow 1$ as $n\rightarrow \infty $ for all 
$\theta $ satisfying $\theta _{i}=0$ is $n^{1/2}\eta _{i,n}\rightarrow
\infty $.

(c) A necessary and sufficient condition for $P_{n,\theta ,\sigma }\left( 
\tilde{\theta}_{i}=0\right) -c_{i,n}\rightarrow 0$ as $n\rightarrow \infty $
for all $\theta $ satisfying $\theta _{i}=0$ and with $c_{i,n}=T_{n-k}\left(
e_{i}\right) -T_{n-k}\left( -e_{i}\right) $ satisfying $\limsup_{n%
\rightarrow \infty }c_{i,n}<1$ is $n^{1/2}\eta _{i,n}\rightarrow e_{i}$, $%
0\leq e_{i}<\infty $.
\end{proposition}

Proposition \ref{select_prob_pointwise_unknown} shows that the dichotomy
regarding conservative tuning and consistent tuning is expressed by the same
conditions in the unknown-variance case as in the known-variance case.
Furthermore, note that $c_{i,n}$ appearing in Part (c) of the above
proposition converges to $c_{i}=\Phi (e_{i})-\Phi (-e_{i})$ in the case
where $n-k\rightarrow \infty $, the limit thus being the same as in the
known-variance case. This is different in case $n-k$ is constant equal to $m$%
, say, eventually, the sequence $c_{i,n}$ then being constant equal to $%
T_{m}\left( e_{i}\right) -T_{m}\left( -e_{i}\right) $ eventually. We finally
note that Remark \ref{uninteresting}\ also applies to Proposition \ref%
{select_prob_pointwise_unknown} above.

For the same reasons as in the known-variance case we next investigate the
asymptotic behavior of the variable selection/deletion probabilities under a
moving-parameter asymptotic framework. We consider the case where $n-k$ is
(eventually) constant and the case where $n-k\rightarrow \infty $. There is
no essential loss in generality in considering these two cases only, since
by compactness of $\mathbb{N\cup \{\infty \}}$ we can always assume
(possibly after passing to subsequences) that $n-k$ converges in $\mathbb{%
N\cup \{\infty \}}$.

\begin{theorem}
\label{select_prob_moving_par_unknown}Suppose that for given $i\geq 1$
satisfying $i\leq k=k(n)$ for large enough $n$ we have $\xi _{i,n}\eta
_{i,n}\rightarrow 0$ and $n^{1/2}\eta _{i,n}\rightarrow e_{i}$ where $0\leq
e_{i}\leq \infty $.

(a) Assume $e_{i}<\infty $. Suppose that the true parameters $\theta
^{(n)}=(\theta _{1,n},\dots ,\theta _{k_{n},n})\in \mathbb{R}^{k_{n}}$ and $%
\sigma _{n}\in (0,\infty )$ satisfy $n^{1/2}\theta _{i,n}/(\sigma _{n}\xi
_{i,n})\rightarrow \nu _{i}\in \mathbb{\bar{R}}$.

\qquad (a1) If $n-k$ is eventually constant equal to $m$, say, then%
\begin{equation*}
\lim_{n\rightarrow \infty }P_{n,\theta ^{(n)},\sigma _{n}}\left( \tilde{%
\theta}_{i}=0\right) =\int_{0}^{\infty }\left( \Phi \left( -\nu
_{i}+se_{i}\right) -\Phi \left( -\nu _{i}-se_{i}\right) \right) \rho
_{m}(s)ds.
\end{equation*}%
\qquad \qquad \qquad \qquad \qquad

\qquad (a2) If $n-k\rightarrow \infty $ holds, then%
\begin{equation*}
\lim_{n\rightarrow \infty }P_{n,\theta ^{(n)},\sigma _{n}}\left( \tilde{%
\theta}_{i}=0\right) =\Phi \left( -\nu _{i}+e_{i}\right) -\Phi \left( -\nu
_{i}-e_{i}\right) .
\end{equation*}

(b) Assume $e_{i}=\infty $. Suppose that the true parameters $\theta
^{(n)}=(\theta _{1,n},\ldots ,\theta _{k_{n},n})\in \mathbb{R}^{k_{n}}$ and $%
\sigma _{n}\in (0,\infty )$ satisfy $\theta _{i,n}/(\sigma _{n}\xi
_{i,n}\eta _{i,n})\rightarrow \zeta _{i}\in \mathbb{\bar{R}}$.

\qquad (b1) If $n-k$ is eventually constant equal to $m$, say, then 
\begin{equation*}
\lim_{n\rightarrow \infty }P_{n,\theta ^{(n)},\sigma _{n}}\left( \tilde{%
\theta}_{i}=0\right) =\int_{\left\vert \zeta _{i}\right\vert }^{\infty }\rho
_{m}(s)ds=\Pr (\chi _{m}^{2}>m\zeta _{i}^{2}).
\end{equation*}%
\qquad

\qquad (b2) If $n-k\rightarrow \infty $ holds, then

\qquad \qquad 1. $\left\vert \zeta _{i}\right\vert <1$ implies $%
\lim_{n\rightarrow \infty }P_{n,\theta ^{(n)},\sigma _{n}}\left( \tilde{%
\theta}_{i}=0\right) =1$.

\qquad \qquad 2. $\left\vert \zeta _{i}\right\vert >1$ implies $%
\lim_{n\rightarrow \infty }P_{n,\theta ^{(n)},\sigma _{n}}\left( \tilde{%
\theta}_{i}=0\right) =0$.

\qquad \qquad 3. $\left\vert \zeta _{i}\right\vert =1$ and $n^{1/2}\eta
_{i,n}/\left( n-k\right) ^{1/2}\rightarrow 0$ imply 
\begin{equation*}
\lim_{n\rightarrow \infty }P_{n,\theta ^{(n)},\sigma _{n}}\left( \tilde{%
\theta}_{i}=0\right) =\Phi (r_{i})
\end{equation*}%
provided $r_{i,n}:=n^{1/2}\left( \eta _{i,n}-\zeta _{i}\theta _{i,n}/(\sigma
_{n}\xi _{i,n})\right) \rightarrow r_{i}$ for some $r_{i}\in \mathbb{\bar{R}}
$.

\qquad \qquad 4. $\left\vert \zeta _{i}\right\vert =1$ and $n^{1/2}\eta
_{i,n}/\left( n-k\right) ^{1/2}\rightarrow 2^{1/2}d_{i}$ with $%
0<d_{i}<\infty $ imply 
\begin{equation*}
\lim_{n\rightarrow \infty }P_{n,\theta ^{(n)},\sigma _{n}}\left( \tilde{%
\theta}_{i}=0\right) =\int_{-\infty }^{\infty }\Phi (d_{i}t+r_{i})\phi (t)dt
\end{equation*}%
provided $r_{i,n}\rightarrow r_{i}$ for some $r_{i}\in \mathbb{\bar{R}}$.
[Note that the integral in the above display reduces to $1$ if $r_{i}=\infty 
$, and to $0$ if $r_{i}=-\infty $.]

\qquad \qquad 5. $\left\vert \zeta _{i}\right\vert =1$ and $n^{1/2}\eta
_{i,n}/\left( n-k\right) ^{1/2}\rightarrow \infty $ imply 
\begin{equation*}
\lim_{n\rightarrow \infty }P_{n,\theta ^{(n)},\sigma _{n}}\left( \tilde{%
\theta}_{i}=0\right) =\Phi (r_{i}^{\prime })
\end{equation*}%
provided $\left( n^{1/2}\eta _{i,n}/\left( n-k\right) ^{1/2}\right)
^{-1}r_{i,n}\rightarrow 2^{-1/2}r_{i}^{\prime }$ for some $r_{i}^{\prime
}\in \mathbb{\bar{R}}$.
\end{theorem}

Theorem \ref{select_prob_moving_par_unknown} shows, in particular, that also
in the unknown-variance case the convergence of the variable
selection/deletion probabilities to their limits in a fixed-parameter
asymptotic framework is not locally uniform in $\theta _{i}$. In the case of
conservative tuning the theorem furthermore shows that the limit of the
variable selection/deletion probabilities in the unknown-variance case is
the same as in the known-variance case if the degrees of freedom $n-k$ go to
infinity (entailing that the distribution of $\hat{\sigma}/\sigma $
concentrates more and more around $1$); if $n-k$ is eventually constant, the
limit turns out to be a mixture of the known-variance case limits (with $%
\sigma $ replaced by $s\sigma $), the mixture being with respect to the
distribution of $\hat{\sigma}/\sigma $. [We note that in the somewhat
uninteresting case $e_{i}=0$ this mixture also reduces to the same limit as
in the known-variance case.] While this result is as one would expect, the
situation is different and more subtle in the case of consistent tuning: If $%
n-k\rightarrow \infty $ the limits are the same as in the known-variance
case if $\left\vert \zeta _{i}\right\vert <1$ or $\left\vert \zeta
_{i}\right\vert >1$ holds, namely $1$ and $0$, respectively. However, in the
"boundary" case $\left\vert \zeta _{i}\right\vert =1$ the rate at which $n-k$
diverges to infinity becomes relevant. If the divergence is fast enough in
the sense that $n^{1/2}\eta _{i,n}/\left( n-k\right) ^{1/2}\rightarrow 0$,
again the same limit as in the known-variance case, namely $\Phi (r_{i})$,
is obtained; but if $n-k$ diverges to infinity more slowly, a different
limit arises (which, e.g., in case 4 of Part (b2) is obtained by averaging $%
\Phi (r_{i}+\cdot )$ with respect to a suitable distribution). The case
where the degrees of freedom $n-k$ is eventually constant looks very much
different from the known-variance case and again some averaging with respect
to the distribution of $\hat{\sigma}/\sigma $ takes place. Note that in this
case the limiting variable deletion probabilities are $1$ and $0$,
respectively, only if $\zeta _{i}=0$ and $\left\vert \zeta _{i}\right\vert
=\infty $, respectively, which is in contrast to the known-variance case
(and the unknown-variance case with $n-k\rightarrow \infty $).

\begin{remark}
\label{costfree}\normalfont(i) For later use we note that Proposition \ref%
{select_prob_moving_par} and Theorem \ref{select_prob_moving_par_unknown}
also hold when applied to subsequences, as is easily seen.

(ii) The convergence conditions in Proposition \ref{select_prob_moving_par}
on the various quantities involving $\theta _{i,n}$ and $\sigma _{n}$ are
essentially cost-free in the sense that given any sequence $(\theta
_{i,n},\sigma _{n})$ we can, due to compactness of $\mathbb{\bar{R}}$,
select from any subsequence $n_{j}$ a further subsubsequence $n_{j(l)}$ such
that along this subsubsequence all relevant quantities such as $%
n^{1/2}\theta _{i,n}/(\sigma _{n}\xi _{i,n})$ (or $\theta _{i,n}/(\sigma
_{n}\xi _{i,n}\eta _{i,n})$ and $r_{i,n}$) converge in $\mathbb{\bar{R}}$.
Since Proposition \ref{select_prob_moving_par} also holds when applied to
subsequences as just noted, an application of this proposition to the
subsubsequence $n_{j(l)}$ then results in a characterization of all possible
accumulation points of the variable selection/deletion probabilities in the
known-variance case.

(iii) In a similar manner, the convergence conditions in Theorem \ref%
{select_prob_moving_par_unknown} (including the ones on $n-k$) are
essentially cost-free, and thus this theorem provides a full
characterization of all possible accumulation points of the variable
selection/deletion probabilities in the unknown-variance case.
\end{remark}

As just discussed, in the case of conservative tuning we get the same
limiting behavior under moving-parameter asymptotics in the known-variance
and in the unknown-variance case along any sequence of parameters if $%
n-k\rightarrow \infty $ or $e_{i}=0$ (which in the conservatively tuned case
can equivalently be stated as $n^{1/2}\eta _{i,n}/\left( n-k\right)
^{1/2}\rightarrow 0$). In the case of consistent tuning the same coincidence
of limits occurs if $n-k\rightarrow \infty $ fast enough such that $%
n^{1/2}\eta _{i,n}/\left( n-k\right) ^{1/2}\rightarrow 0$. This is not
accidental but a consequence of the following fact:

\begin{proposition}
\label{closeness_prob}Suppose that for given $i\geq 1$ satisfying $i\leq
k=k(n)$ for large enough $n$ we have $n^{1/2}\eta
_{i,n}(n-k)^{-1/2}\rightarrow 0$ as $n\rightarrow \infty $. Then 
\begin{equation*}
\sup_{\theta \in \mathbb{R}^{k},0<\sigma <\infty }\left\vert P_{n,\theta
,\sigma }\left( \hat{\theta}_{i}=0\right) -P_{n,\theta ,\sigma }\left( 
\tilde{\theta}_{i}=0\right) \right\vert \rightarrow 0\qquad \text{for }%
n\rightarrow \infty .
\end{equation*}
\end{proposition}

\begin{remark}
\normalfont\label{weekend}Suppose that $\xi _{i,n}\eta _{i,n}\rightarrow 0$
holds as $n\rightarrow \infty $, the other case being of little interest as
noted earlier. If $n^{1/2}\eta _{i,n}(n-k)^{-1/2}$ does not converge to zero
as $n\rightarrow \infty $, it can be shown from Proposition \ref%
{select_prob_moving_par} and Theorem \ref{select_prob_moving_par_unknown}
that the limits of the variable deletion probabilities (along appropriate
(sub)sequences $(\theta ^{(n_{j})},\sigma _{n_{j}})$) for the known-variance
and the unknown-variance case do not coincide. This shows that the condition 
$n^{1/2}\eta _{i,n}(n-k)^{-1/2}\rightarrow 0$ in the above proposition
cannot be weakened (at least in case $\xi _{i,n}\eta _{i,n}\rightarrow 0$
holds).
\end{remark}

\section{Consistency, Uniform Consistency, and Uniform Convergence Rate\label%
{minimax}}

For purposes of comparison we start with the following obvious proposition,
which immediately follows from the observation that $\hat{\theta}_{LS,i}$ is 
$N(\theta _{i},\sigma ^{2}\xi _{i,n}^{2}/n)$-distributed.

\begin{proposition}
\label{ls_consistency}For every $i\geq 1$ satisfying $i\leq k=k(n)$ for
large enough $n$ we have the following:

(a) $\xi _{i,n}/n^{1/2}\rightarrow 0$ is a necessary and sufficient
condition for $\hat{\theta}_{LS,i}$ to be consistent for $\theta _{i}$, the
convergence rate being $\xi _{i,n}/n^{1/2}$.

(b) Suppose $\xi _{i,n}/n^{1/2}\rightarrow 0$. Then $\hat{\theta}_{LS,i}$ is
uniformly consistent for $\theta _{i}$ in the sense that for every $%
\varepsilon >0$%
\begin{equation*}
\lim_{n\rightarrow \infty }\sup_{\theta \in \mathbb{R}^{k}}\sup_{0<\sigma
<\infty }P_{n,\theta ,\sigma }\left( \left\vert \hat{\theta}_{LS,i}-\theta
_{i}\right\vert >\sigma \varepsilon \right) =0.
\end{equation*}

In fact, $\hat{\theta}_{LS,i}$ is uniformly $n^{1/2}/\xi _{i,n}$-consistent
for $\theta _{i}$ in the sense that for every $\varepsilon >0$ there exists
a real number $M>0$ such that 
\begin{equation*}
\sup_{n\in \mathbb{N}}\sup_{\theta \in \mathbb{R}^{k}}\sup_{0<\sigma <\infty
}P_{n,\theta ,\sigma }\left( \left( n^{1/2}/\xi _{i,n}\right) \left\vert 
\hat{\theta}_{LS,i}-\theta _{i}\right\vert >\sigma M\right) <\varepsilon .
\end{equation*}%
[Note that the probabilities in the displays above in fact neither depend on 
$\theta $ nor $\sigma $. In particular, the l.h.s.~of the above displays
equal $2\Phi (-\varepsilon n^{1/2}/\xi _{i,n})$ and $2\Phi (-M)$,
respectively.]
\end{proposition}

The corresponding result for the estimators $\tilde{\theta}_{H,i}$, $\tilde{%
\theta}_{S,i}$, or $\tilde{\theta}_{AS,i}$ and their infeasible counterparts 
$\hat{\theta}_{H,i}$, $\hat{\theta}_{S,i}$, or $\hat{\theta}_{AS,i}$ is now
as follows.

\begin{theorem}
\label{thresh_consistency}Let $\tilde{\theta}_{i}$ stand for any of the
estimators $\tilde{\theta}_{H,i}$, $\tilde{\theta}_{S,i}$, or $\tilde{\theta}%
_{AS,i}$. Then for every $i\geq 1$ satisfying $i\leq k=k(n)$ for large
enough $n$ we have the following:

(a) $\tilde{\theta}_{i}$ is consistent for $\theta _{i}$ if and only if $\xi
_{i,n}\eta _{i,n}\rightarrow 0$ and $\xi _{i,n}/n^{1/2}\rightarrow 0$.

(b) Suppose $\xi _{i,n}\eta _{i,n}\rightarrow 0$ and $\xi
_{i,n}/n^{1/2}\rightarrow 0$. Then $\tilde{\theta}_{i}$ is uniformly
consistent in the sense that for every $\varepsilon >0$%
\begin{equation*}
\lim_{n\rightarrow \infty }\sup_{\theta \in \mathbb{R}^{k}}\sup_{0<\sigma
<\infty }P_{n,\theta ,\sigma }\left( \left\vert \tilde{\theta}_{i}-\theta
_{i}\right\vert >\sigma \varepsilon \right) =0.
\end{equation*}%
Furthermore, $\tilde{\theta}_{i}$ is uniformly $a_{i,n}$-consistent with $%
a_{i,n}=\min \left( n^{1/2}/\xi _{i,n},(\xi _{i,n}\eta _{i,n})^{-1}\right) $
in the sense that for every $\varepsilon >0$ there exists a real number $M>0$
such that 
\begin{equation*}
\sup_{n\in \mathbb{N}}\sup_{\theta \in \mathbb{R}^{k}}\sup_{0<\sigma <\infty
}P_{n,\theta ,\sigma }\left( a_{i,n}\left\vert \tilde{\theta}_{i}-\theta
_{i}\right\vert >\sigma M\right) <\varepsilon .
\end{equation*}

(c) Suppose $\xi _{i,n}\eta _{i,n}\rightarrow 0$ and $\xi
_{i,n}/n^{1/2}\rightarrow 0$ and $b_{i,n}\geq 0$. If for every $\varepsilon
>0$ there exists a real number $M>0$ such that 
\begin{equation}
\limsup_{n\rightarrow \infty }\sup_{\theta \in \mathbb{R}^{k}}\sup_{0<\sigma
<\infty }P_{n,\theta ,\sigma }\left( b_{i,n}\left\vert \tilde{\theta}%
_{i}-\theta _{i}\right\vert >\sigma M\right) <\varepsilon  \label{nec}
\end{equation}%
holds, then $b_{i,n}=O(a_{i,n})$ necessarily holds.

(d) Let $\hat{\theta}_{i}$ stand for any of the estimators $\hat{\theta}%
_{H,i}$, $\hat{\theta}_{S,i}$, or $\hat{\theta}_{AS,i}$. Then the results in
(a)-(c) also hold for $\hat{\theta}_{i}$.
\end{theorem}

The preceding theorem shows that the thresholding estimators $\tilde{\theta}%
_{H,i}$, $\tilde{\theta}_{S,i}$, and $\tilde{\theta}_{AS,i}$ (as well as
their infeasible versions) are uniformly $a_{i,n}$-consistent and that this
rate is sharp and cannot be improved. In particular, if the tuning is
conservative these estimators are uniformly $n^{1/2}/\xi _{i,n}$-consistent,
which is the usual rate one expects to find in a linear regression model as
considered here. However, if consistent tuning is employed, the preceding
theorem shows that these thresholding estimators are then only uniformly $%
(\xi _{i,n}\eta _{i,n})^{-1}$-consistent, i.e., have a slower uniform
convergence rate than the least-squares (maximum likelihood) estimator (or
the conservatively tuned thresholding estimators for that matter). For a
discussion of the pointwise convergence rate see Section \ref{oracle}.

\begin{remark}
\normalfont\label{asy-equiv}If $n^{1/2}\eta _{i,n}\rightarrow e_{i}=0$, then 
$\tilde{\theta}_{i}$ is asymptotically equivalent to $\hat{\theta}_{LS,i}$
in the sense that for every $\varepsilon >0$%
\begin{equation*}
\lim_{n\rightarrow \infty }\sup_{\theta \in \mathbb{R}^{k}}\sup_{0<\sigma
<\infty }P_{n,\theta ,\sigma }\left( \left( n^{1/2}/\xi _{i,n}\right) |%
\tilde{\theta}_{i}-\hat{\theta}_{LS,i}|>\sigma \varepsilon \right) =0.
\end{equation*}%
A similar statement holds for $\hat{\theta}_{i}$. For $\tilde{\theta}_{i}$
this follows immediately from (\ref{closeness_H_S_AS_LS_unknown}) in Section %
\ref{prfs} and the fact that the family of distributions corresponding to $%
\rho _{n-k}$ is tight; for $\hat{\theta}_{i}$ this follows from the relation 
$\left\vert \hat{\theta}_{i}-\hat{\theta}_{LS,i}\right\vert \leq \sigma \xi
_{i,n}\eta _{i,n}$.
\end{remark}

\begin{remark}
\normalfont(i) A variation of the proof of Theorem \ref{thresh_consistency}
shows that in case of consistent tuning for the infeasible estimators
additionally also%
\begin{equation*}
\lim_{n\rightarrow \infty }\sup_{\theta \in \mathbb{R}^{k}}\sup_{0<\sigma
<\infty }P_{n,\theta ,\sigma }\left( a_{i,n}\left\vert \hat{\theta}%
_{i}-\theta _{i}\right\vert >\sigma M\right) =0
\end{equation*}%
holds for every $M>1$, and that for the feasible estimators%
\begin{equation*}
\lim_{n\rightarrow \infty }\sup_{\theta \in \mathbb{R}^{k}}\sup_{0<\sigma
<\infty }P_{n,\theta ,\sigma }\left( a_{i,n}\left\vert \tilde{\theta}%
_{i}-\theta _{i}\right\vert >\sigma M\right) =0
\end{equation*}%
holds for every $M>1$ provided that $n-k\rightarrow \infty $.

(ii) Inspection of the proof shows that the conclusion of Theorem \ref%
{thresh_consistency}(c) continues to hold if the supremum over $\mathbb{R}%
^{k}$ is replaced by the supremum over an arbitrarily small neighborhood of $%
0$ and $\sigma $ is held fixed at an arbitrary positive value.

(iii) If $\sigma \varepsilon $ and $\sigma M$ are replaced by $\varepsilon $
and $M$, respectively, in the displays in Proposition \ref{ls_consistency}
and Theorem \ref{thresh_consistency} as well as in Remark \ref{asy-equiv},
the resulting statements remain true provided the suprema over $0<\sigma
<\infty $ are replaced by suprema over $0<\sigma \leq c$, where $c>0$ is an
arbitrary real number.
\end{remark}

\section{Finite-Sample Distributions\label{FS}}

\subsection{ Known-Variance Case}

We next present the finite-sample distributions of the infeasible
thresholding estimators. It will turn out to be convenient to give the
results for scaled versions, where the scaling factor $\alpha _{i,n}$ is a
positive real number, but is otherwise arbitrary. \emph{Note that below we
suppress the dependence of the distribution functions of the thresholding
estimators on the scaling sequence }$\alpha _{i,n}$\emph{\ in the notation.}
Furthermore, observe that the finite-sample distributions depend on $\theta $
only through $\theta _{i}$.

\begin{proposition}
\label{1}The cdf $H_{H,n,\theta ,\sigma }^{i}:=H_{H,\eta _{i,n},n,\theta
,\sigma }^{i}$ of $\sigma ^{-1}\alpha _{i,n}(\hat{\theta}_{H,i}-\theta _{i})$
is given by 
\begin{eqnarray}
H_{H,n,\theta ,\sigma }^{i}(x) &=&\Phi \left( n^{1/2}x/(\alpha _{i,n}\xi
_{i,n})\right) \boldsymbol{1}\left( \left\vert \alpha _{i,n}^{-1}x+\theta
_{i}/\sigma \right\vert >\xi _{i,n}\eta _{i,n}\right)  \notag \\
&&+\Phi \left( n^{1/2}\left( -\theta _{i}/(\sigma \xi _{i,n})+\eta
_{i,n}\right) \right) \boldsymbol{1}\left( 0\leq \alpha _{i,n}^{-1}x+\theta
_{i}/\sigma \leq \xi _{i,n}\eta _{i,n}\right)  \notag \\
&&+\Phi \left( n^{1/2}\left( -\theta _{i}/(\sigma \xi _{i,n})-\eta
_{i,n}\right) \right) \boldsymbol{1}\left( -\xi _{i,n}\eta _{i,n}\leq \alpha
_{i,n}^{-1}x+\theta _{i}/\sigma <0\right) ,  \label{hard_finite_sample}
\end{eqnarray}%
or, equivalently, 
\begin{eqnarray}
dH_{H,n,\theta ,\sigma }^{i}(x) &=&\left\{ \Phi \left( n^{1/2}\left( -\theta
_{i}/(\sigma \xi _{i,n})+\eta _{i,n}\right) \right) -\Phi \left(
n^{1/2}\left( -\theta _{i}/(\sigma \xi _{i,n})-\eta _{i,n}\right) \right)
\right\} d\delta _{-\alpha _{i,n}\theta _{i}/\sigma }(x)  \notag \\
&&+\left( n^{1/2}/(\alpha _{i,n}\xi _{i,n})\right) \phi \left(
n^{1/2}x/(\alpha _{i,n}\xi _{i,n})\right) \boldsymbol{1}\left( \left\vert
\alpha _{i,n}^{-1}x+\theta _{i}/\sigma \right\vert >\xi _{i,n}\eta
_{i,n}\right) dx  \label{hard_finite_sample_density}
\end{eqnarray}%
where $\delta _{z}$ denotes pointmass at $z$.
\end{proposition}

\begin{proposition}
\label{2}The cdf $H_{S,n,\theta ,\sigma }^{i}:=H_{S,\eta _{i,n},n,\theta
,\sigma }^{i}$ of $\sigma ^{-1}\alpha _{i,n}(\hat{\theta}_{S,i}-\theta _{i})$
is given by%
\begin{eqnarray}
H_{S,n,\theta ,\sigma }^{i}(x) &=&\Phi \left( n^{1/2}x/(\alpha _{i,n}\xi
_{i,n})+n^{1/2}\eta _{i,n}\right) \boldsymbol{1}\left( \alpha
_{i,n}^{-1}x+\theta _{i}/\sigma \geq 0\right)  \notag \\
&&+\Phi \left( n^{1/2}x/(\alpha _{i,n}\xi _{i,n})-n^{1/2}\eta _{i,n}\right) 
\boldsymbol{1}\left( \alpha _{i,n}^{-1}x+\theta _{i}/\sigma <0\right) ,
\label{soft_finite_sample}
\end{eqnarray}%
or, equivalently, 
\begin{eqnarray}
dH_{S,n,\theta ,\sigma }^{i}(x) &=&\left\{ \Phi \left( n^{1/2}\left( -\theta
_{i}/(\sigma \xi _{i,n})+\eta _{i,n}\right) \right) -\Phi \left(
n^{1/2}\left( -\theta _{i}/(\sigma \xi _{i,n})-\eta _{i,n}\right) \right)
\right\} d\delta _{-\alpha _{i,n}\theta _{i}/\sigma }(x)  \notag \\
&&+\left( n^{1/2}/(\alpha _{i,n}\xi _{i,n})\right) \left\{ \phi \left(
n^{1/2}x/(\alpha _{i,n}\xi _{i,n})+n^{1/2}\eta _{i,n}\right) \boldsymbol{1}%
\left( \alpha _{i,n}^{-1}x+\theta _{i}/\sigma >0\right) \right.
\label{soft_finite_sample_density} \\
&&\left. +\phi \left( n^{1/2}x/(\alpha _{i,n}\xi _{i,n})-n^{1/2}\eta
_{i,n}\right) \boldsymbol{1}\left( \alpha _{i,n}^{-1}x+\theta _{i}/\sigma
<0\right) \right\} dx.  \notag
\end{eqnarray}
\end{proposition}

\begin{proposition}
\label{3}The cdf $H_{AS,n,\theta ,\sigma }^{i}:=H_{AS,\eta _{i,n},n,\theta
,\sigma }^{i}$ of $\sigma ^{-1}\alpha _{i,n}(\hat{\theta}_{AS,i}-\theta
_{i}) $ is given by 
\begin{equation}
H_{AS,n,\theta ,\sigma }^{i}(x)=\Phi \left( z_{n,\theta ,\sigma
}^{(2)}(x,\eta _{i,n})\right) \boldsymbol{1}\left( \alpha
_{i,n}^{-1}x+\theta _{i}/\sigma \geq 0\right) +\Phi \left( z_{n,\theta
,\sigma }^{(1)}(x,\eta _{i,n})\right) \boldsymbol{1}\left( \alpha
_{i,n}^{-1}x+\theta _{i}/\sigma <0\right) ,  \label{adaptive_finite_sample}
\end{equation}%
where $z_{n,\theta ,\sigma }^{(1)}(x,y)\leq z_{n,\theta ,\sigma }^{(2)}(x,y)$
are defined by 
\begin{equation*}
0.5n^{1/2}\xi _{i,n}^{-1}(\alpha _{i,n}^{-1}x-\theta _{i}/\sigma )\pm n^{1/2}%
\sqrt{\left( 0.5\xi _{i,n}^{-1}(\alpha _{i,n}^{-1}x+\theta _{i}/\sigma
)\right) ^{2}+y^{2}}.
\end{equation*}%
Or, equivalently, 
\begin{eqnarray*}
dH_{AS,n,\theta ,\sigma }^{i}(x) &=&\left\{ \Phi \left( n^{1/2}\left(
-\theta _{i}/(\sigma \xi _{i,n})+\eta _{i,n}\right) \right) -\Phi \left(
n^{1/2}\left( -\theta _{i}/(\sigma \xi _{i,n})-\eta _{i,n}\right) \right)
\right\} d\delta _{-\alpha _{i,n}\theta _{i}/\sigma }(x) \\
&&+(0.5n^{1/2}/(\alpha _{i,n}\xi _{i,n}))\left\{ \phi \left( z_{n,\theta
,\sigma }^{(2)}(x,\eta _{i,n})\right) (1+t_{n,\theta ,\sigma }(x,\eta
_{i,n}))\boldsymbol{1}\left( \alpha _{i,n}^{-1}x+\theta _{i}/\sigma
>0\right) \right. \\
&&+\left. \phi \left( z_{n,\theta ,\sigma }^{(1)}(x,\eta _{i,n})\right)
(1-t_{n,\theta ,\sigma }(x,\eta _{i,n}))\boldsymbol{1}\left( \alpha
_{i,n}^{-1}x+\theta _{i}/\sigma <0\right) \right\} ,
\end{eqnarray*}%
where $t_{n,\theta ,\sigma }(x,y)=0.5\xi _{i,n}^{-1}\left( \alpha
_{i,n}^{-1}x+\theta _{i}/\sigma \right) /\left( (0.5\xi _{i,n}^{-1}\left(
\alpha _{i,n}^{-1}x+\theta _{i}/\sigma \right) )^{2}+y^{2}\right) ^{1/2}.$
\end{proposition}

The finite-sample distributions of $\hat{\theta}_{H,i}$, $\hat{\theta}_{S,i}$%
, and $\hat{\theta}_{AS,i}$ are seen to be non-normal. They are made up of
two components, one being a multiple of pointmass at $-\alpha _{i,n}\theta
_{i}/\sigma $ and the other one being absolutely continuous with a density
that is generally bimodal. For more discussion and some graphical
illustrations in a special case see P\"{o}tscher and Leeb (2009) and P\"{o}%
tscher and Schneider (2009).

\begin{remark}
\label{diag}\normalfont In the case where $X^{\prime }X$ is diagonal, the
estimators of the components $\theta _{i}$ and $\theta _{j}$ for $i\neq j$
are independent and hence the above results immediately allow one to
determine the finite-sample distributions of the entire vectors $\hat{\theta}%
_{H}$, $\hat{\theta}_{S}$, and $\hat{\theta}_{AS}$. In particular, this
provides the finite-sample distribution of the Lasso $\hat{\theta}_{L}$ and
the adaptive Lasso $\hat{\theta}_{AS}$ in the diagonal case (cf.~Remarks \ref%
{LASSO} and \ref{ALASSO}).
\end{remark}

\subsection{Unknown-Variance Case}

The finite-sample distributions of $\tilde{\theta}_{H,i}$, $\tilde{\theta}%
_{S,i}$, $\tilde{\theta}_{AS,i}$ are obtained next. The same remark on the
scaling as in the previous section applies here.

\begin{proposition}
\label{4}The cdf $H_{H,n,\theta ,\sigma }^{i\maltese }:=H_{H,\eta
_{i,n},n,\theta ,\sigma }^{i\maltese }$ of $\sigma ^{-1}\alpha _{i,n}(\tilde{%
\theta}_{H,i}-\theta _{i})$ is given by 
\begin{eqnarray}
H_{H,n,\theta ,\sigma }^{i\maltese }(x) &=&\Phi \left( n^{1/2}x/(\alpha
_{i,n}\xi _{i,n})\right) \int_{0}^{\infty }\boldsymbol{1}\left( \left\vert
\alpha _{i,n}^{-1}x+\theta _{i}/\sigma \right\vert >\xi _{i,n}s\eta
_{i,n}\right) \rho _{n-k}(s)ds  \label{hard_finite_sample_unknown} \\
&&+\int_{0}^{\infty }\Phi \left( n^{1/2}\left( -\theta _{i}/(\sigma \xi
_{i,n})+s\eta _{i,n}\right) \right) \boldsymbol{1}\left( 0\leq \alpha
_{i,n}^{-1}x+\theta _{i}/\sigma \leq \xi _{i,n}s\eta _{i,n}\right) \rho
_{n-k}(s)ds  \notag \\
&&+\int_{0}^{\infty }\Phi \left( n^{1/2}\left( -\theta _{i}/(\sigma \xi
_{i,n})-s\eta _{i,n}\right) \right) \boldsymbol{1}\left( -\xi _{i,n}s\eta
_{i,n}\leq \alpha _{i,n}^{-1}x+\theta _{i}/\sigma <0\right) \rho _{n-k}(s)ds.
\notag
\end{eqnarray}%
Or, equivalently, 
\begin{eqnarray}
dH_{H,n,\theta ,\sigma }^{i\maltese }(x) &=&\int_{0}^{\infty }\left\{ \Phi
\left( n^{1/2}\left( -\theta _{i}/(\sigma \xi _{i,n})+s\eta _{i,n}\right)
\right) \right.  \label{hard_finite_sample_unknown_density} \\
&&\left. -\Phi \left( n^{1/2}\left( -\theta _{i}/(\sigma \xi _{i,n})-s\eta
_{i,n}\right) \right) \right\} \rho _{n-k}(s)dsd\delta _{-\alpha
_{i,n}\theta _{i}/\sigma }(x)+n^{1/2}\alpha _{i,n}^{-1}\xi _{i,n}^{-1} 
\notag \\
&&\times \phi \left( n^{1/2}x/(\alpha _{i,n}\xi _{i,n})\right)
\int_{0}^{\infty }\boldsymbol{1}\left( \left\vert \alpha _{i,n}^{-1}x+\theta
_{i}/\sigma \right\vert >\xi _{i,n}s\eta _{i,n}\right) \rho _{n-k}(s)dsdx. 
\notag
\end{eqnarray}
\end{proposition}

\begin{proposition}
\label{5}The cdf $H_{S,n,\theta ,\sigma }^{i\maltese }:=H_{S,\eta
_{i,n},n,\theta ,\sigma }^{i\maltese }$ of $\sigma ^{-1}\alpha _{i,n}(\tilde{%
\theta}_{S,i}-\theta _{i})$ is given by%
\begin{eqnarray}
H_{S,n,\theta ,\sigma }^{i\maltese }(x) &=&\int_{0}^{\infty }\Phi \left(
n^{1/2}x/(\alpha _{i,n}\xi _{i,n})+n^{1/2}s\eta _{i,n}\right) \rho
_{n-k}(s)ds\boldsymbol{1}\left( \alpha _{i,n}^{-1}x+\theta _{i}/\sigma \geq
0\right)  \notag \\
&&+\int_{0}^{\infty }\Phi \left( n^{1/2}x/(\alpha _{i,n}\xi
_{i,n})-n^{1/2}s\eta _{i,n}\right) \rho _{n-k}(s)ds\boldsymbol{1}\left(
\alpha _{i,n}^{-1}x+\theta _{i}/\sigma <0\right)  \notag \\
&=&T_{n-k,-n^{1/2}x/(\alpha _{i,n}\xi _{i,n})}\left( n^{1/2}\eta
_{i,n}\right) \boldsymbol{1}\left( \alpha _{i,n}^{-1}x+\theta _{i}/\sigma
\geq 0\right)  \notag \\
&&+T_{n-k,-n^{1/2}x/(\alpha _{i,n}\xi _{i,n})}\left( -n^{1/2}\eta
_{i,n}\right) \boldsymbol{1}\left( \alpha _{i,n}^{-1}x+\theta _{i}/\sigma
<0\right) .  \label{soft_finite_sample_unknown}
\end{eqnarray}%
Or, equivalently,%
\begin{eqnarray}
dH_{S,n,\theta ,\sigma }^{i\maltese }(x) &=&\int_{0}^{\infty }\left\{ \Phi
\left( n^{1/2}\left( -\theta _{i}/(\sigma \xi _{i,n})+s\eta _{i,n}\right)
\right) \right.  \label{soft_finite_sample_unknown_density} \\
&&\left. -\Phi \left( n^{1/2}\left( -\theta _{i}/(\sigma \xi _{i,n})-s\eta
_{i,n}\right) \right) \right\} \rho _{n-k}(s)dsd\delta _{-\alpha
_{i,n}\theta _{i}/\sigma }(x)+n^{1/2}\alpha _{i,n}^{-1}\xi _{i,n}^{-1} 
\notag \\
&&\times \left\{ \int_{0}^{\infty }\phi \left( n^{1/2}x/(\alpha _{i,n}\xi
_{i,n})+n^{1/2}s\eta _{i,n}\right) \rho _{n-k}(s)ds\boldsymbol{1}\left(
\alpha _{i,n}^{-1}x+\theta _{i}/\sigma >0\right) \right.  \notag \\
&&\left. +\int_{0}^{\infty }\phi \left( n^{1/2}x/(\alpha _{i,n}\xi
_{i,n})-n^{1/2}s\eta _{i,n}\right) \rho _{n-k}(s)ds\boldsymbol{1}\left(
\alpha _{i,n}^{-1}x+\theta _{i}/\sigma <0\right) \right\} dx.  \notag
\end{eqnarray}
\end{proposition}

\begin{proposition}
\label{6}The cdf $H_{AS,n,\theta ,\sigma }^{i\maltese }:=H_{AS,\eta
_{i,n},n,\theta ,\sigma }^{i\maltese }$ of $\sigma ^{-1}\alpha _{i,n}(\tilde{%
\theta}_{AS,i}-\theta _{i})$ is given by 
\begin{eqnarray}
H_{AS,n,\theta ,\sigma }^{i\maltese }(x) &=&\int_{0}^{\infty }\Phi \left(
z_{n,\theta ,\sigma }^{(2)}(x,s\eta _{i,n})\right) \rho _{n-k}(s)ds%
\boldsymbol{1}\left( \alpha _{i,n}^{-1}x+\theta _{i}/\sigma \geq 0\right) 
\notag \\
&&+\int_{0}^{\infty }\Phi \left( z_{n,\theta ,\sigma }^{(1)}(x,s\eta
_{i,n})\right) \rho _{n-k}(s)ds\boldsymbol{1}\left( \alpha
_{i,n}^{-1}x+\theta _{i}/\sigma <0\right) .
\label{adaptive_finite_sample_unknown}
\end{eqnarray}%
Or, equivalently, 
\begin{eqnarray}
dH_{AS,n,\theta ,\sigma }^{i\maltese }(x) &=&\int_{0}^{\infty }\left\{ \Phi
\left( n^{1/2}\left( -\theta _{i}/(\sigma \xi _{i,n})+s\eta _{i,n}\right)
\right) \right.  \label{adaptive_finite_sample_unknown_density} \\
&&\left. -\Phi \left( n^{1/2}\left( -\theta _{i}/(\sigma \xi _{i,n})-s\eta
_{i,n}\right) \right) \right\} \rho _{n-k}(s)dsd\delta _{-\alpha
_{i,n}\theta _{i}/\sigma }(x)+(0.5n^{1/2}/(\alpha _{i,n}\xi _{i,n}))  \notag
\\
&&\times \left\{ \int_{0}^{\infty }\phi \left( z_{n,\theta ,\sigma
}^{(2)}(x,s\eta _{i,n})\right) (1+t_{n,\theta ,\sigma }(x,s\eta _{i,n}))\rho
_{n-k}(s)ds\boldsymbol{1}\left( \alpha _{i,n}^{-1}x+\theta _{i}/\sigma
>0\right) \right.  \notag \\
&&\left. +\int_{0}^{\infty }\phi \left( z_{n,\theta ,\sigma }^{(1)}(x,s\eta
_{i,n})\right) (1-t_{n,\theta ,\sigma }(x,s\eta _{i,n}))\rho _{n-k}(s)ds%
\boldsymbol{1}\left( \alpha _{i,n}^{-1}x+\theta _{i}/\sigma <0\right)
\right\} dx.  \notag
\end{eqnarray}
\end{proposition}

As in the known-variance case the distributions are a convex combination of
pointmass and an absolutely continuous part. In case of hard-thresholding,
the averaging with respect to the density $\rho _{n-k}$ smoothes the
indicator functions leading to a continuous density function for the
absolutely continuous part (while in the known-variance case the density
function is only piece-wise continuous, cf.~Figure 1 in P\"{o}tscher and
Leeb (2009)). This is not so for soft-thresholding and adaptive
soft-thresholding, where the averaging with respect to the density $\rho
_{n-k}$ does not affect the indicator functions involved; here the shape of
the distribution is qualitatively the same as in the known-variance case
(Figure 2 in P\"{o}tscher and Leeb (2009) and Figure 1 in P\"{o}tscher and
Schneider (2009)).

\begin{remark}
\normalfont In the case where $X^{\prime }X$ is diagonal, the finite-sample
distributions of the entire vectors $\tilde{\theta}_{H}$, $\tilde{\theta}%
_{S} $, and $\tilde{\theta}_{AS}$ can be found from the distributions of $%
\hat{\theta}_{H}$, $\hat{\theta}_{S}$, and $\hat{\theta}_{AS}$ (see Remark %
\ref{diag}) by conditioning on $\hat{\sigma}=s\sigma $ and integrating with
respect to $\rho _{n-k}(s)$. In particular, this provides the finite-sample
distributions of the Lasso $\tilde{\theta}_{L}$ and the adaptive Lasso $%
\tilde{\theta}_{AS}$ in the diagonal case (cf.~Remarks \ref{LASSO} and \ref%
{ALASSO}).
\end{remark}

\section{Large-Sample Distributions\label{LS}}

We next derive the asymptotic distributions of the thresholding estimators
under a moving-parameter (and not only under a fixed-parameter) framework
since it is well-known that asymptotics based only on a fixed-parameter
framework often lead to misleading conclusions regarding the performance of
the estimators (cf.~also the discussion in Section \ref{oracle}).

\subsection{The Known-Variance Case\label{LSDKVC}}

We first consider the infeasible versions of the thresholding estimators.

\begin{proposition}
\label{LSDK_H}Suppose that for given $i\geq 1$ satisfying $i\leq k=k(n)$ for
large enough $n$ we have $\xi _{i,n}\eta _{i,n}\rightarrow 0$ and $%
n^{1/2}\eta _{i,n}\rightarrow e_{i}$ where $0\leq e_{i}\leq \infty $.

(a) Assume $e_{i}<\infty $. Set the scaling factor $\alpha
_{i,n}=n^{1/2}/\xi _{i,n}$. Suppose that the true parameters $\theta
^{(n)}=(\theta _{1,n},\dots ,\theta _{k_{n},n})\in \mathbb{R}^{k_{n}}$ and $%
\sigma _{n}\in (0,\infty )$ satisfy $n^{1/2}\theta _{i,n}/(\sigma _{n}\xi
_{i,n})\rightarrow \nu _{i}\in \mathbb{\bar{R}}$. Then $H_{H,n,\theta
^{(n)},\sigma _{n}}^{i}$ converges weakly to the distribution with cdf%
\begin{equation*}
\Phi \left( x\right) \boldsymbol{1}\left( \left\vert x+\nu _{i}\right\vert
>e_{i}\right) +\Phi \left( -\nu _{i}+e_{i}\right) \boldsymbol{1}\left( 0\leq
x+\nu _{i}\leq e_{i}\right) +\Phi \left( -\nu _{i}-e_{i}\right) \boldsymbol{1%
}\left( -e_{i}\leq x+\nu _{i}<0\right) ,
\end{equation*}%
the corresponding measure being%
\begin{equation}
\left\{ \Phi \left( -\nu _{i}+e_{i}\right) -\Phi \left( -\nu
_{i}-e_{i}\right) \right\} d\delta _{-\nu _{i}}(x)+\phi \left( x\right) 
\boldsymbol{1}\left( \left\vert x+\nu _{i}\right\vert >e_{i}\right) dx.
\label{hard_large_sample_unknown_density_B}
\end{equation}%
[This distribution reduces to a standard normal distribution in case $%
\left\vert \nu _{i}\right\vert =\infty $ or $e_{i}=0$.]

(b) Assume $e_{i}=\infty $. Set the scaling factor $\alpha _{i,n}=\left( \xi
_{i,n}\eta _{i,n}\right) ^{-1}$. Suppose that the true parameters $\theta
^{(n)}=(\theta _{1,n},\ldots ,\theta _{k_{n},n})\in \mathbb{R}^{k_{n}}$ and $%
\sigma _{n}\in (0,\infty )$ satisfy $\theta _{i,n}/(\sigma _{n}\xi
_{i,n}\eta _{i,n})\rightarrow \zeta _{i}\in \mathbb{\bar{R}}$.

\qquad 1. If $\left\vert \zeta _{i}\right\vert <1$, then $H_{H,n,\theta
^{(n)},\sigma _{n}}^{i}$ converges weakly to $\delta _{-\zeta _{i}}$.

\qquad 2. If $\left\vert \zeta _{i}\right\vert >1$, then $H_{H,n,\theta
^{(n)},\sigma _{n}}^{i}$ converges weakly to $\delta _{0}$.

\qquad 3. If $\left\vert \zeta _{i}\right\vert =1$ and $n^{1/2}\left( \eta
_{i,n}-\zeta _{i}\theta _{i,n}/(\sigma _{n}\xi _{i,n})\right) \rightarrow
r_{i}$, for some $r_{i}\in \mathbb{\bar{R}}$, then $H_{H,n,\theta
^{(n)},\sigma _{n}}^{i}$ converges weakly to%
\begin{equation*}
\Phi (r_{i})\delta _{-\zeta _{i}}+(1-\Phi (r_{i}))\delta _{0}.
\end{equation*}
\end{proposition}

\begin{proposition}
\label{LSDK_S}Suppose that for given $i\geq 1$ satisfying $i\leq k=k(n)$ for
large enough $n$ we have $\xi _{i,n}\eta _{i,n}\rightarrow 0$ and $%
n^{1/2}\eta _{i,n}\rightarrow e_{i}$ where $0\leq e_{i}\leq \infty $.

(a) Assume $e_{i}<\infty $. Set the scaling factor $\alpha
_{i,n}=n^{1/2}/\xi _{i,n}$. Suppose that the true parameters $\theta
^{(n)}=(\theta _{1,n},\dots ,\theta _{k_{n},n})\in \mathbb{R}^{k_{n}}$ and $%
\sigma _{n}\in (0,\infty )$ satisfy $n^{1/2}\theta _{i,n}/(\sigma _{n}\xi
_{i,n})\rightarrow \nu _{i}\in \mathbb{\bar{R}}$. Then $H_{S,n,\theta
^{(n)},\sigma _{n}}^{i}$ converges weakly to the distribution with cdf%
\begin{equation*}
\Phi \left( x+e_{i}\right) \boldsymbol{1}\left( x+\nu _{i}\geq 0\right)
+\Phi \left( x-e_{i}\right) \boldsymbol{1}\left( x+\nu _{i}<0\right) ,
\end{equation*}%
the corresponding measure being%
\begin{equation}
\left\{ \Phi \left( -\nu _{i}+e_{i}\right) -\Phi \left( -\nu
_{i}-e_{i}\right) \right\} d\delta _{-\nu _{i}}(x)+\left\{ \phi \left(
x+e_{i}\right) \boldsymbol{1}\left( x+\nu _{i}>0\right) +\phi \left(
x-e_{i}\right) \boldsymbol{1}\left( x+\nu _{i}<0\right) \right\} dx.
\label{soft_large_sample_unknown_density_B}
\end{equation}%
[This distribution reduces to a $N(-\limfunc{sign}(\nu _{i})e_{i},1)$%
-distribution in case $\left\vert \nu _{i}\right\vert =\infty $ or $e_{i}=0$%
.]

(b) Assume $e_{i}=\infty $. Set the scaling factor $\alpha _{i,n}=\left( \xi
_{i,n}\eta _{i,n}\right) ^{-1}$. Suppose that the true parameters $\theta
^{(n)}=(\theta _{1,n},\ldots ,\theta _{k_{n},n})\in \mathbb{R}^{k_{n}}$ and $%
\sigma _{n}\in (0,\infty )$ satisfy $\theta _{i,n}/(\sigma _{n}\xi
_{i,n}\eta _{i,n})\rightarrow \zeta _{i}\in \mathbb{\bar{R}}$. Then $%
H_{S,n,\theta ^{(n)},\sigma _{n}}^{i}$ converges weakly to $\delta _{-%
\limfunc{sign}(\zeta _{i})\min (1,\left\vert \zeta _{i}\right\vert )}$.
\end{proposition}

\begin{proposition}
\label{LSDK_AS}Suppose that for given $i\geq 1$ satisfying $i\leq k=k(n)$
for large enough $n$ we have $\xi _{i,n}\eta _{i,n}\rightarrow 0$ and $%
n^{1/2}\eta _{i,n}\rightarrow e_{i}$ where $0\leq e_{i}\leq \infty $.

(a) Assume $e_{i}<\infty $. Set the scaling factor $\alpha
_{i,n}=n^{1/2}/\xi _{i,n}$. Suppose that the true parameters $\theta
^{(n)}=(\theta _{1,n},\dots ,\theta _{k_{n},n})\in \mathbb{R}^{k_{n}}$ and $%
\sigma _{n}\in (0,\infty )$ satisfy $n^{1/2}\theta _{i,n}/(\sigma _{n}\xi
_{i,n})\rightarrow \nu _{i}\in \mathbb{\bar{R}}$. Then $H_{AS,n,\theta
^{(n)},\sigma _{n}}^{i}$ converges weakly to the distribution with cdf%
\begin{eqnarray}
&&\Phi \left( 0.5(x-\nu _{i})+\sqrt{\left( 0.5(x+\nu _{i})\right)
^{2}+e_{i}^{2}}\right) \boldsymbol{1}\left( x+\nu _{i}\geq 0\right)  \notag
\\
&&+\Phi \left( 0.5(x-\nu _{i})-\sqrt{\left( 0.5(x+\nu _{i})\right)
^{2}+e_{i}^{2}}\right) \boldsymbol{1}\left( x+\nu _{i}<0\right)
\label{adaptive_soft_large_sample_unknown_cdf_B}
\end{eqnarray}%
in case $\left\vert \nu _{i}\right\vert <\infty $, the corresponding measure
being%
\begin{eqnarray*}
&&\left\{ \Phi \left( -\nu _{i}+e_{i}\right) -\Phi \left( -\nu
_{i}-e_{i}\right) \right\} d\delta _{-\nu _{i}}(x) \\
&&+0.5\left\{ \phi \left( 0.5(x-\nu _{i})+\sqrt{\left( 0.5(x+\nu
_{i})\right) ^{2}+e_{i}^{2}}\right) \left( 1+t(x)\right) \boldsymbol{1}%
\left( x+\nu _{i}>0\right) \right. \\
&&+\left. \phi \left( 0.5(x-\nu _{i})-\sqrt{\left( 0.5(x+\nu _{i})\right)
^{2}+e_{i}^{2}}\right) \left( 1-t(x)\right) \boldsymbol{1}\left( x+\nu
_{i}<0\right) \right\} dx,
\end{eqnarray*}%
where $t(x)=\left( x+\nu _{i}\right) /\sqrt{\left( \left( x+\nu _{i}\right)
^{2}+4e_{i}^{2}\right) }$. In case $\left\vert \nu _{i}\right\vert =\infty $%
, the cdf $H_{AS,n,\theta ^{(n)},\sigma _{n}}^{i}$ converges weakly to $\Phi 
$, i.e., to a standard normal distribution. [In case $e_{i}=0$ the limit
always reduces to a standard normal distribution.]

(b) Assume $e_{i}=\infty $. Set the scaling factor $\alpha _{i,n}=\left( \xi
_{i,n}\eta _{i,n}\right) ^{-1}$. Suppose that the true parameters $\theta
^{(n)}=(\theta _{1,n},\ldots ,\theta _{k_{n},n})\in \mathbb{R}^{k_{n}}$ and $%
\sigma _{n}\in (0,\infty )$ satisfy $\theta _{i,n}/(\sigma _{n}\xi
_{i,n}\eta _{i,n})\rightarrow \zeta _{i}\in \mathbb{\bar{R}}$.

\qquad 1. If $\left\vert \zeta _{i}\right\vert <1$, then $H_{AS,n,\theta
^{(n)},\sigma _{n}}^{i}$ converges weakly to $\delta _{-\zeta _{i}}$.

\qquad 2. If $1\leq \left\vert \zeta _{i}\right\vert <\infty $, then $%
H_{AS,n,\theta ^{(n)},\sigma _{n}}^{i}$ converges weakly to $\delta
_{-1/\zeta _{i}}$.

\qquad 3. If $\left\vert \zeta _{i}\right\vert =\infty $, then $%
H_{AS,n,\theta ^{(n)},\sigma _{n}}^{i}$ converges weakly to $\delta _{0}$.
\end{proposition}

Observe that the scaling factors $\alpha _{i,n}$ used in the above
propositions are exactly of the same order as $a_{i,n}$ in the case of
conservative as well as in the case of consistent tuning and thus correspond
to the uniform rate of convergence in both cases. In the case of
conservative tuning the limiting distributions have essentially the same
form as the finite-sample distributions, demonstrating that the
moving-parameter asymptotic framework captures the finite-sample behavior of
the estimators in a satisfactory way. In contrast, a fixed-parameter
asymptotic framework, which corresponds to setting $\theta _{i,n}\equiv
\theta _{i}$ and $\sigma _{n}\equiv \sigma $ in the above propositions,
misrepresents the finite-sample properties of the thresholding estimators
whenever $\theta _{i}\neq 0$ but small, as the fixed-parameter limiting
distribution is -- in case of hard-thresholding and adaptive
soft-thresholding -- then always $N(0,1)$, regardless of the size of $\theta
_{i}$. For soft-thresholding we also observe a strong discrepancy between
the finite-sample distribution and the fixed-parameter limit for $\theta
_{i}\neq 0$ which is given by $N(-\limfunc{sign}(\theta _{i})e_{i},1)$. In
particular, the above propositions demonstrate non-uniformity in the
convergence of finite-sample distributions to their limit in a
fixed-parameter framework.

In the case of consistent tuning we observe an interesting phenomenon,
namely that the limiting distributions now correspond to pointmasses (but
not always located at zero!), or are convex combinations of two pointmasses
in some cases when considering the hard-thresholding estimator. This
essentially means that consistently tuned thresholding estimators are
plagued by a bias-problem in that the "bias-component" is the dominant
component and is of larger order than the "stochastic variability" of the
estimator.\footnote{%
For the hard-thresholding estimator some randomness survives in the limit in
the case $\left\vert \zeta _{i}\right\vert =1$, where we can achieve a
limiting probability for $\hat{\theta}_{H,i}=0$ that is strictly between $0$
and $1$. That this randomness does not survive for the other two estimators
in the limit seems to be connected to the fact that these estimators are
continuous functions of the data, whereas $\hat{\theta}_{H,i}$ is not.} In a
fixed-parameter framework we get the trivial limits $\delta _{0}$ for every
value of $\theta _{i}$ in case of hard-thresholding and adaptive
soft-thresholding. At first glance this seems to suggest that we have used a
scaling sequence that does not increase fast enough with $n$, but recall
that the scaling used here corresponds to the uniform convergence rate. We
shall take this issue further up in Section \ref{oracle}. The situation is
different for the soft-thresholding estimator where the fixed-parameter
limit is $\delta _{-\limfunc{sign}(\theta _{i})}$, which reduces to $\delta
_{0}$ only for $\theta _{i}=0$; this is a reflection of the well-known fact
that soft-thresholding is plagued by bias problems to a higher degree than
are hard-thresholding and adaptive soft-thresholding.

\subsection{Uniform Closeness of Distributions in the Known- and
Unknown-Variance Case\label{uniform_close}}

We next show that the finite-sample cdfs of $\tilde{\theta}_{H,i}$, $\tilde{%
\theta}_{S,i}$, and $\tilde{\theta}_{AS,i}$ and of their infeasible
counterparts $\hat{\theta}_{H,i}$, $\hat{\theta}_{S,i}$, and $\hat{\theta}%
_{AS,i}$, respectively, are uniformly (with respect to the parameters) close
in the total variation distance (or the supremum norm) provided the number
of degrees of freedom $n-k$ diverges to infinity fast enough. Apart from
being of interest in their own right, these results will be instrumental in
the subsequent section. We note that the results in Theorem \ref{closeness}
below hold for any choice of the scaling factors $\alpha _{i,n}$.

\begin{theorem}
\label{closeness} Suppose that for given $i\geq 1$ satisfying $i\leq k=k(n)$
for large enough $n$ we have $n^{1/2}\eta _{i,n}(n-k)^{-1/2}\rightarrow 0$
as $n\rightarrow \infty $. Then%
\begin{equation*}
\sup_{\theta \in \mathbb{R}^{k},0<\sigma <\infty }\left\Vert H_{H,n,\theta
,\sigma }^{i}-H_{H,n,\theta ,\sigma }^{i\maltese }\right\Vert
_{TV}\rightarrow 0\qquad \text{for }n\rightarrow \infty ,
\end{equation*}%
\begin{equation*}
\sup_{\theta \in \mathbb{R}^{k},0<\sigma <\infty }\left\Vert H_{S,n,\theta
,\sigma }^{i}-H_{S,n,\theta ,\sigma }^{i\maltese }\right\Vert
_{TV}\rightarrow 0\qquad \text{for }n\rightarrow \infty ,
\end{equation*}%
and 
\begin{equation*}
\sup_{\theta \in \mathbb{R}^{k},0<\sigma <\infty }\left\Vert H_{AS,n,\theta
,\sigma }^{i}-H_{AS,n,\theta ,\sigma }^{i\maltese }\right\Vert _{\infty
}\rightarrow 0\qquad \text{for }n\rightarrow \infty
\end{equation*}%
hold.\footnote{%
Uniform closeness of the respective cdfs of the adaptive soft-thresholding
estimators in the total variation distance, and not only in the supremum
norm, could probably be obtained at the expense of a more cumbersome proof.
We do not pursue this.}
\end{theorem}

\begin{remark}
\label{n-k}\normalfont In case of conservative tuning, the condition $%
n^{1/2}\eta _{i,n}(n-k)^{-1/2}\rightarrow 0$ is always satisfied if $%
n-k\rightarrow \infty $. [In fact it is then equivalent to $n-k\rightarrow
\infty $ or $e_{i}=0$.] In case of consistent tuning $n-k\rightarrow \infty $
is clearly a weaker condition than $n^{1/2}\eta
_{i,n}(n-k)^{-1/2}\rightarrow 0$. However, in general, a sufficient
condition for $n^{1/2}\eta _{i,n}(n-k)^{-1/2}\rightarrow 0$ is that $\eta
_{i,n}\rightarrow 0$ and $\limsup_{n\rightarrow \infty }k/n<1$.
\end{remark}

\begin{remark}
\label{scaleinv}\normalfont Suppose that $\xi _{i,n}\eta _{i,n}\rightarrow 0$
holds as $n\rightarrow \infty $. If $n^{1/2}\eta _{i,n}(n-k)^{-1/2}$ does
not converge to zero as $n\rightarrow \infty $, Remark \ref{weekend} shows
that none of the convergence results in Theorem \ref{closeness} holds. [To
see this note that the variable deletion probabilities constitute the weight
of the pointmass in the respective distribution functions.] This shows that
the condition $n^{1/2}\eta _{i,n}(n-k)^{-1/2}\rightarrow 0$ in the above
theorem cannot be weakened (at least in case $\xi _{i,n}\eta
_{i,n}\rightarrow 0$ holds).
\end{remark}

\subsection{The Unknown-Variance Case\label{LSDUKVC}}

\subsubsection{Conservative Tuning\label{conservative}}

We next obtain the limiting distributions of $\tilde{\theta}_{H,i}$, $\tilde{%
\theta}_{S,i}$, and $\tilde{\theta}_{AS,i}$ in a moving-parameter framework
under conservative tuning.

\begin{theorem}
\label{HTconservative}(Hard-thresholding with conservative tuning) Suppose
that for given $i\geq 1$ satisfying $i\leq k=k(n)$ for large enough $n$ we
have $\xi _{i,n}\eta _{i,n}\rightarrow 0$ and $n^{1/2}\eta _{i,n}\rightarrow
e_{i}$ where $0\leq e_{i}<\infty $. Set the scaling factor $\alpha
_{i,n}=n^{1/2}/\xi _{i,n}$. Suppose that the true parameters $\theta
^{(n)}=(\theta _{1,n},\dots ,\theta _{k_{n},n})\in \mathbb{R}^{k_{n}}$ and $%
\sigma _{n}\in (0,\infty )$ satisfy $n^{1/2}\theta _{i,n}/(\sigma _{n}\xi
_{i,n})\rightarrow \nu _{i}\in \mathbb{\bar{R}}$.

(a) If $n-k$ is eventually constant equal to $m$, say, then $H_{H,n,\theta
^{(n)},\sigma _{n}}^{i\maltese }$ converges weakly to the distribution with
cdf%
\begin{eqnarray*}
&&\int_{0}^{\infty }\left\{ \Phi \left( x\right) \boldsymbol{1}\left(
\left\vert x+\nu _{i}\right\vert >se_{i}\right) +\Phi \left( -\nu
_{i}+se_{i}\right) \boldsymbol{1}\left( 0\leq x+\nu _{i}\leq se_{i}\right)
\right. \\
&&+\left. \Phi \left( -\nu _{i}-se_{i}\right) \boldsymbol{1}\left(
-se_{i}\leq x+\nu _{i}<0\right) \right\} \rho _{m}(s)ds,
\end{eqnarray*}%
the corresponding measure being%
\begin{equation}
\int_{0}^{\infty }\left\{ \Phi \left( -\nu _{i}+se_{i}\right) -\Phi \left(
-\nu _{i}-se_{i}\right) \right\} \rho _{m}(s)dsd\delta _{-\nu _{i}}(x)+\phi
\left( x\right) \int_{0}^{\infty }\boldsymbol{1}\left( \left\vert x+\nu
_{i}\right\vert >se_{i}\right) \rho _{m}(s)dsdx.
\label{hard_large_sample_unknown_density_A}
\end{equation}%
[The distribution reduces to a standard normal distribution in case $%
\left\vert \nu _{i}\right\vert =\infty $ or $e_{i}=0$.]

(b) If $n-k\rightarrow \infty $ holds, then $H_{H,n,\theta ^{(n)},\sigma
_{n}}^{i\maltese }$ converges weakly to the distribution given in
Proposition \ref{LSDK_H}(a).
\end{theorem}

\begin{theorem}
\label{STconservative}(Soft-thresholding with conservative tuning) Suppose
that for given $i\geq 1$ satisfying $i\leq k=k(n)$ for large enough $n$ we
have $\xi _{i,n}\eta _{i,n}\rightarrow 0$ and $n^{1/2}\eta _{i,n}\rightarrow
e_{i}$ where $0\leq e_{i}<\infty $. Set the scaling factor $\alpha
_{i,n}=n^{1/2}/\xi _{i,n}$. Suppose that the true parameters $\theta
^{(n)}=(\theta _{1,n},\dots ,\theta _{k_{n},n})\in \mathbb{R}^{k_{n}}$ and $%
\sigma _{n}\in (0,\infty )$ satisfy $n^{1/2}\theta _{i,n}/(\sigma _{n}\xi
_{i,n})\rightarrow \nu _{i}\in \mathbb{\bar{R}}$.

(a) If $n-k$ is eventually constant equal to $m$, say, then $H_{S,n,\theta
^{(n)},\sigma _{n}}^{i\maltese }$ converges weakly to the distribution with
cdf%
\begin{equation*}
\int_{0}^{\infty }\left\{ \Phi \left( x+se_{i}\right) \boldsymbol{1}\left(
x+\nu _{i}\geq 0\right) +\Phi \left( x-se_{i}\right) \boldsymbol{1}\left(
x+\nu _{i}<0\right) \right\} \rho _{m}(s)ds,
\end{equation*}%
the corresponding measure being%
\begin{eqnarray}
&&\int_{0}^{\infty }\left\{ \Phi \left( -\nu _{i}+se_{i}\right) -\Phi \left(
-\nu _{i}-se_{i}\right) \right\} \rho _{m}(s)dsd\delta _{-\nu _{i}}(x) 
\notag \\
&&+\int_{0}^{\infty }\left\{ \phi \left( x+se_{i}\right) \boldsymbol{1}%
\left( x+\nu _{i}>0\right) +\phi \left( x-se_{i}\right) \boldsymbol{1}\left(
x+\nu _{i}<0\right) \right\} \rho _{m}(s)dsdx.
\label{soft_large_sample_unknown_density_A}
\end{eqnarray}%
[The atomic part in the above expression is absent in case $\left\vert \nu
_{i}\right\vert =\infty $. Furthermore, the distribution reduces to a
standard normal distribution if $e_{i}=0$.]

(b) If $n-k\rightarrow \infty $ holds, then $H_{S,n,\theta ^{(n)},\sigma
_{n}}^{i\maltese }$ converges weakly to the distribution given in
Proposition \ref{LSDK_S}(a).
\end{theorem}

\begin{theorem}
\label{ASTconservative}(Adaptive soft-thresholding with conservative tuning)
Suppose that for given $i\geq 1$ satisfying $i\leq k=k(n)$ for large enough $%
n$ we have $\xi _{i,n}\eta _{i,n}\rightarrow 0$ and $n^{1/2}\eta
_{i,n}\rightarrow e_{i}$ where $0\leq e_{i}<\infty $. Set the scaling factor 
$\alpha _{i,n}=n^{1/2}/\xi _{i,n}$. Suppose that the true parameters $\theta
^{(n)}=(\theta _{1,n},\dots ,\theta _{k_{n},n})\in \mathbb{R}^{k_{n}}$ and $%
\sigma _{n}\in (0,\infty )$ satisfy $n^{1/2}\theta _{i,n}/(\sigma _{n}\xi
_{i,n})\rightarrow \nu _{i}\in \mathbb{\bar{R}}$.

(a) Suppose $n-k$ is eventually constant equal to $m$, say. Then $%
H_{AS,n,\theta ^{(n)},\sigma _{n}}^{i\maltese }$ converges weakly to the
distribution with cdf%
\begin{eqnarray}
&&\int_{0}^{\infty }\Phi \left( 0.5(x-\nu _{i})+\sqrt{\left( 0.5(x+\nu
_{i})\right) ^{2}+s^{2}e_{i}^{2}}\right) \rho _{m}(s)ds\boldsymbol{1}\left(
x+\nu _{i}\geq 0\right)  \notag \\
&&+\int_{0}^{\infty }\Phi \left( 0.5(x-\nu _{i})-\sqrt{\left( 0.5(x+\nu
_{i})\right) ^{2}+s^{2}e_{i}^{2}}\right) \rho _{m}(s)ds\boldsymbol{1}\left(
x+\nu _{i}<0\right)  \label{adaptive_soft_large_sample_unknown_cdf_A}
\end{eqnarray}%
in case $\left\vert \nu _{i}\right\vert <\infty $, the corresponding measure
being given by%
\begin{eqnarray*}
&&\int_{0}^{\infty }\left\{ \Phi \left( -\nu _{i}+se_{i}\right) -\Phi \left(
-\nu _{i}-se_{i}\right) \right\} \rho _{m}(s)dsd\delta _{-\nu _{i}}(x) \\
&&+0.5\int_{0}^{\infty }\left\{ \phi \left( 0.5(x-\nu _{i})+\sqrt{\left(
0.5(x+\nu _{i})\right) ^{2}+s^{2}e_{i}^{2}}\right) \left( 1+t(x,s)\right) 
\boldsymbol{1}\left( x+\nu _{i}>0\right) \right. \\
&&+\left. \phi \left( 0.5(x-\nu _{i})-\sqrt{\left( 0.5(x+\nu _{i})\right)
^{2}+s^{2}e_{i}^{2}}\right) \left( 1-t(x,s)\right) \boldsymbol{1}\left(
x+\nu _{i}<0\right) \right\} \rho _{m}(s)dsdx,
\end{eqnarray*}%
where $t(x,s)=\left( x+\nu _{i}\right) /\sqrt{\left( \left( x+\nu
_{i}\right) ^{2}+4s^{2}e_{i}^{2}\right) }$. In case $\left\vert \nu
_{i}\right\vert =\infty $, the cdf $H_{AS,n,\theta ^{(n)},\sigma
_{n}}^{i\maltese }$ converges weakly to $\Phi $, i.e., a standard normal
distribution. [If $e_{i}=0$, the limit always reduces to a standard normal
distribution.]

(b) If $n-k\rightarrow \infty $, then $H_{AS,n,\theta ^{(n)},\sigma
_{n}}^{i\maltese }$ converges weakly to the distribution given in
Proposition \ref{LSDK_AS}(a).
\end{theorem}

It transpires that in case of conservative tuning and $n-k\rightarrow \infty 
$ we obtain exactly the same limiting distributions as in the known-variance
case and hence the relevant discussion given at the end of Section \ref%
{LSDKVC} applies also here. [That one obtains the same limits does not come
as a surprise given the results in Section \ref{uniform_close} and the
observation made in Remark \ref{n-k}.] In the case, where $n-k$ is
eventually constant, the limits are obtained from the limits in the
known-variance case (with $\sigma $ replaced by $\sigma s$) by averaging
with respect to the distribution of $\hat{\sigma}/\sigma $. Again the
limiting distributions essentially have the same structure as the
corresponding finite-sample distributions. The fixed-parameter limiting
distributions (corresponding to setting $\theta _{i,n}\equiv \theta _{i}$
and $\sigma _{n}\equiv \sigma $ in the above theorems) again misrepresent
the finite-sample properties of the thresholding estimators whenever $\theta
_{i}\neq 0$ but small, as the fixed-parameter limiting distribution is -- in
case of hard-thresholding and adaptive soft-thresholding -- then always $%
N(0,1)$, regardless of the size of $\theta _{i}$. For soft-thresholding we
also observe a strong discrepancy between the finite-sample distribution and
the fixed-parameter limit especially for $\theta _{i}\neq 0$ but small,
which is given by the distribution with pdf $\int_{0}^{\infty }\phi \left(
x+s\limfunc{sign}(\theta _{i})e_{i}\right) \rho _{m}(s)ds$ regardless of the
size of $\theta _{i}$. As a consequence, we again observe non-uniformity in
the convergence of finite-sample distributions to their limit in a
fixed-parameter framework also in the case where the number of degrees of
freedom is (eventually) constant.

\subsubsection{Consistent Tuning\label{consistent}}

We next derive the limiting distributions of $\tilde{\theta}_{H,i}$, $\tilde{%
\theta}_{S,i}$, and $\tilde{\theta}_{AS,i}$ in a moving-parameter framework
under consistent tuning.

\begin{theorem}
\label{HTconsistent}(Hard-thresholding with consistent tuning) Suppose that
for given $i\geq 1$ satisfying $i\leq k=k(n)$ for large enough $n$ we have $%
\xi _{i,n}\eta _{i,n}\rightarrow 0$ and $n^{1/2}\eta _{i,n}\rightarrow
\infty $. Set the scaling factor $\alpha _{i,n}=\left( \xi _{i,n}\eta
_{i,n}\right) ^{-1}$. Suppose that the true parameters $\theta
^{(n)}=(\theta _{1,n},\ldots ,\theta _{k_{n},n})\in \mathbb{R}^{k_{n}}$ and $%
\sigma _{n}\in (0,\infty )$ satisfy $\theta _{i,n}/(\sigma _{n}\xi
_{i,n}\eta _{i,n})\rightarrow \zeta _{i}\in \mathbb{\bar{R}}$.

(a) If $n-k$ is eventually constant equal to $m$, say, then $H_{H,n,\theta
^{(n)},\sigma _{n}}^{i\maltese }$ converges weakly to%
\begin{eqnarray*}
&&\left( \int_{\left\vert \zeta _{i}\right\vert }^{\infty }\rho
_{m}(s)ds\right) \delta _{-\zeta _{i}}+\left( 1-\int_{\left\vert \zeta
_{i}\right\vert }^{\infty }\rho _{m}(s)ds\right) \delta _{0} \\
&=&\Pr (\chi _{m}^{2}>m\zeta _{i}^{2})\delta _{-\zeta _{i}}+\Pr (\chi
_{m}^{2}\leq m\zeta _{i}^{2})\delta _{0}.
\end{eqnarray*}%
[The above display reduces to $\delta _{0}$ for $\left\vert \zeta
_{i}\right\vert =\infty $.]

(b) If $n-k\rightarrow \infty $ holds, then

\qquad 1. $\left\vert \zeta _{i}\right\vert <1$ implies that $H_{H,n,\theta
^{(n)},\sigma _{n}}^{i\maltese }$ converges weakly to $\delta _{-\zeta _{i}}$%
.

\qquad 2. $\left\vert \zeta _{i}\right\vert >1$ implies that $H_{H,n,\theta
^{(n)},\sigma _{n}}^{i\maltese }$ converges weakly to $\delta _{0}$.

\qquad 3. $\left\vert \zeta _{i}\right\vert =1$ and $n^{1/2}\eta
_{i,n}/\left( n-k\right) ^{1/2}\rightarrow 0$ imply that $H_{H,n,\theta
^{(n)},\sigma _{n}}^{i\maltese }$ converges weakly to%
\begin{equation*}
\Phi (r_{i})\delta _{-\zeta _{i}}+\left( 1-\Phi (r_{i})\right) \delta _{0}
\end{equation*}%
provided $r_{i,n}=n^{1/2}\left( \eta _{i,n}-\zeta _{i}\theta _{i,n}/(\sigma
_{n}\xi _{i,n})\right) \rightarrow r_{i}$ for some $r_{i}\in \mathbb{\bar{R}}
$.

\qquad 4. $\left\vert \zeta _{i}\right\vert =1$ and $n^{1/2}\eta
_{i,n}/\left( n-k\right) ^{1/2}\rightarrow 2^{1/2}d_{i}$ with $%
0<d_{i}<\infty $ imply that $H_{H,n,\theta ^{(n)},\sigma _{n}}^{i\maltese }$
converges weakly to%
\begin{equation*}
\left( \int_{-\infty }^{\infty }\Phi (d_{i}t+r_{i})\phi (t)dt\right) \delta
_{-\zeta _{i}}+\left( 1-\int_{-\infty }^{\infty }\Phi (d_{i}t+r_{i})\phi
(t)dt\right) \delta _{0}
\end{equation*}%
provided $r_{i,n}\rightarrow r_{i}$ for some $r_{i}\in \mathbb{\bar{R}}$.
[Note that the above display reduces to $\delta _{-\zeta _{i}}$ if $%
r_{i}=\infty $, and to $\delta _{0}$ if $r_{i}=-\infty $.]

\qquad 5. $\left\vert \zeta _{i}\right\vert =1$ and $n^{1/2}\eta
_{i,n}/\left( n-k\right) ^{1/2}\rightarrow \infty $ imply that $%
H_{H,n,\theta ^{(n)},\sigma _{n}}^{i\maltese }$ converges weakly to%
\begin{equation*}
\Phi (r_{i}^{\prime })\delta _{-\zeta _{i}}+\left( 1-\Phi (r_{i}^{\prime
})\right) \delta _{0}
\end{equation*}%
provided $\left( n^{1/2}\eta _{i,n}/\left( n-k\right) ^{1/2}\right)
^{-1}r_{i,n}\rightarrow 2^{-1/2}r_{i}^{\prime }$ for some $r_{i}^{\prime
}\in \mathbb{\bar{R}}$.
\end{theorem}

\begin{theorem}
\label{STconsistent}(Soft-thresholding with consistent tuning) Suppose that
for given $i\geq 1$ satisfying $i\leq k=k(n)$ for large enough $n$ we have $%
\xi _{i,n}\eta _{i,n}\rightarrow 0$ and $n^{1/2}\eta _{i,n}\rightarrow
\infty $. Set the scaling factor $\alpha _{i,n}=\left( \xi _{i,n}\eta
_{i,n}\right) ^{-1}$. Suppose that the true parameters $\theta
^{(n)}=(\theta _{1,n},\ldots ,\theta _{k_{n},n})\in \mathbb{R}^{k_{n}}$ and $%
\sigma _{n}\in (0,\infty )$ satisfy $\theta _{i,n}/(\sigma _{n}\xi
_{i,n}\eta _{i,n})\rightarrow \zeta _{i}\in \mathbb{\bar{R}}$.

(a) If $n-k$ is eventually constant equal to $m$, say, then $H_{S,n,\theta
^{(n)},\sigma _{n}}^{i\maltese }$ converges weakly to the distribution given
by%
\begin{eqnarray}
&&\int_{\left\vert \zeta _{i}\right\vert }^{\infty }\rho _{m}(s)dsd\delta
_{-\zeta _{i}}(x)+\left\{ \rho _{m}(x)\boldsymbol{1}\left( x+\zeta
_{i}<0\right) +\rho _{m}(-x)\boldsymbol{1}\left( x+\zeta _{i}>0\right)
\right\} dx  \notag \\
&=&\Pr (\chi _{m}^{2}>m\zeta _{i}^{2})d\delta _{-\zeta _{i}}(x)+\left\{ \rho
_{m}(x)\boldsymbol{1}\left( x+\zeta _{i}<0\right) +\rho _{m}(-x)\boldsymbol{1%
}\left( x+\zeta _{i}>0\right) \right\} dx,
\label{soft_large_sample_unknown_density_C}
\end{eqnarray}%
where we recall the convention that $\rho _{m}(x)=0$ for $x<0$. [In case $%
\left\vert \zeta _{i}\right\vert =\infty $, the atomic part in (\ref%
{soft_large_sample_unknown_density_C}) is absent and (\ref%
{soft_large_sample_unknown_density_C}) reduces to $\rho _{m}(-\limfunc{sign}%
(\zeta _{i})x)dx$.]

(b) If $n-k\rightarrow \infty $ holds, then $H_{S,n,\theta ^{(n)},\sigma
_{n}}^{i\maltese }$ converges weakly to $\delta _{-\limfunc{sign}(\zeta
_{i})\min (1,\left\vert \zeta _{i}\right\vert )}$.
\end{theorem}

\begin{theorem}
\label{ASTconsistent}(Adaptive soft-thresholding with consistent tuning)
Suppose that for given $i\geq 1$ satisfying $i\leq k=k(n)$ for large enough $%
n$ we have $\xi _{i,n}\eta _{i,n}\rightarrow 0$ and $n^{1/2}\eta
_{i,n}\rightarrow \infty $. Set the scaling factor $\alpha _{i,n}=\left( \xi
_{i,n}\eta _{i,n}\right) ^{-1}$. Suppose that the true parameters $\theta
^{(n)}=(\theta _{1,n},\ldots ,\theta _{k_{n},n})\in \mathbb{R}^{k_{n}}$ and $%
\sigma _{n}\in (0,\infty )$ satisfy $\theta _{i,n}/(\sigma _{n}\xi
_{i,n}\eta _{i,n})\rightarrow \zeta _{i}\in \mathbb{\bar{R}}$.

(a) Suppose $n-k$ is eventually constant equal to $m$, say. Then $%
H_{AS,n,\theta ^{(n)},\sigma _{n}}^{i\maltese }$ converges weakly to the
distribution with cdf%
\begin{eqnarray*}
&&\int_{\sqrt{\left\vert x\zeta _{i}\right\vert }}^{\infty }\rho _{m}(s)ds%
\boldsymbol{1}\left( -\zeta _{i}\leq x<0\right) +\boldsymbol{1}\left( x\geq
0\right) \\
&=&\Pr (\chi _{m}^{2}>m\left\vert x\zeta _{i}\right\vert )\boldsymbol{1}%
\left( -\zeta _{i}\leq x<0\right) +\boldsymbol{1}\left( x\geq 0\right)
\end{eqnarray*}%
in case $0\leq \zeta _{i}<\infty $, and to the distribution with cdf%
\begin{eqnarray*}
&&\int_{0}^{\sqrt{\left\vert x\zeta _{i}\right\vert }}\rho _{m}(s)ds%
\boldsymbol{1}\left( 0\leq x<-\zeta _{i}\right) +\boldsymbol{1}\left( x\geq
-\zeta _{i}\right) \\
&=&\Pr (\chi _{m}^{2}\leq m\left\vert x\zeta _{i}\right\vert )\boldsymbol{1}%
\left( 0\leq x<-\zeta _{i}\right) +\boldsymbol{1}\left( x\geq -\zeta
_{i}\right)
\end{eqnarray*}%
in case $-\infty <\zeta _{i}<0$. Furthermore, $H_{AS,n,\theta ^{(n)},\sigma
_{n}}^{i\maltese }$ converges weakly to $\delta _{0}$ if $\left\vert \zeta
_{i}\right\vert =\infty $. [In case $\left\vert \zeta _{i}\right\vert
<\infty $, the distribution has a jump of height $\int_{\left\vert \zeta
_{i}\right\vert }^{\infty }\rho _{m}(s)=\Pr (\chi _{m}^{2}>m\zeta _{i}^{2})$
at $x=-\zeta _{i}$ and is otherwise absolutely continuous. In particular, it
reduces to $\delta _{0}$ in case $\zeta _{i}=0$.]

(b) If $n-k\rightarrow \infty $ holds, then

\qquad 1. $\left\vert \zeta _{i}\right\vert \leq 1$ implies that $%
H_{AS,n,\theta ^{(n)},\sigma _{n}}^{i\maltese }$ converges weakly to $\delta
_{-\zeta _{i}}$,

\qquad 2. $1<\left\vert \zeta _{i}\right\vert <\infty $ implies that $%
H_{AS,n,\theta ^{(n)},\sigma _{n}}^{i\maltese }$ converges weakly to $\delta
_{-1/\zeta _{i}}$,

\qquad 3. $\left\vert \zeta _{i}\right\vert =\infty $ implies that $%
H_{AS,n,\theta ^{(n)},\sigma _{n}}^{i\maltese }$ converges weakly to $\delta
_{0}$.
\end{theorem}

We know from Theorem \ref{closeness} that we obtain the same limiting
distributions for $\tilde{\theta}_{H,i}$, $\tilde{\theta}_{S,i}$, and $%
\tilde{\theta}_{AS,i}$ as for $\hat{\theta}_{H,i}$, $\hat{\theta}_{S,i}$,
and $\hat{\theta}_{AS,i}$, respectively, provided $n-k$ diverges to infinity
sufficiently fast in the sense that $n^{1/2}\eta
_{i,n}(n-k)^{-1/2}\rightarrow 0$. The theorems in this section now show that
for the soft-thresholding as well as for the adaptive soft-thresholding
estimator we actually get the same limiting distribution as in the
unknown-variance case whenever $n-k$ diverges even if $n^{1/2}\eta
_{i,n}(n-k)^{-1/2}\rightarrow 0$ is violated. However, for the
hard-thresholding estimator the picture is different, and in case $n-k$
diverges but $n^{1/2}\eta _{i,n}(n-k)^{-1/2}\rightarrow 0$ is violated,
limit distributions different from the known-variance case arise (these
limiting distributions still being convex combinations of two pointmasses,
but with weights different from the known-variance case). It seems that this
is a reflection of the fact that the hard-thresholding estimator is a
discontinuous function of the data, whereas the other two estimators
considered depend continuously on the data. The fixed-parameter limiting
distributions for all three estimators are again the same as in the
known-variance case.

In the case where the degrees of freedom $n-k$ are eventually constant, the
limiting distribution of the hard-thresholding estimator is again a convex
combination of two pointmasses, with weights that are in general different
from the known-variance case. However, for the soft-thresholding as well as
for the adaptive soft-thresholding estimator the limiting distributions can
also contain an absolutely continuous component. This component seems to
stem from an interaction of the more pronounced "bias-component" (as
compared to hard-thresholding) with the nonvanishing randomness in the
estimated variance. The fixed-parameter limiting distributions for
hard-thresholding and adaptive soft-thresholding are again given by $\delta
_{0}$ for all values of $\theta _{i}$ as in the known-variance case, whereas
for soft-thresholding the fixed-parameter limiting distribution is $\delta
_{0}$ only for $\theta _{i}=0$ and otherwise has a pdf given by $\rho _{m}(-%
\limfunc{sign}(\theta _{i})x)$ (as compared to a limit of $\delta _{-%
\limfunc{sign}(\theta _{i})}$ in the known-variance case).

\subsection{Consistent Tuning: Some Comments on Fixed-Parameter Large-Sample
Distributions and the "Oracle-Property" \label{oracle}}

\subsubsection{Hard-Thresholding and Adaptive Soft-Thresholding}

As already mentioned at the end of Sections \ref{LSDKVC} and \ref{consistent}%
, under consistent tuning the \emph{fixed-parameter} limiting distributions
of the hard-thresholding and of the adaptive soft-thresholding estimator --
in the known-variance as well as in the unknown-variance case -- always
degenerate to pointmass at zero. Recall that in these results the estimators
(after centering at $\theta _{i}$) are scaled by $\sigma ^{-1}\left( \xi
_{i,n}\eta _{i,n}\right) ^{-1}$, which corresponds to the uniform
convergence rate. We next show that if the estimators are scaled by $\sigma
^{-1}n^{1/2}\xi _{i,n}^{-1}$ instead, a limit distribution under \emph{%
fixed-parameter} asymptotics arises that is not degenerate in general (under
an additional condition on the tuning parameter in case of adaptive
soft-thresholding). In fact, we show that the hard-thresholding as well as
the adaptive soft-thresholding estimators then satisfy what has been called
the "oracle-property". However, it should be kept in mind that -- with this
faster scaling sequence $\sigma ^{-1}n^{1/2}\xi _{i,n}^{-1}$ -- the centered
estimators are no longer stochastically bounded in a moving-parameter
framework (for certain sequences of parameters), cf.~Theorem \ref%
{thresh_consistency}. This shows the fragility of the "oracle-property",
which is a fixed-parameter concept, and calls into question the statistical
significance of this notion. For a more extensive discussion of the
"oracle-property" and its consequences see Leeb and P\"{o}tscher (2008), P%
\"{o}tscher and Leeb (2009), and P\"{o}tscher and Schneider (2009).

\begin{proposition}
\label{oracle_1}Let $0<\sigma <\infty $ be given. Suppose that for given $%
i\geq 1$ satisfying $i\leq k=k(n)$ for large enough $n$ we have $\xi
_{i,n}\eta _{i,n}\rightarrow 0$ and $n^{1/2}\eta _{i,n}\rightarrow \infty $.

(a) $\sigma ^{-1}n^{1/2}\xi _{i,n}^{-1}\left( \tilde{\theta}_{H,i}-\theta
_{i}\right) $ as well as $\sigma ^{-1}n^{1/2}\xi _{i,n}^{-1}\left( \hat{%
\theta}_{H,i}-\theta _{i}\right) $ converge in distribution to $N(0,1)$ when 
$\theta _{i}\neq 0$, and to $\delta _{0}=N(0,0)$ when $\theta _{i}=0$.

(b) $\sigma ^{-1}n^{1/2}\xi _{i,n}^{-1}\left( \tilde{\theta}_{AS,i}-\theta
_{i}\right) $ as well as $\sigma ^{-1}n^{1/2}\xi _{i,n}^{-1}\left( \hat{%
\theta}_{AS,i}-\theta _{i}\right) $ converge in distribution to $N(0,1)$
when $\theta _{i}\neq 0$, and to $\delta _{0}=N(0,0)$ when $\theta _{i}=0$,
provided the tuning parameter additionally satisfies $n^{1/4}\xi
_{i,n}^{1/2}\eta _{i,n}\rightarrow 0$ for $n\rightarrow \infty $.
\end{proposition}

\begin{remark}
\normalfont Inspection of the proof of Part (b) given in Section \ref%
{prfs_LS} shows that the condition $n^{1/4}\xi _{i,n}^{1/2}\eta
_{i,n}\rightarrow 0$ is used for the result only in case $\theta _{i}\neq 0$%
. If now $n^{1/4}\xi _{i,n}^{1/2}\eta _{i,n}\rightarrow \omega $ with $%
0<\omega <\infty $, inspection of the proof shows that then in case $\theta
_{i}\neq 0$ we have that $\sigma ^{-1}n^{1/2}\xi _{i,n}^{-1}\left( \tilde{%
\theta}_{AS,i}-\theta _{i}\right) =Z_{n}-\sigma \omega ^{2}\theta
_{i}^{-1}\left( \hat{\sigma}/\sigma \right) ^{2}+o_{p}(1)$, where $Z_{n}$ is
standard normal and is independent of $\hat{\sigma}/\sigma $. Hence, we see
that the distribution of $\sigma ^{-1}n^{1/2}\xi _{i,n}^{-1}\left( \tilde{%
\theta}_{AS,i}-\theta _{i}\right) $ asymptotically behaves like the
convolution of an $N(0,1)$-distribution and the distribution of $-\sigma
\omega ^{2}\theta _{i}^{-1}(n-k)^{-1}$ times a chi-square distributed random
variable with $n-k$ degrees of freedom (if $n-k\rightarrow \infty $ this
reduces to an $N(-\sigma \omega ^{2}\theta _{i}^{-1},1)$-distribution). If $%
n^{1/4}\xi _{i,n}^{1/2}\eta _{i,n}\rightarrow \infty $, then $\sigma
^{-1}n^{1/2}\xi _{i,n}^{-1}\left( \tilde{\theta}_{AS,i}-\theta _{i}\right) $
is stochastically unbounded. Note that this shows that the consistently
tuned adaptive soft-thresholding estimator -- even in a fixed-parameter
setting -- has a convergence rate slower than $n^{1/2}\xi _{i,n}^{-1}$ if $%
\theta _{i}\neq 0$ and if the tuning parameter is "too large" in the sense
that $n^{1/4}\xi _{i,n}^{1/2}\eta _{i,n}\rightarrow \infty $. The same
conclusion applies to the infeasible estimator $\hat{\theta}_{AS,i}$ (with
the simplification that one always obtains an $N(-\sigma \omega ^{2}\theta
_{i}^{-1},1)$-distribution in case $n^{1/4}\xi _{i,n}^{1/2}\eta
_{i,n}\rightarrow \omega $ with $0<\omega <\infty $).
\end{remark}

We further illustrate the fragility of the fixed-parameter asymptotic
results under a $\sigma ^{-1}n^{1/2}\xi _{i,n}^{-1}$-scaling obtained above
by providing the moving-parameter limits under this scaling. Let $%
F_{H,n,\theta ,\sigma }^{i}:=F_{H,\eta _{i,n},n,\theta ,\sigma }^{i}$ denote
the cdf of $\sigma ^{-1}n^{1/2}\xi _{i,n}^{-1}(\hat{\theta}_{H,i}-\theta
_{i})$, and define $F_{S,n,\theta ,\sigma }^{i}$ and $F_{AS,n,\theta ,\sigma
}^{i}$ analogously. The proofs of the subsequent propositions are completely
analogous to the proofs of Theorem 9 in P\"{o}tscher and Leeb (2009) and
Theorem 5 in P\"{o}tscher and Schneider (2009), respectively.

\begin{proposition}
\label{oracle_H}(Hard-thresholding) Suppose that for given $i\geq 1$
satisfying $i\leq k=k(n)$ for large enough $n$ we have $\xi _{i,n}\eta
_{i,n}\rightarrow 0$ and $n^{1/2}\eta _{i,n}\rightarrow \infty $. Suppose
that the true parameters $\theta ^{(n)}=(\theta _{1,n},\ldots ,\theta
_{k_{n},n})\in \mathbb{R}^{k_{n}}$ and $\sigma _{n}\in (0,\infty )$ satisfy $%
n^{1/2}\theta _{i,n}/(\sigma _{n}\xi _{i,n})\rightarrow \nu _{i}\in \mathbb{%
\bar{R}}$ and $\theta _{i,n}/(\sigma _{n}\xi _{i,n}\eta _{i,n})\rightarrow
\zeta _{i}\in \mathbb{\bar{R}}$. [Note that in case $\zeta _{i}\neq 0$ the
convergence of $n^{1/2}\theta _{i,n}/(\sigma _{n}\xi _{i,n})$ already
follows from that of $\theta _{i,n}/(\sigma _{n}\xi _{i,n}\eta _{i,n})$, and 
$\nu _{i}$ is then given by $\limfunc{sign}(\zeta _{i})\infty $.]

1. Suppose $\left\vert \zeta _{i}\right\vert <1$. Then $F_{H,n,\theta
^{(n)},\sigma _{n}}^{i}$converges weakly to $\delta _{-\nu _{i}}$ if $%
\left\vert \nu _{i}\right\vert <\infty $; if $\left\vert \nu _{i}\right\vert
=\infty $ the total mass of $F_{H,n,\theta ^{(n)},\sigma _{n}}^{i}$ escapes
to $-\nu _{i}$, in the sense that $F_{H,n,\theta ^{(n)},\sigma
_{n}}^{i}(x)\rightarrow 0$ for every $x\in \mathbb{R}$ if $\nu _{i}=-\infty $%
, and that $F_{H,n,\theta ^{(n)},\sigma _{n}}^{i}(x)\rightarrow 1$ for every 
$x\in \mathbb{R}$ if $\nu _{i}=\infty $.

2. Suppose $\left\vert \zeta _{i}\right\vert >1$. Then $F_{H,n,\theta
^{(n)},\sigma _{n}}^{i}$converges weakly to $\Phi $.

3. Suppose $\left\vert \zeta _{i}\right\vert =1$ and $n^{1/2}\left( \eta
_{i,n}-\zeta _{i}\theta _{i,n}/(\sigma _{n}\xi _{i,n})\right) \rightarrow
r_{i}$ for some $r_{i}\in \mathbb{\bar{R}}$. Then\linebreak\ $F_{H,n,\theta
^{(n)},\sigma _{n}}^{i}(x)$ converges to%
\begin{equation*}
\Phi (r_{i})\boldsymbol{1}\left( \zeta _{i}=1\right) +\int_{-\infty
}^{x}\phi (t)\boldsymbol{1}\left( \zeta _{i}t>r_{i}\right) dt
\end{equation*}%
for every $x\in \mathbb{R}$. [In case $r_{i}=-\infty $ the limit reduces to
a standard normal distribution.]
\end{proposition}

\begin{proposition}
\label{oracle_AS}(Adaptive soft-thresholding) Suppose that for given $i\geq
1 $ satisfying $i\leq k=k(n)$ for large enough $n$ we have $\xi _{i,n}\eta
_{i,n}\rightarrow 0$ and $n^{1/2}\eta _{i,n}\rightarrow \infty $. Suppose
that the true parameters $\theta ^{(n)}=(\theta _{1,n},\ldots ,\theta
_{k_{n},n})\in \mathbb{R}^{k_{n}}$ and $\sigma _{n}\in (0,\infty )$ satisfy $%
\theta _{i,n}/(\sigma _{n}\xi _{i,n}\eta _{i,n})\rightarrow \zeta _{i}\in 
\mathbb{\bar{R}}$.

1. If $\zeta _{i}=0$ and $n^{1/2}\theta _{i,n}/(\sigma _{n}\xi
_{i,n})\rightarrow \nu _{i}\in \mathbb{R}$, then $F_{AS,n,\theta
^{(n)},\sigma _{n}}^{i}$converges weakly to $\delta _{-\nu _{i}}$.

2. The total mass of $F_{AS,n,\theta ^{(n)},\sigma _{n}}^{i}$ escapes to $%
\infty $ or $-\infty $ in the following cases: If $-\infty <\zeta _{i}<0$,
or if $\zeta _{i}=0$ and $n^{1/2}\theta _{i,n}/(\sigma _{n}\xi
_{i,n})\rightarrow -\infty $, or if $\zeta _{i}=-\infty $ and $n^{1/2}\eta
_{i,n}^{2}\xi _{i,n}\theta _{i,n}^{-1}\sigma _{n}\rightarrow -\infty $, then 
$F_{AS,n,\theta ^{(n)},\sigma _{n}}^{i}(x)\rightarrow 0$ for every $x\in 
\mathbb{R}$. If $0<\zeta _{i}<\infty $, or if $\zeta _{i}=0$ and $%
n^{1/2}\theta _{i,n}/(\sigma _{n}\xi _{i,n})\rightarrow \infty $, or if $%
\zeta _{i}=\infty $ and $n^{1/2}\eta _{i,n}^{2}\xi _{i,n}\theta
_{i,n}^{-1}\sigma _{n}\rightarrow \infty $, then $F_{AS,n,\theta
^{(n)},\sigma _{n}}^{i}(x)\rightarrow 1$ for every $x\in \mathbb{R}$.

3. If $\left\vert \zeta _{i}\right\vert =\infty $ and $n^{1/2}\eta
_{i,n}^{2}\xi _{i,n}\theta _{i,n}^{-1}\sigma _{n}\rightarrow w_{i}\in 
\mathbb{R}$, then $F_{AS,n,\theta ^{(n)},\sigma _{n}}^{i}$converges weakly
to $\Phi (\cdot +w_{i})$.
\end{proposition}

It is easy to see that setting $\theta _{i,n}\equiv \theta _{i}$ and $\sigma
_{n}\equiv \sigma $ in Proposition \ref{oracle_H} immediately recovers the
"oracle-property" for $\hat{\theta}_{H,i}$. Similarly, we recover the
"oracle property" for $\hat{\theta}_{AS,i}$ from Proposition \ref{oracle_AS}
provided $n^{1/4}\xi _{i,n}^{1/2}\eta _{i,n}\rightarrow 0$. The propositions
also characterize the sequences of parameters along which the mass of the
distributions of the hard-thresholding and the adaptive soft-thresholding
estimator escapes to infinity; loosely speaking these are sequences along
which the bias of the estimators exceeds all bounds.

The theorems in Section \ref{uniform_close} also show that the last two
propositions above carry over immediately to the unknown-variance case
whenever $n-k\rightarrow \infty $ sufficiently fast such that $n^{1/2}\eta
_{i,n}(n-k)^{-1/2}\rightarrow 0$ holds. To save space, we do not extend
these two propositions to the case where the latter condition fails to hold.

\subsubsection{Soft-Thresholding}

The situation is somewhat different for the soft-thresholding estimator. It
follows from Theorem \ref{STconsistent} that the distribution of $\sigma
^{-1}(\xi _{i,n}\eta _{i,n})^{-1}\left( \tilde{\theta}_{S,i}-\theta
_{i}\right) $ does not degenerate to pointmass at zero (in fact, has no mass
at zero) if $\theta _{i}\neq 0$ and is held fixed. Consequently, $(\xi
_{i,n}\eta _{i,n})^{-1}$ is also the fixed-parameter convergence rate of $%
\tilde{\theta}_{S,i}$, in the sense that scaling with a faster rate (e.g., $%
n^{1/2}\xi _{i,n}^{-1}$) leads to the escape of the total mass of the
finite-sample distribution of the so-scaled (and centered) estimator to $-%
\limfunc{sign}(\theta _{i})\infty $. For $\theta _{i}=0$ we get with the
same argument as for hard-thresholding that $\sigma ^{-1}n^{1/2}\xi
_{i,n}^{-1}\left( \tilde{\theta}_{S,i}-\theta _{i}\right) $ converges to $%
\delta _{0}$. For the infeasible version $\hat{\theta}_{S,i}$ the situation
is identical. We conclude by a result analogous to Propositions \ref%
{oracle_H} and \ref{oracle_AS}. The proof of this result is completely
analogous to the proof of Theorem 10 in P\"{o}tscher and Leeb (2009).

\begin{proposition}
(Soft-thresholding)\ Suppose that for given $i\geq 1$ satisfying $i\leq
k=k(n)$ for large enough $n$ we have $\xi _{i,n}\eta _{i,n}\rightarrow 0$
and $n^{1/2}\eta _{i,n}\rightarrow \infty $. Suppose that the true
parameters $\theta ^{(n)}=(\theta _{1,n},\ldots ,\theta _{k_{n},n})\in 
\mathbb{R}^{k_{n}}$ and $\sigma _{n}\in (0,\infty )$ satisfy $n^{1/2}\theta
_{i,n}/(\sigma _{n}\xi _{i,n})\rightarrow \nu _{i}\in \mathbb{\bar{R}}$.
Then $F_{S,n,\theta ^{(n)},\sigma _{n}}^{i}$converges weakly to $\delta
_{-\nu _{i}}$ if $\left\vert \nu _{i}\right\vert <\infty $; and if $%
\left\vert \nu _{i}\right\vert =\infty $, the total mass of $F_{S,n,\theta
^{(n)},\sigma _{n}}^{i}$ escapes to $-\nu _{i}$, in the sense that $%
F_{S,n,\theta ^{(n)},\sigma _{n}}^{i}(x)\rightarrow 0$ for every $x\in 
\mathbb{R}$ if $\nu _{i}=-\infty $, and that $F_{S,n,\theta ^{(n)},\sigma
_{n}}^{i}(x)\rightarrow 1$ for every $x\in \mathbb{R}$ if $\nu _{i}=\infty $.
\end{proposition}

Again, this proposition immediately extends to the unknown-variance case
whenever $n-k\rightarrow \infty $ sufficiently fast such that $n^{1/2}\eta
_{i,n}(n-k)^{-1/2}\rightarrow 0$ holds. We abstain from extending the result
to the case where the latter condition fails to hold.

\subsection{Remarks\label{remarks}}

\begin{remark}
\label{costfree2}\normalfont(i) The convergence conditions on the various
quantities involving $\theta _{i,n}$ and $\sigma _{n}$ (and on $n-k$) in the
propositions in Sections \ref{LSDKVC} and \ref{oracle} as well as in the
theorems in Section \ref{LSDUKVC} are essentially cost-free for the same
reason as explained in Remark \ref{costfree}.

(ii) We note that all possible forms of the moving-parameter limiting
distributions in the results in this section already arise for sequences $%
\theta _{i,n}$ belonging to an arbitrarily small neighborhood of zero (and
with $\sigma >0$ fixed). Consequently, the non-uniformity in the convergence
to the fixed-parameter limits is of a local nature.
\end{remark}

\begin{remark}
\normalfont P\"{o}tscher and Leeb (2009) and P\"{o}tscher and Schneider
(2009) present impossibility results for estimating the finite-sample
distribution of the thresholding estimators considered in these papers. In
the present context, corresponding impossibility results could be derived
under appropriate assumptions. We abstain from presenting such results.
\end{remark}

\section{Numerical Study\label{numstudy}}

As has been discussed in Remarks \ref{LASSO} and \ref{ALASSO} in Section \ref%
{model}, the soft-thresholding estimator coincides with the Lasso, and the
adaptive soft-thresholding estimator coincides with the adaptive Lasso in
case of orthogonal design. A natural question now is if the distributional
results for the (adaptive) soft-thresholding estimator derived in this paper
are in any way indicative for the distribution of the (adaptive) Lasso in
case of non-orthogonal design. In order to gain some insight into this we
provide a simulation study to compare the finite-sample distributions of the
respective estimators.

We simulate the Lasso estimator as defined in Remark \ref{LASSO} (with $\eta
_{i,n}^{\prime }=\eta _{i,n}\xi _{i,n}^{-1}$ and $\eta _{i,n}=\eta _{n}$ not
depending on $i$) and the adaptive Lasso estimator as defined in Remark \ref%
{ALASSO} (with $\eta _{i,n}^{\prime }=\eta _{n}$ not depending on $i$) and
show histograms of $n^{1/2}\sigma ^{-1}\xi _{i,n}^{-1}\left( \bar{\theta}%
_{i}-\theta _{i}\right) $ where $\bar{\theta}_{i}$ stands for the $i$-th
component of Lasso or adaptive Lasso. [The scaling used here is chosen on
the basis that with this scaling the $i$-th component of the least-squares
estimator is standard normally distributed.]

We set $n=8$ and $k=4$, resulting in $n-k=4$ degrees of freedom. Two
different types of designs are considered: for Design I we use $X^{\prime
}X=n\Omega (\rho )$ with $\Omega (\rho )_{i,j}=\rho ^{|i-j|}$. More
concretely, $X$ is partitioned into $d=n/k=2$ blocks of size $k\times k$ and
each of these blocks is set equal to $k^{1/2}L$ with $LL^{\prime }=\Omega
(\rho )$, the Cholesky factorization of $\Omega (\rho )$. The value of $\rho 
$ is set equal to $0.3$, $0.5$, and $0.9$, implying condition numbers for $%
X^{\prime }X$ of $2.7$, $5.6$, and $57.0$, respectively. Design II is an
"equicorrelated" design. Here we set the matrix comprised of the first $k$
rows of $X$ equal to $I_{k}+cE_{k}$, where $E_{k}$ is the $k\times k$ matrix
with all components equal to $1$ and $c$ is a real number greater than $%
-1/k=-0.25$. The remaining entries of $X$ are all set equal to $0$. We
choose three values for $c$: first, $c=0.2$ which implies a correlation of $%
0.36$ between any two regressors and a condition number of $3.2$ for $%
X^{\prime }X$; second, $c=2$ which implies a correlation of $0.952$ and a
condition number of $81$; and $c=-0.2$ which implies a correlation of $-0.32$
and a condition number of $25$. For either type of design we proceed as
follows: For the given parameters $\theta =(3,1.5,0,0)^{\prime }$ and $%
\sigma =1$, we simulate $10,000$ data vectors $Y$ and compute the
corresponding estimator, i.e., the Lasso and adaptive Lasso as specified
above. We set $\eta _{n}=n^{-1/2}\Phi ^{-1}(0.975)$, implying that the
thresholding estimators delete a given irrelevant variable with probability $%
0.95$.

For the non-zero outcomes of the estimators, we plot the histogram of $%
n^{1/2}\sigma ^{-1}\xi _{i,n}^{-1}\left( \bar{\theta}_{i}-\theta _{i}\right) 
$ which is normalized such that its mass corresponds to the proportion of
the non-zero values. The zero values are accounted for by plotting
"pointmass" with height representing the proportion of zero values, i.e.,
the simulated variable selection probability. For the purpose of comparison
the graph of the distribution of the corresponding (centered and scaled)
thresholding estimator (using the same $\eta _{i,n}=\eta _{n}$) as derived
analytically in Section \ref{FS} is then superimposed in red color. The
results of the simulation study are presented in Figures 1-12 below.

In comparing the adaptive Lasso with the adaptive soft-thresholding
estimator, we find remarkable agreement between the respective marginal
distributions in all cases where the design matrix is not too
multicollinear, see Figures 1, 2, and 4. For the cases where the design
matrix is no longer well-conditioned a difference between the respective
marginal distributions emerges but seems to be surprisingly moderate, see
Figures 3, 5, and 6.

Turning to the Lasso and its thresholding counterpart, we find a similar
situation with a somewhat stronger disagreement between the respective
marginal distributions. Again in the cases where the design matrix is
well-conditioned (Figures 7, 8, and 10) the difference is less pronounced
than in the case of an ill-conditioned design matrix (Figures 9, 11, and 12).

We have also experimented with other values of $n$, $k$, $\theta $, $\rho $, 
$c$, and $\eta _{n}$ and have found the results to be qualitatively the same
for these choices.

\newpage


\begin{figure}[h]
\begin{center}
\includegraphics[width=0.95\textwidth]{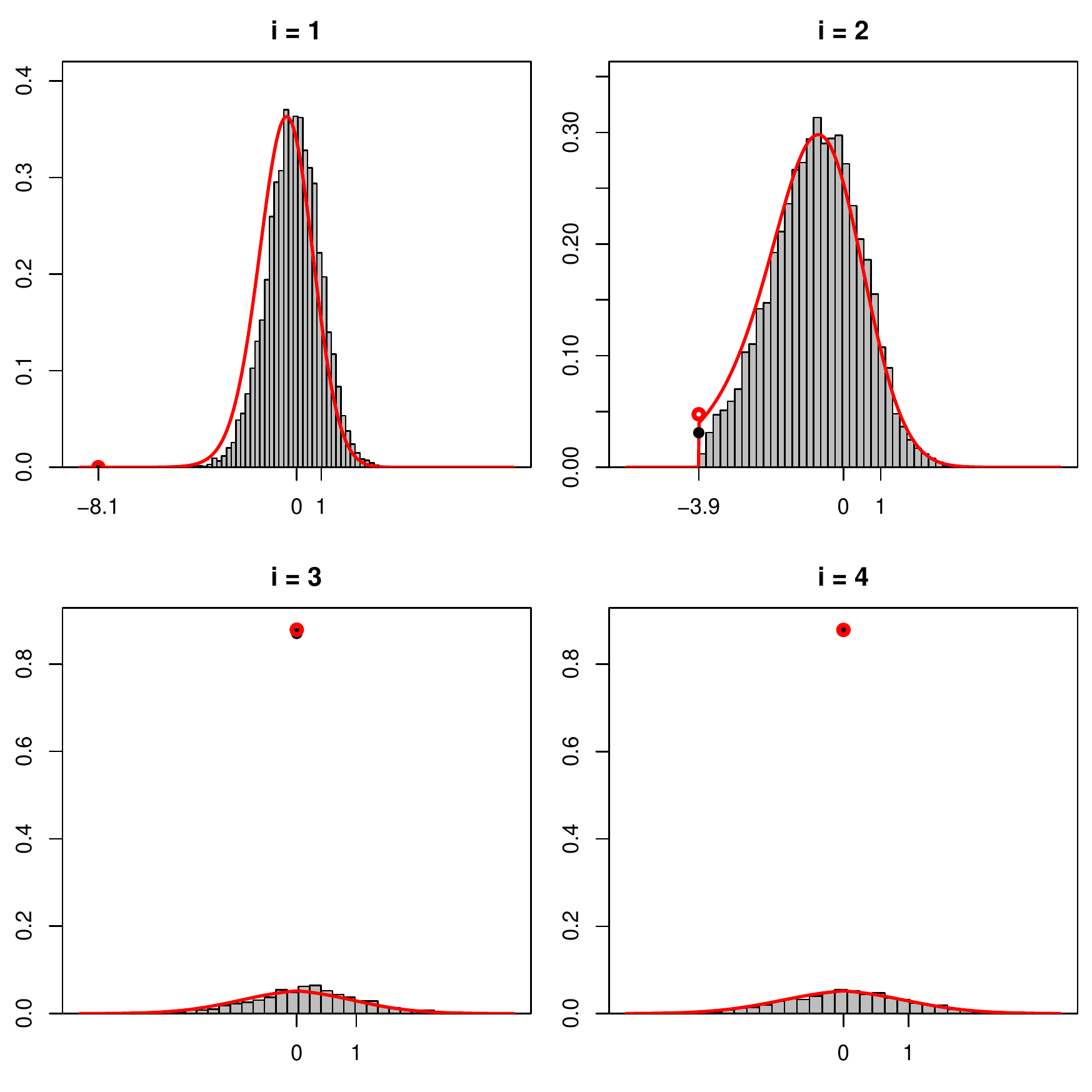}
\end{center}
\caption{Adaptive Lasso, Design I: $\rho = 0.3$}
\end{figure}

\pagebreak


\begin{figure}[h]
\begin{center}
\includegraphics[width=0.95\textwidth]{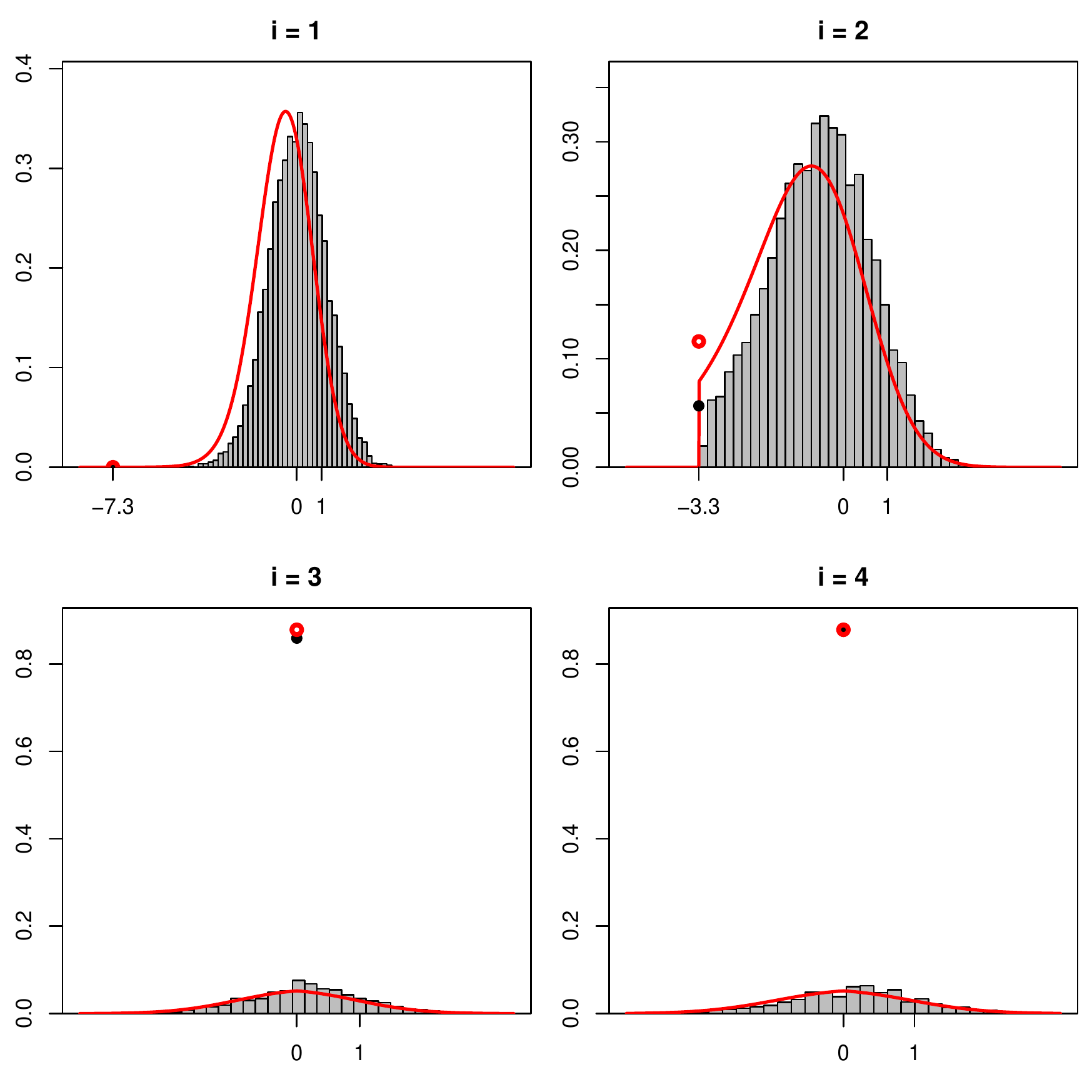}
\end{center}
\caption{Adaptive Lasso, Design I: $\protect\rho = 0.5$}
\end{figure}

\pagebreak


\begin{figure}[h]
\begin{center}
\includegraphics[width=0.95\textwidth]{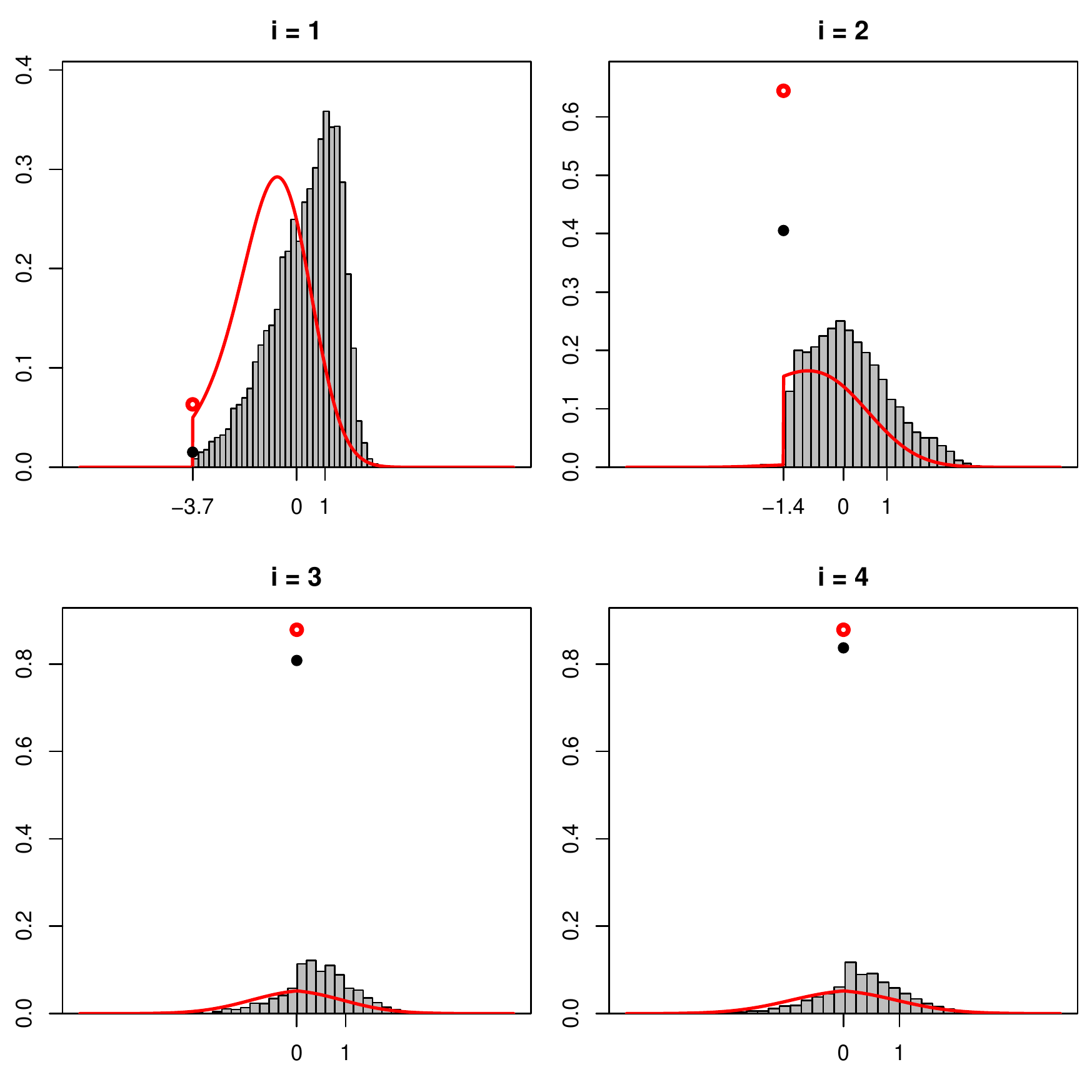}
\end{center}
\caption{Adaptive Lasso, Design I: $\protect\rho = 0.9$}
\end{figure}

\pagebreak


\begin{figure}[h]
\begin{center}
\includegraphics[width=0.95\textwidth]{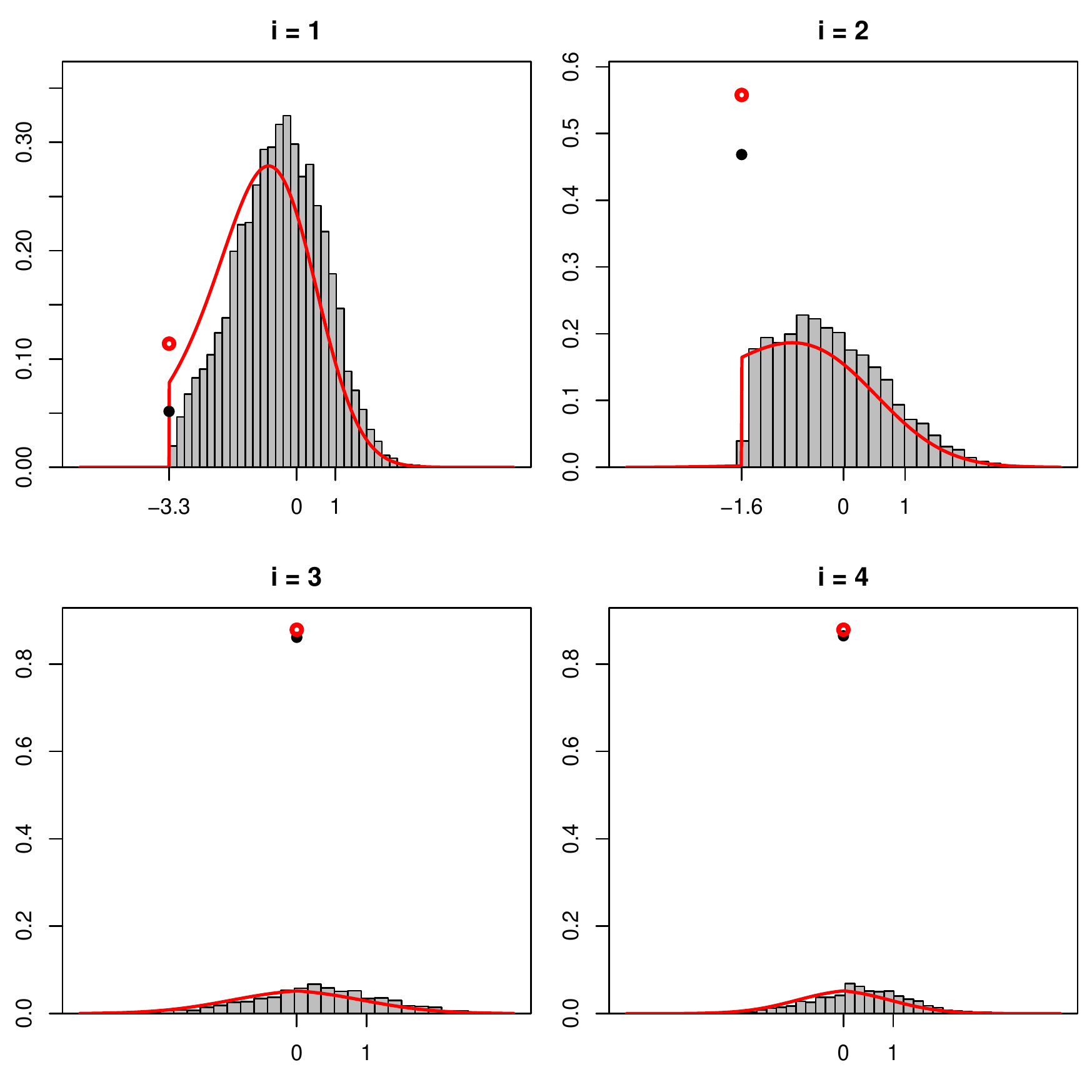}
\end{center}
\caption{Adaptive Lasso, Design II: $c = 0.2$}
\end{figure}

\pagebreak


\begin{figure}[h]
\begin{center}
\includegraphics[width=0.95\textwidth]{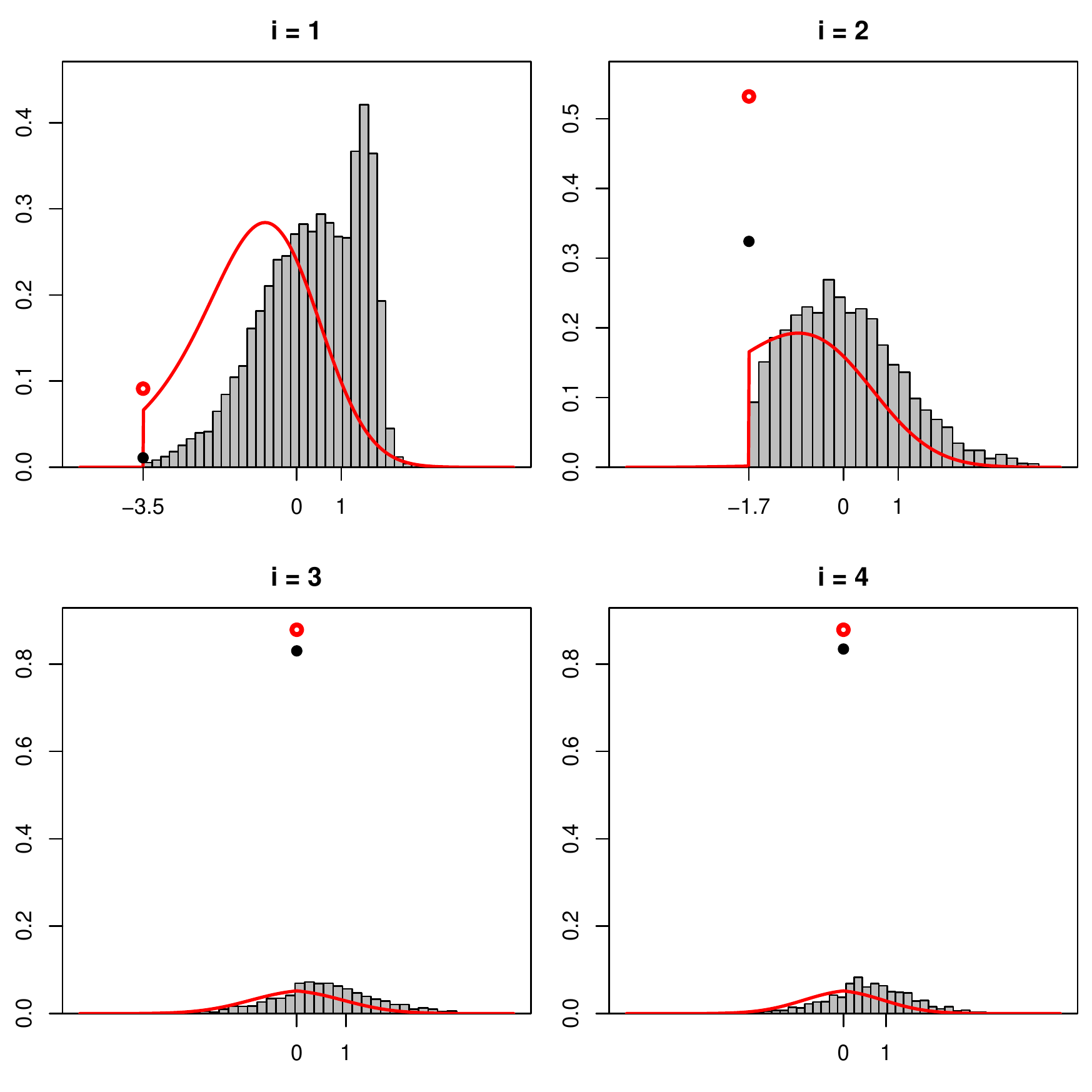}
\end{center}
\caption{Adaptive Lasso, Design II: $c = 2$}
\end{figure}

\pagebreak


\begin{figure}[h]
\begin{center}
\includegraphics[width=0.95\textwidth]{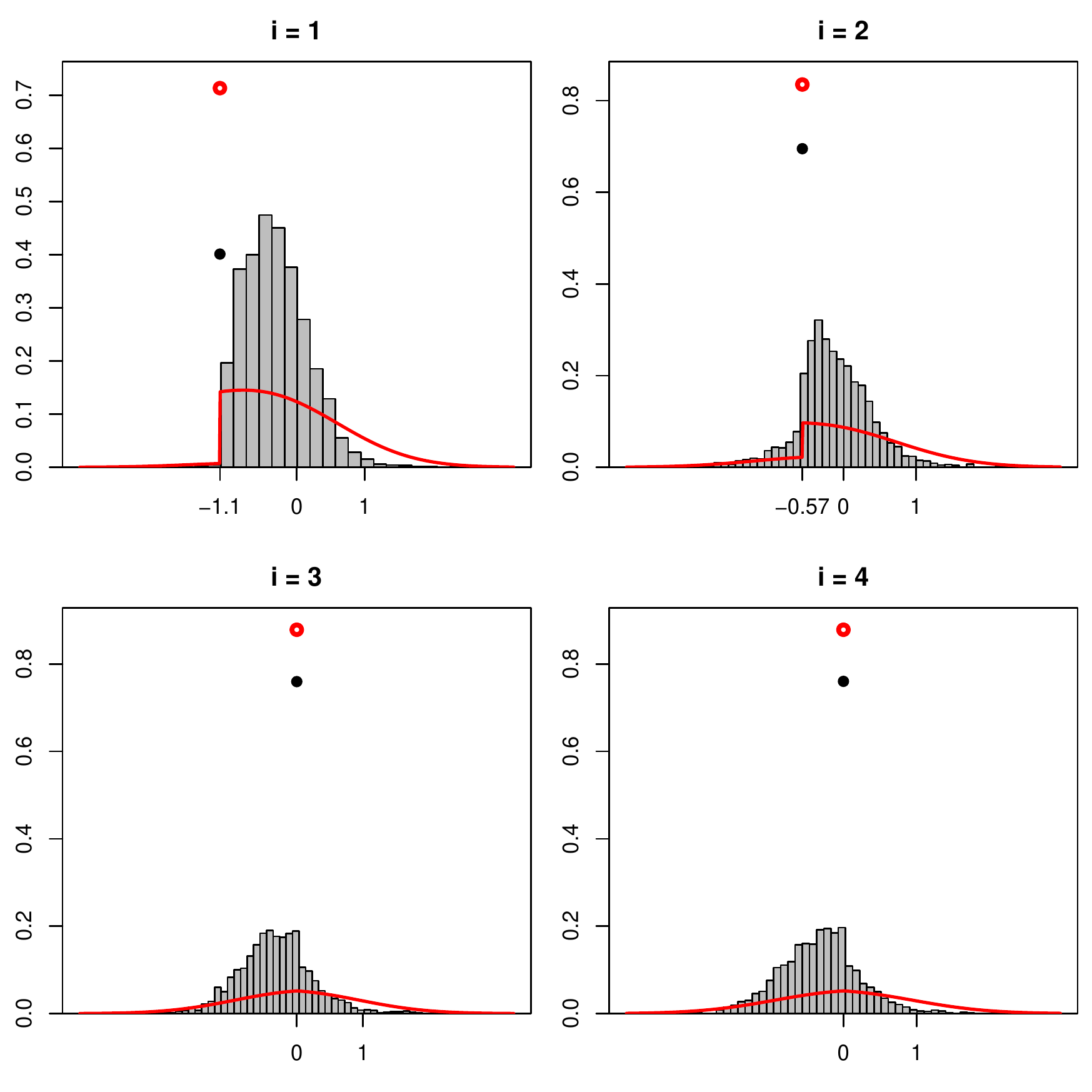}
\end{center}
\caption{Adaptive Lasso, Design II: $c = -0.2$}
\end{figure}

\pagebreak


\begin{figure}[h]
\begin{center}
\includegraphics[width=0.95\textwidth]{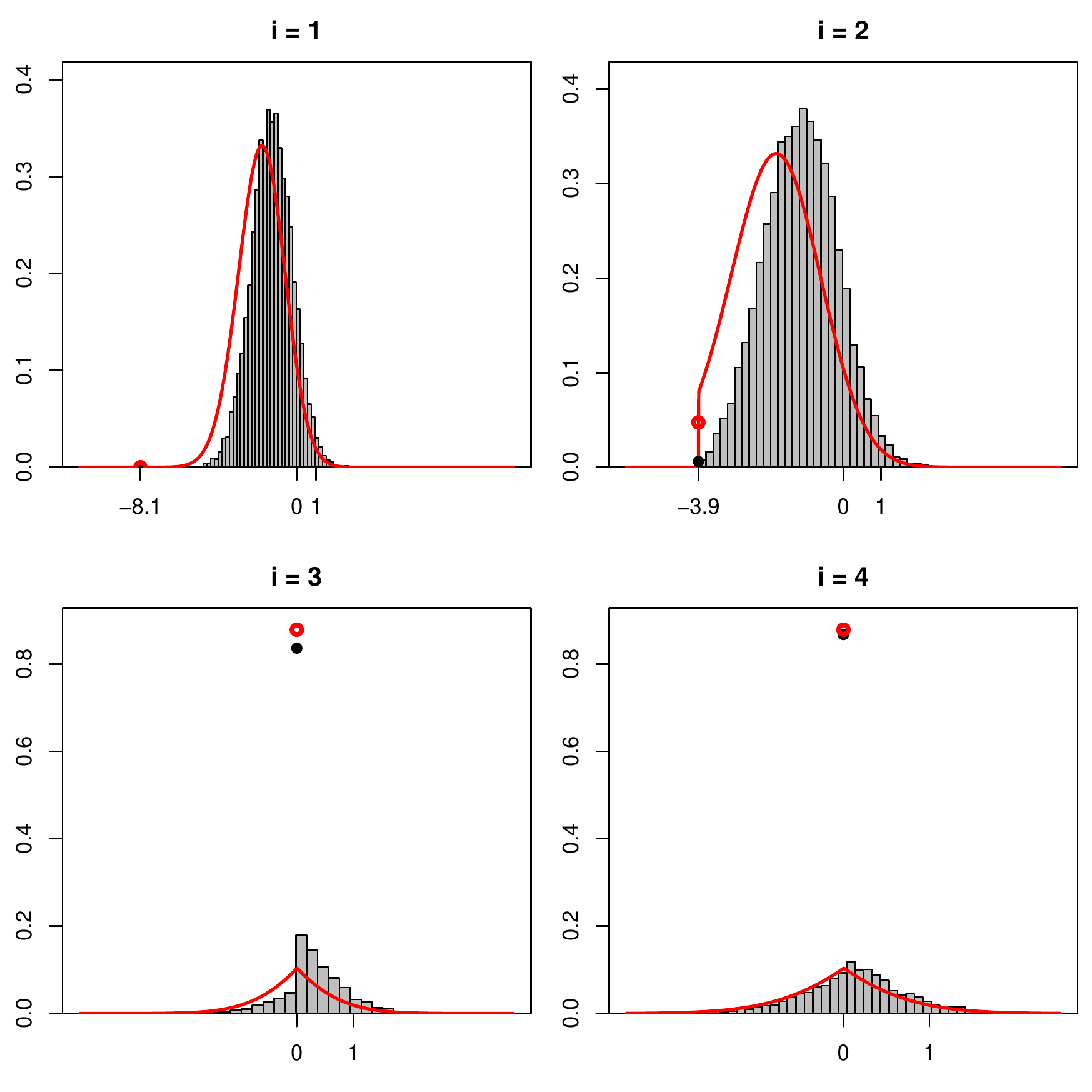}
\end{center}
\caption{Lasso, Design I: $\protect\rho = 0.3$}
\end{figure}

\pagebreak


\begin{figure}[h]
\begin{center}
\includegraphics[width=0.95\textwidth]{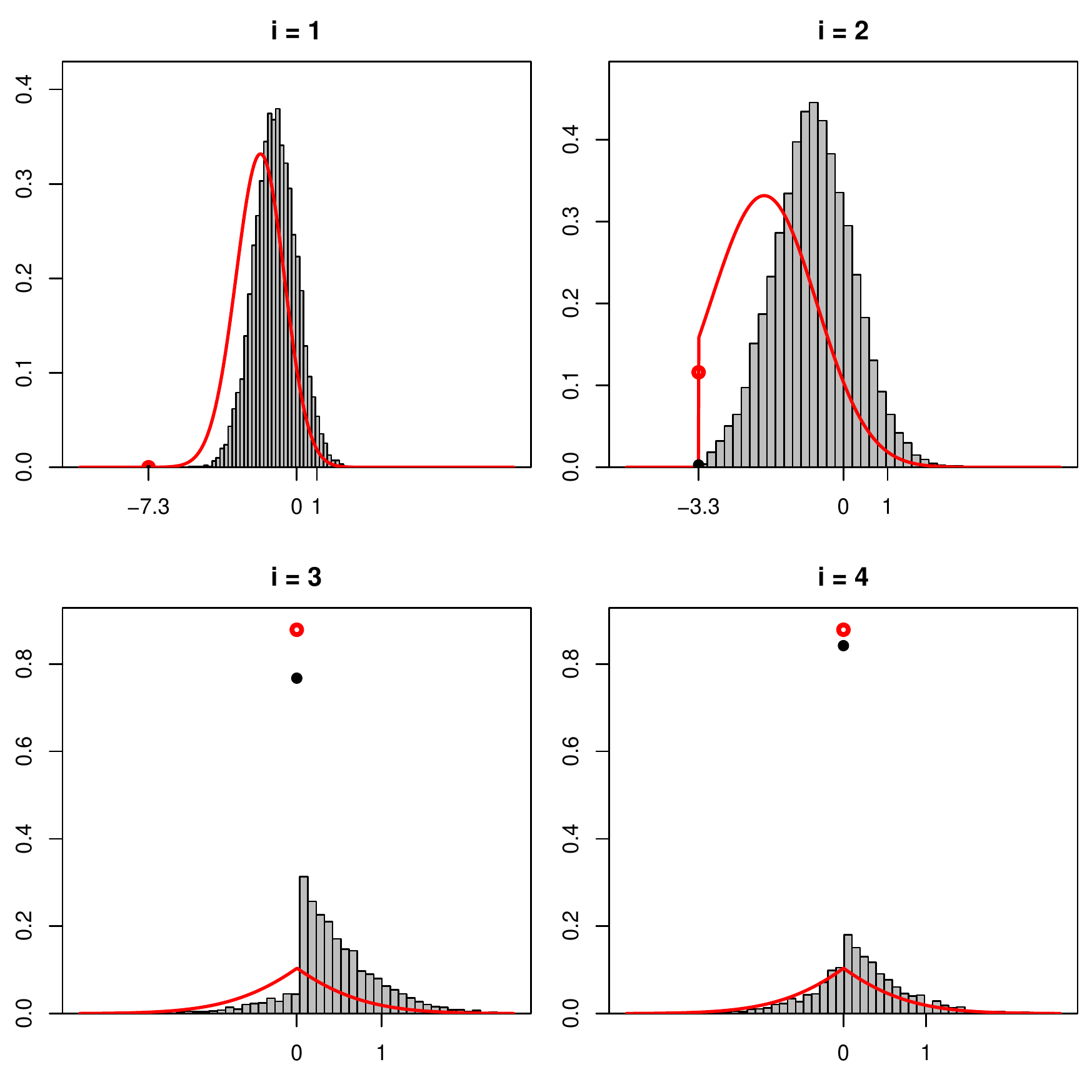}
\end{center}
\caption{Lasso, Design I: $\protect\rho = 0.5$}
\end{figure}

\pagebreak


\begin{figure}[h]
\begin{center}
\includegraphics[width=0.95\textwidth]{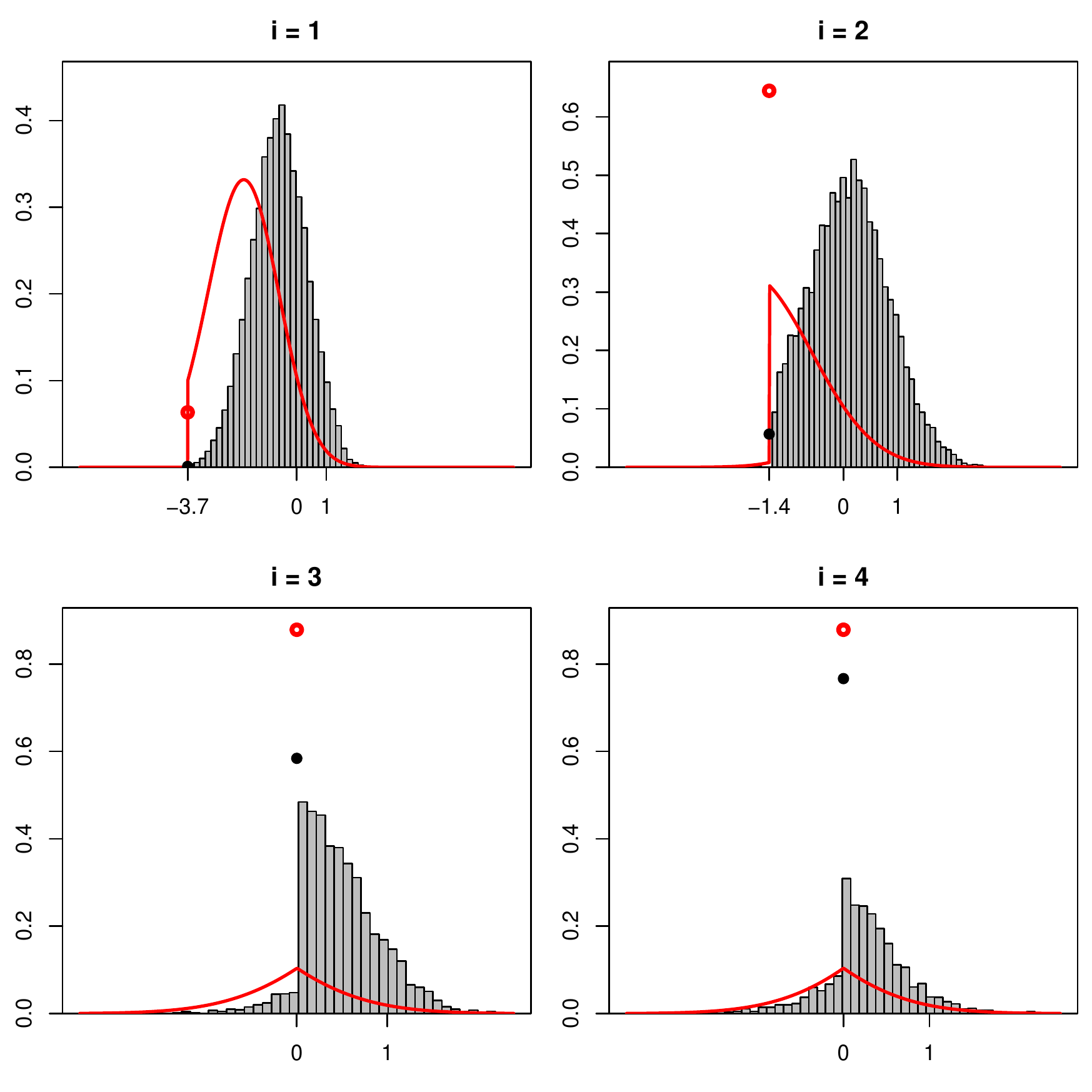}
\end{center}
\caption{Lasso, Design I: $\protect\rho = 0.9$}
\end{figure}

\pagebreak


\begin{figure}[h]
\begin{center}
\includegraphics[width=0.95\textwidth]{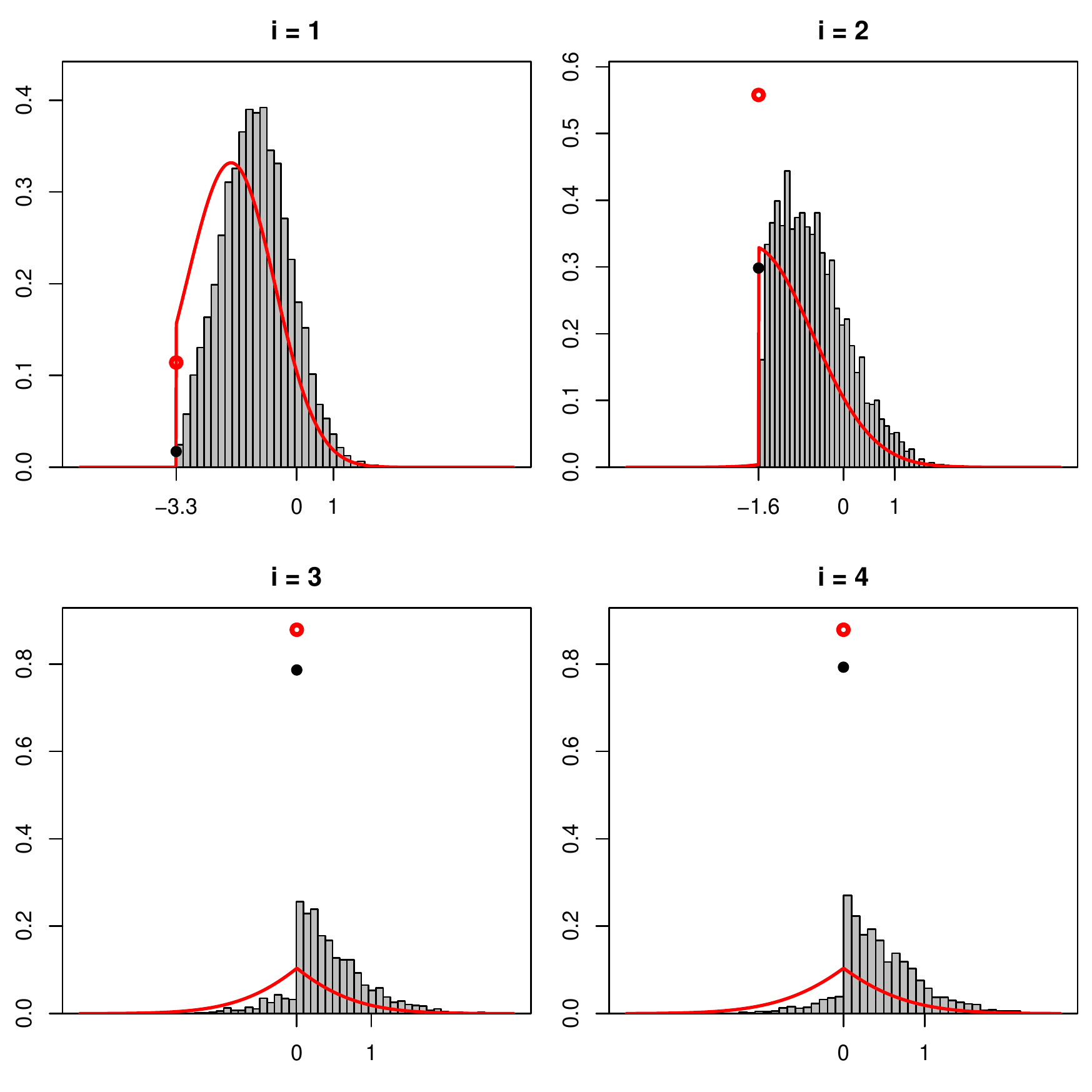}
\end{center}
\caption{Lasso, Design II: $c = 0.2$}
\end{figure}

\pagebreak


\begin{figure}[h]
\begin{center}
\includegraphics[width=0.95\textwidth]{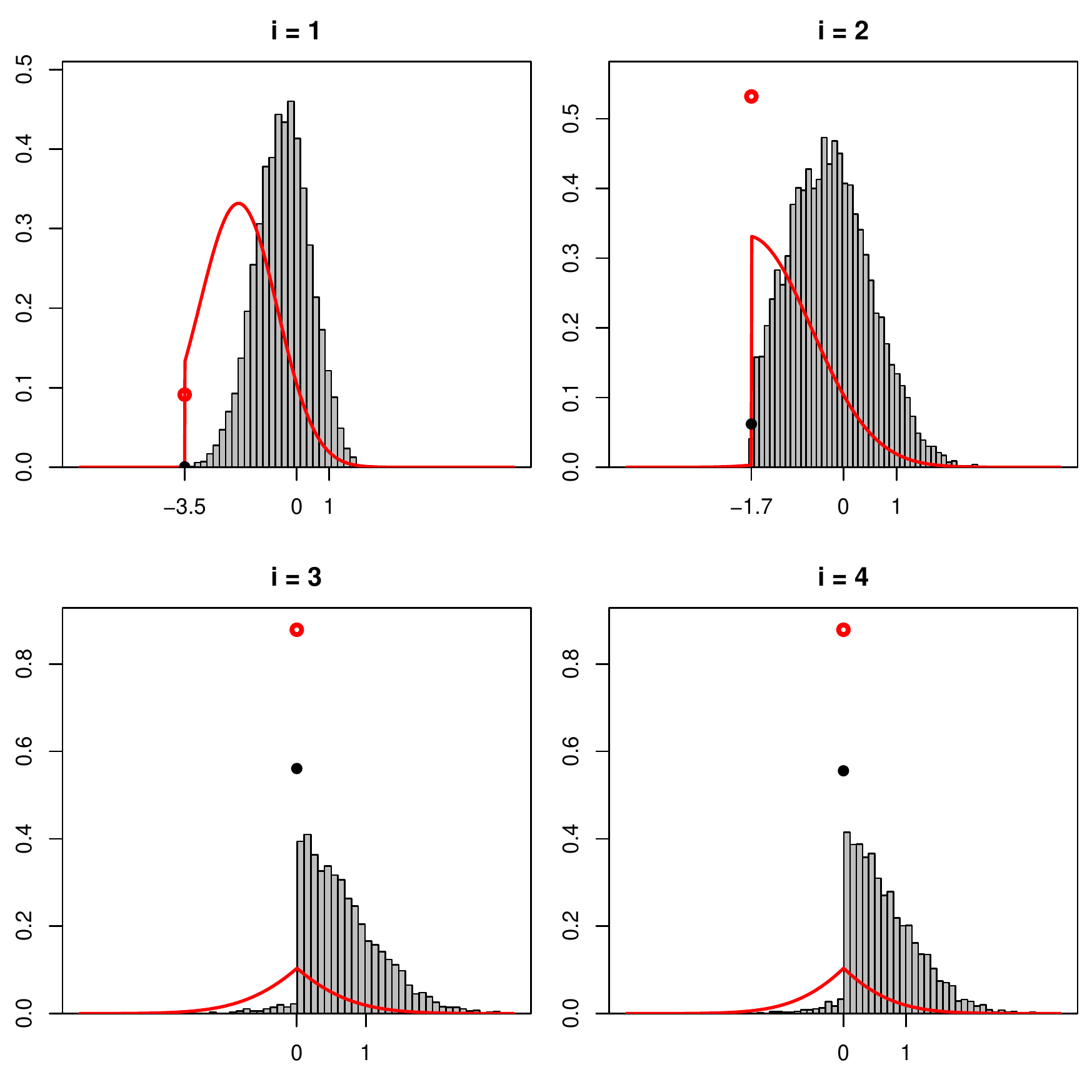}
\end{center}
\caption{Lasso, Design II: $c = 2$}
\end{figure}

\pagebreak


\begin{figure}[h]
\begin{center}
\includegraphics[width=0.95\textwidth]{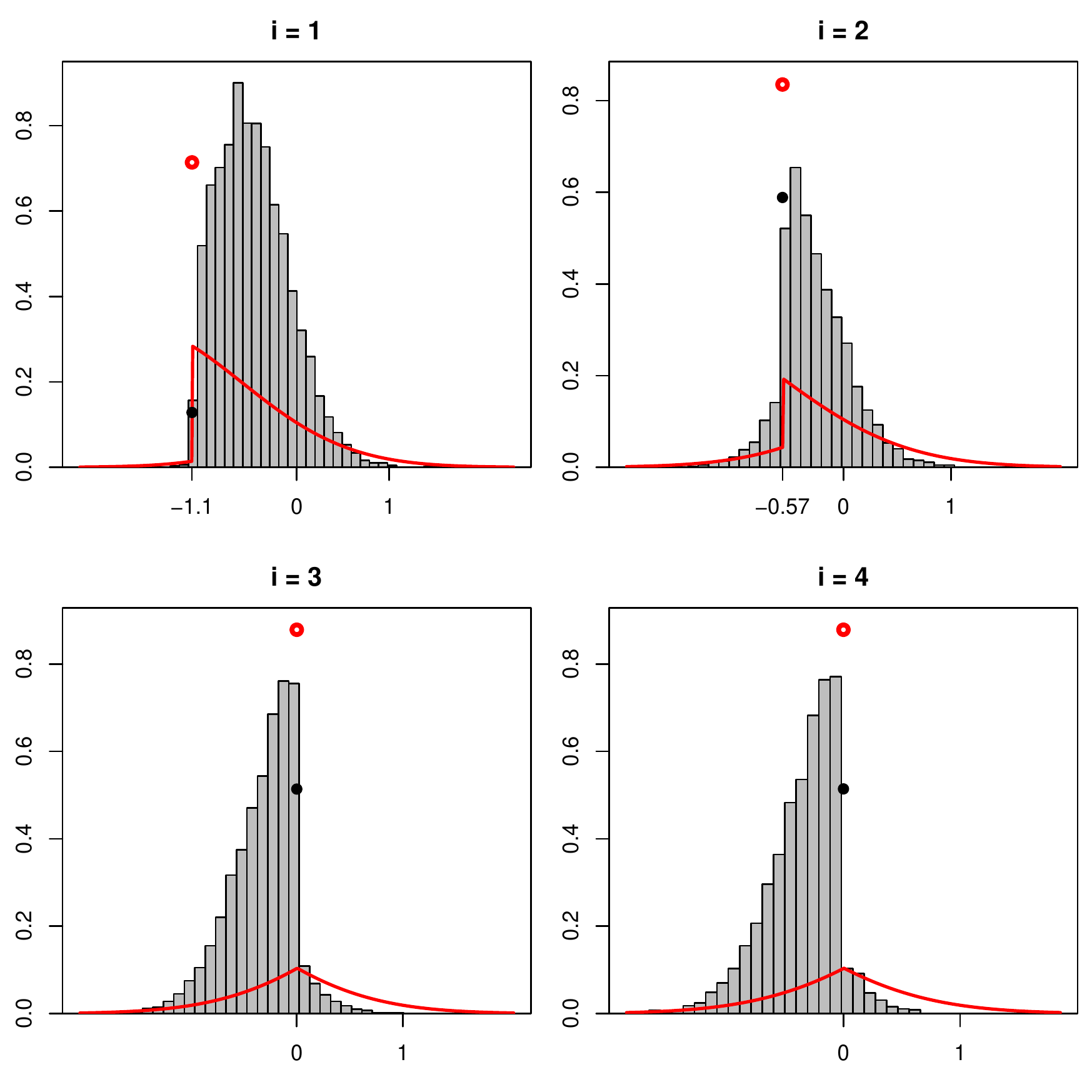}
\end{center}
\caption{Lasso, Design II: $c = -0.2$}
\end{figure}

\section{Proofs\label{prfs}}

\subsection{Proofs for Section \protect\ref{variable}}

\textbf{Proof of Proposition \ref{select_prob_pointwise}: }We first prove
Part (a). Rewrite $P_{n,\theta ,\sigma }\left( \hat{\theta}_{i}=0\right) $
as 
\begin{equation}
\Phi \left( n^{1/2}\xi _{i,n}^{-1}\left( -\theta _{i}/\sigma +\xi _{i,n}\eta
_{i,n}\right) \right) -\Phi \left( n^{1/2}\xi _{i,n}^{-1}\left( -\theta
_{i}/\sigma -\xi _{i,n}\eta _{i,n}\right) \right) .  \label{fi}
\end{equation}%
Assume first that $\xi _{i,n}\eta _{i,n}\rightarrow 0$ and fix $\theta
_{i}\neq 0$. By a standard subsequence argument we may assume without loss
of generality that $n^{1/2}\xi _{i,n}^{-1}$ converges to a constant $\kappa $
which by our maintained assumption (\ref{xi}) must satisfy $0<\kappa \leq
\infty $. Now $-\theta _{i}/\sigma \pm \xi _{i,n}\eta _{i,n}$ both converge
to $-\theta _{i}/\sigma $, which is non-zero, and consequently both
arguments in (\ref{fi}) converge to $-\kappa \theta _{i}/\sigma $. Since $%
\Phi $ is continuous on $\mathbb{\bar{R}}$, the expression (\ref{fi})
converges to zero. To prove the converse, now assume that (\ref{fi})
converges to zero for all $\theta _{i}\neq 0$. By a standard subsequence
argument, we may assume without loss of generality that $\xi _{i,n}\eta
_{i,n}$ converges to a constant $\varkappa $ satisfying $0\leq \varkappa
\leq \infty $. Suppose $\varkappa >0$ holds. Choose $\theta _{i}$ such that $%
0<-\theta _{i}/\sigma <\varkappa $ holds. It follows that $-\theta
_{i}/\sigma +\xi _{i,n}\eta _{i,n}$ and $-\theta _{i}/\sigma -\xi _{i,n}\eta
_{i,n}$ eventually have opposite signs and are bounded away from zero. By
our maintained assumption (\ref{xi}), the same is then true for the
arguments in (\ref{fi}) leading to a contradiction. Hence $\varkappa =0$
must hold, completing the proof of Part (a). Parts (b) and (c) are obvious
since $P_{n,\theta ,\sigma }\left( \hat{\theta}_{i}=0\right) =\Phi \left(
n^{1/2}\eta _{i,n}\right) $ $-$ $\Phi \left( -n^{1/2}\eta _{i,n}\right) $
whenever $\theta _{i}=0$. $\blacksquare $

\bigskip

\textbf{Proof of Proposition \ref{select_prob_moving_par}: }Part (a) follows
immediately from (\ref{select_prob}) and the assumptions. To prove Part (b)
we use (\ref{select_prob}) to write%
\begin{equation*}
P_{n,\theta ^{(n)},\sigma _{n}}\left( \hat{\theta}_{i}=0\right) =\Phi \left(
n^{1/2}\eta _{i,n}\left( 1-\theta _{i,n}/(\sigma _{n}\xi _{i,n}\eta
_{i,n})\right) \right) -\Phi \left( n^{1/2}\eta _{i,n}\left( -1-\theta
_{i,n}/(\sigma _{n}\xi _{i,n}\eta _{i,n})\right) \right) .
\end{equation*}%
The first and the second claim then follow immediately. For the third claim,
assume first that $\zeta _{i}=1$. Then%
\begin{align*}
P_{n,\theta ^{(n)},\sigma _{n}}\left( \hat{\theta}_{i}=0\right) & =\Phi
\left( n^{1/2}\left( \eta _{i,n}-\zeta _{i}\theta _{i,n}/(\sigma _{n}\xi
_{i,n})\right) \right) \\
& -\Phi \left( n^{1/2}\eta _{i,n}\left( -1-\theta _{i,n}/(\sigma _{n}\xi
_{i,n}\eta _{i,n})\right) \right) \rightarrow \Phi (r_{i}).
\end{align*}%
The case $\zeta _{i}=-1$ is handled analogously. $\blacksquare $

\bigskip

\textbf{Proof of Proposition \ref{select_prob_pointwise_unknown}: }We prove
Part (b) first. Observe that 
\begin{align*}
P_{n,\theta ,\sigma }\left( \tilde{\theta}_{i}=0\right) & =\int_{0}^{\infty }%
\left[ \Phi \left( n^{1/2}s\eta _{i,n}\right) -\Phi \left( -n^{1/2}s\eta
_{i,n}\right) \right] \rho _{n-k}(s)ds \\
& =T_{n-k}\left( n^{1/2}\eta _{i,n}\right) -T_{n-k}\left( -n^{1/2}\eta
_{i,n}\right) .
\end{align*}%
By a subsequence argument it suffices to prove the result under the
assumption that $n-k=n-k(n)$ converges in $\mathbb{N}\cup \{\infty \}$. If
the limit is finite, then $n-k(n)$ is eventually constant and the result
follows since every $t$-distribution has unbounded support. If $%
n-k\rightarrow \infty $ then 
\begin{eqnarray*}
&&\Phi \left( n^{1/2}\eta _{i,n}\right) -\Phi \left( -n^{1/2}\eta
_{i,n}\right) -2\left\Vert T_{n-k}-\Phi \right\Vert _{\infty } \\
&\leq &P_{n,\theta ,\sigma }\left( \tilde{\theta}_{i}=0\right) \leq \Phi
\left( n^{1/2}\eta _{i,n}\right) -\Phi \left( -n^{1/2}\eta _{i,n}\right)
+2\left\Vert T_{n-k}-\Phi \right\Vert _{\infty },
\end{eqnarray*}%
where $\left\Vert \cdot \right\Vert _{\infty }$ denotes the supremum norm.
Since $\left\Vert T_{n-k}-\Phi \right\Vert _{\infty }\rightarrow 0$ if $%
n-k\rightarrow \infty $ by Polya's Theorem, the result follows. Part (c) is
proved analogously.

We next prove Part (a). Observe that the collection of distributions
corresponding to $\left\{ \rho _{m}:m\in \mathbb{N}\right\} $ is tight on $%
(0,\infty )$, meaning that for every $0<\delta <1$ there exist $0<c_{\ast
}(\delta )<c^{\ast }(\delta )<\infty $ such that $\sup_{m\in \mathbb{N}%
}\int_{0}^{c_{\ast }(\delta )}\rho _{m}ds<\delta $ and $\sup_{m\in \mathbb{N}%
}\int_{c^{\ast }(\delta )}^{\infty }\rho _{m}ds<\delta $. Note that the map $%
s\mapsto P_{n,\theta ,\sigma }\left( \hat{\theta}_{i}(s\eta _{i,n})=0\right) 
$ is monotonically nondecreasing. Hence,%
\begin{eqnarray*}
\left( 1-\delta \right) P_{n,\theta ,\sigma }\left( \hat{\theta}_{i}(c_{\ast
}(\delta )\eta _{i,n})=0\right) &\leq &\int_{c_{\ast }(\delta )}^{\infty
}P_{n,\theta ,\sigma }\left( \hat{\theta}_{i}(s\eta _{i,n})=0\right) \rho
_{n-k}(s)ds \\
&\leq &P_{n,\theta ,\sigma }\left( \tilde{\theta}_{i}=0\right)
=\int_{0}^{\infty }P_{n,\theta ,\sigma }\left( \hat{\theta}_{i}(s\eta
_{i,n})=0\right) \rho _{n-k}(s)ds \\
&\leq &P_{n,\theta ,\sigma }\left( \hat{\theta}_{i}(c^{\ast }(\delta )\eta
_{i,n})=0\right) +\delta .
\end{eqnarray*}%
Since $\xi _{i,n}c_{\ast }(\delta )\eta _{i,n}$ ($\xi _{i,n}c^{\ast }(\delta
)\eta _{i,n}$, respectively) converges to zero if and only if $\xi
_{i,n}\eta _{i,n}$ does so, Part (a) follows from Proposition \ref%
{select_prob_pointwise} applied to the estimators $\hat{\theta}_{i}(c_{\ast
}(\delta )\eta _{i,n})$and $\hat{\theta}_{i}(c^{\ast }(\delta )\eta _{i,n})$%
. $\blacksquare $

\bigskip

\textbf{Proof of Theorem \ref{select_prob_moving_par_unknown}: }(a) Set $%
P_{n}(s)=P_{n,\theta ^{(n)},\sigma _{n}}\left( \hat{\theta}_{i}(s\eta
_{i,n})=0\right) $ for $s>0$. By Proposition \ref{select_prob_moving_par} we
have that $P_{n}(s)$ converges to $P(s)$ for all $s>0$, where $P(s)=\Phi
\left( -\nu _{i}+se_{i}\right) -\Phi \left( -\nu _{i}-se_{i}\right) $ for $%
s>0$. Since $P_{n}(s)$ as well as $P(s)$ are continuous functions of $s$,
are monotonically nondecreasing in $s$, and have the property that their
limits for $s\rightarrow 0$ are $0$ while the limits for $s\rightarrow
\infty $ are $1$, it follows from Polya's Theorem that the convergence is
uniform in $s$. But then using (\ref{select_prob_unknown}) gives%
\begin{eqnarray*}
&&\left\vert P_{n,\theta ^{(n)},\sigma _{n}}\left( \tilde{\theta}%
_{i}=0\right) -\int_{0}^{\infty }\left( \Phi \left( -\nu _{i}+se_{i}\right)
-\Phi \left( -\nu _{i}-se_{i}\right) \right) \rho _{n-k}(s)ds\right\vert \\
&\leq &\sup_{s>0}\left\vert P_{n}(s)-P(s)\right\vert \int_{0}^{\infty }\rho
_{n-k}(s)ds=\sup_{s>0}\left\vert P_{n}(s)-P(s)\right\vert \rightarrow 0
\end{eqnarray*}%
as $n\rightarrow \infty $. This completes the proof in case $n-k=m$
eventually; in case $n-k\rightarrow \infty $ observe that $\int_{0}^{\infty
}\left( \Phi \left( -\nu _{i}+se_{i}\right) -\Phi \left( -\nu
_{i}-se_{i}\right) \right) \rho _{n-k}(s)ds$ then converges to $\Phi \left(
-\nu _{i}+e_{i}\right) -\Phi \left( -\nu _{i}-e_{i}\right) $ as the
distribution corresponding to $\rho _{n-k}$ converges weakly to pointmass at 
$s=1$ and the integrand is bounded and continuous.

(b) Observe that $P_{n,\theta ^{(n)},\sigma _{n}}\left( \hat{\theta}%
_{i}(s\eta _{i,n})=0\right) $ converges to $1$ for $s>\left\vert \zeta
_{i}\right\vert $ and to $0$ for $s<\left\vert \zeta _{i}\right\vert $ by
Proposition \ref{select_prob_moving_par} applied to the estimator $\hat{%
\theta}_{i}(s\eta _{i,n})$. Now (\ref{select_prob_unknown}) and dominated
convergence deliver the result in (b1).

Next consider (b2): Suppose first that $\left\vert \zeta _{i}\right\vert <1$%
. Choose $\varepsilon >0$ small enough such that $\left\vert \zeta
_{i}\right\vert +\varepsilon <1$. Then, recalling that $P_{n,\theta
^{(n)},\sigma _{n}}\left( \hat{\theta}_{i}(s\eta _{i,n})=0\right) $ is
monotonically nondecreasing in $s$, eq.~(\ref{select_prob_unknown}) gives%
\begin{eqnarray*}
P_{n,\theta ^{(n)},\sigma _{n}}\left( \tilde{\theta}_{i}=0\right) &\geq
&\int_{\left\vert \zeta _{i}\right\vert +\varepsilon }^{\infty }P_{n,\theta
^{(n)},\sigma _{n}}\left( \hat{\theta}_{i}(s\eta _{i,n})=0\right) \rho
_{n-k}(s)ds \\
&\geq &P_{n,\theta ^{(n)},\sigma _{n}}\left( \hat{\theta}_{i}(\left(
\left\vert \zeta _{i}\right\vert +\varepsilon \right) \eta _{i,n})=0\right)
\int_{\left\vert \zeta _{i}\right\vert +\varepsilon }^{\infty }\rho
_{n-k}(s)ds.
\end{eqnarray*}%
Now the integral on the r.h.s.~converges to $1$ since $\left\vert \zeta
_{i}\right\vert +\varepsilon <1$, and the probability on the
r.h.s.~converges to $1$ by Proposition \ref{select_prob_moving_par} applied
to the estimator $\hat{\theta}_{i}(\left( \left\vert \zeta _{i}\right\vert
+\varepsilon \right) \eta _{i,n})$. This completes the proof for the case $%
\left\vert \zeta _{i}\right\vert <1$. Next assume that $\left\vert \zeta
_{i}\right\vert >1$. Choose $\varepsilon >0$ small enough such that $%
\left\vert \zeta _{i}\right\vert -\varepsilon >1$ holds. Then from (\ref%
{select_prob_unknown}) we have%
\begin{eqnarray*}
P_{n,\theta ^{(n)},\sigma _{n}}\left( \tilde{\theta}_{i}=0\right) &\leq
&\int_{0}^{\left\vert \zeta _{i}\right\vert -\varepsilon }P_{n,\theta
^{(n)},\sigma _{n}}\left( \hat{\theta}_{i}(s\eta _{i,n})=0\right) \rho
_{n-k}(s)ds+\int_{\left\vert \zeta _{i}\right\vert -\varepsilon }^{\infty
}\rho _{n-k}(s)ds \\
&\leq &P_{n,\theta ^{(n)},\sigma _{n}}\left( \hat{\theta}_{i}(\left(
\left\vert \zeta _{i}\right\vert -\varepsilon \right) \eta _{i,n})=0\right)
+\int_{\left\vert \zeta _{i}\right\vert -\varepsilon }^{\infty }\rho
_{n-k}(s)ds
\end{eqnarray*}%
since $P_{n}(s)$ is monotonically nondecreasing in $s$ and $%
\int_{0}^{\left\vert \zeta _{i}\right\vert -\varepsilon }\rho _{n-k}(s)ds$
is not larger than $1$. Since $\left\vert \zeta _{i}\right\vert -\varepsilon
>1$ and $n-k\rightarrow \infty $ the second term on the r.h.s.~goes to zero,
while the first term goes to zero by Proposition \ref{select_prob_moving_par}
applied to the estimator $\hat{\theta}_{i}(\left( \left\vert \zeta
_{i}\right\vert -\varepsilon \right) \eta _{i,n})$.

Next we prove 3.\&4.~and assume $\zeta _{i}=1$ first. Then using eq.~(\ref%
{select_prob_unknown}) and performing the substitution $s-1=\left( 2\left(
n-k\right) \right) ^{-1/2}t$ we obtain (recalling that $\rho _{n-k}$ is zero
for negative arguments and using the abbreviations $r_{i,n}=n^{1/2}\left(
\eta _{i,n}-\theta _{i,n}/(\sigma _{n}\xi _{i,n})\right) $ and $%
r_{i,n}^{\ast }=n^{1/2}\left( -\eta _{i,n}-\theta _{i,n}/(\sigma _{n}\xi
_{i,n})\right) $)%
\begin{eqnarray*}
&&P_{n,\theta ^{(n)},\sigma _{n}}\left( \tilde{\theta}_{i}=0\right) \\
&=&\int_{-\infty }^{\infty }\left[ \Phi \left( r_{i,n}+n^{1/2}\eta
_{i,n}\left( 2\left( n-k\right) \right) ^{-1/2}t\right) -\Phi \left(
r_{i,n}^{\ast }-n^{1/2}\eta _{i,n}\left( 2\left( n-k\right) \right)
^{-1/2}t\right) \right] \\
&&\times \left( 2\left( n-k\right) \right) ^{-1/2}\rho _{n-k}(\left( 2\left(
n-k\right) \right) ^{-1/2}t+1)dt \\
&=&\int_{-\infty }^{\infty }\left[ \Phi \left( r_{i,n}+n^{1/2}\eta
_{i,n}\left( 2\left( n-k\right) \right) ^{-1/2}t\right) -\Phi \left(
r_{i,n}^{\ast }-n^{1/2}\eta _{i,n}\left( 2\left( n-k\right) \right)
^{-1/2}t\right) \right] \\
&&\times \phi (t)dt+o(1).
\end{eqnarray*}%
The indicated term in the above display is $o(1)$ by the Lemma in the
Appendix and because the expression in brackets inside the integral is
bounded by $1$. Since $r_{i,n}\rightarrow r_{i}$ and $r_{i,n}^{\ast
}\rightarrow -\infty $, the integrand converges to $\Phi \left( r_{i}\right) 
$ under 3.$~$and to $\Phi \left( r_{i}+d_{i}t\right) $ under 4. The
dominated convergence theorem then completes the proof. The case $\zeta
_{i}=-1$ is treated similarly.

It remains to prove 5. Again assume $\zeta _{i}=1$ first. Define $%
r_{i,n}^{\prime }=2^{1/2}n^{-1/2}\eta _{i,n}^{-1}\left( n-k\right)
^{1/2}r_{i,n}$ and $r_{i,n}^{\prime \prime }=2^{1/2}n^{-1/2}\eta
_{i,n}^{-1}\left( n-k\right) ^{1/2}r_{i,n}^{\ast }$ and rewrite the above
display as%
\begin{eqnarray*}
&&P_{n,\theta ^{(n)},\sigma _{n}}\left( \tilde{\theta}_{i}=0\right) \\
&=&\int_{-\infty }^{\infty }\left[ \Phi \left( n^{1/2}\eta _{i,n}\left(
2\left( n-k\right) \right) ^{-1/2}\left( r_{i,n}^{\prime }+t\right) \right)
-\Phi \left( n^{1/2}\eta _{i,n}\left( 2\left( n-k\right) \right)
^{-1/2}\left( r_{i,n}^{\prime \prime }-t\right) \right) \right] \\
&&\times \phi (t)dt+o(1).
\end{eqnarray*}%
Observe that $r_{i,n}^{\prime }\rightarrow r_{i}^{\prime }$ and $%
r_{i,n}^{\prime \prime }\rightarrow -\infty $. The expression in brackets
inside the integral hence converges to $1$ for $t>-r_{i}^{\prime }$ and to $%
0 $ for $t<-r_{i}^{\prime }$. By dominated convergence the integral
converges to $\int_{-r_{i}^{\prime }}^{\infty }\phi (t)dt=\Phi
(r_{i}^{\prime })$. The case $\zeta _{i}=-1$ is treated similarly. $%
\blacksquare $

\bigskip

\textbf{Proof of Proposition \ref{closeness_prob}: }Observe that%
\begin{eqnarray}
&&\left\vert P_{n,\theta ,\sigma }\left( \hat{\theta}_{i}=0\right)
-P_{n,\theta ,\sigma }\left( \tilde{\theta}_{i}=0\right) \right\vert  \notag
\\
&\leq &\int_{0}^{\infty }\left\{ \left\vert \Phi \left( n^{1/2}\left(
-\theta _{i}/(\sigma \xi _{i,n})+\eta _{i,n}\right) \right) -\Phi \left(
n^{1/2}\left( -\theta _{i}/(\sigma \xi _{i,n})+s\eta _{i,n}\right) \right)
\right\vert \right.  \notag \\
&&\left. +\left\vert \Phi \left( n^{1/2}\left( -\theta _{i}/(\sigma \xi
_{i,n})-\eta _{i,n}\right) \right) -\Phi \left( n^{1/2}\left( -\theta
_{i}/(\sigma \xi _{i,n})-s\eta _{i,n}\right) \right) \right\vert \right\}
\rho _{n-k}(s)ds.  \label{prob_diff}
\end{eqnarray}%
By a trivial modification of Lemma 13 in P\"{o}tscher and Schneider (2010)
we conclude that for every $\varepsilon >0$ there exists a real number $%
c=c(\varepsilon )>0$ such that 
\begin{equation*}
\int_{\left\vert s-1\right\vert >(n-k)^{-1/2}c}\rho _{n-k}(s)ds<\varepsilon
\end{equation*}%
for every $n>k$. Using the fact, that $\Phi $ is globally Lipschitz with
constant $(2\pi )^{-1/2}$, this gives 
\begin{eqnarray*}
&&\sup_{\theta \in \mathbb{R}^{k},0<\sigma <\infty }\left\vert P_{n,\theta
,\sigma }\left( \hat{\theta}_{i}=0\right) -P_{n,\theta ,\sigma }\left( 
\tilde{\theta}_{i}=0\right) \right\vert \\
&\leq &2\int_{\left\vert s-1\right\vert >(n-k)^{-1/2}c}\rho _{n-k}(s)ds \\
&&+2(2\pi )^{-1/2}n^{1/2}\eta _{i,n}\int_{\left\vert s-1\right\vert \leq
(n-k)^{-1/2}c}\left\vert s-1\right\vert \rho _{n-k}(s)ds \\
&\leq &2\varepsilon +2(2\pi )^{-1/2}n^{1/2}\eta _{i,n}(n-k)^{-1/2}c
\end{eqnarray*}%
which proves the result since $\varepsilon $ can be made arbitrarily small. $%
\blacksquare $

\subsection{Proofs for Section \protect\ref{minimax}}

\textbf{Proof of Theorem \ref{thresh_consistency}: }(a) Observe that 
\begin{equation}
\left\vert \tilde{\theta}_{i}-\hat{\theta}_{LS,i}\right\vert \leq \hat{\sigma%
}\xi _{i,n}\eta _{i,n}  \label{closeness_H_S_AS_LS_unknown}
\end{equation}%
holds for any of the estimators. Hence, consistency of $\tilde{\theta}_{i}$
under $\xi _{i,n}\eta _{i,n}\rightarrow 0$ and $\xi
_{i,n}/n^{1/2}\rightarrow 0$ follows immediately from Proposition \ref%
{ls_consistency}(a) since the distributions of $\hat{\sigma}/\sigma $ are
tight. Conversely, suppose $\tilde{\theta}_{i}$ is consistent. Then clearly $%
P_{n,\theta ,\sigma }\left( \tilde{\theta}_{i}=0\right) \rightarrow 0$
whenever $\theta _{i}\neq 0$ must hold, which implies $\xi _{i,n}\eta
_{i,n}\rightarrow 0$ by Proposition \ref{select_prob_pointwise_unknown}(a).
This then entails consistency of $\hat{\theta}_{LS,i}$ by (\ref%
{closeness_H_S_AS_LS_unknown}) and tightness of the distributions of $\hat{%
\sigma}/\sigma $; this in turn implies $\xi _{i,n}/n^{1/2}\rightarrow 0$ by
Proposition \ref{ls_consistency}(a).

(b) Since $a_{i,n}\rightarrow \infty $, it suffices to prove the second
claim in (b). Now for every real $M>0$ we have%
\begin{eqnarray*}
&&P_{n,\theta ,\sigma }\left( a_{i,n}\left\vert \tilde{\theta}_{H,i}-\theta
_{i}\right\vert >\sigma M\right) \\
&=&P_{n,\theta ,\sigma }\left( a_{i,n}\left\vert \hat{\theta}_{LS,i}-\theta
_{i}\right\vert >\sigma M,\left\vert \hat{\theta}_{LS,i}\right\vert >\hat{%
\sigma}\xi _{i,n}\eta _{i,n}\right) \\
&&+\boldsymbol{1}\left( a_{i,n}\left\vert \theta _{i}\right\vert >\sigma
M\right) P_{n,\theta ,\sigma }\left( \left\vert \hat{\theta}%
_{LS,i}\right\vert \leq \hat{\sigma}\xi _{i,n}\eta _{i,n}\right) \\
&\leq &P_{n,\theta ,\sigma }\left( a_{i,n}\left\vert \hat{\theta}%
_{LS,i}-\theta _{i}\right\vert >\sigma M\right) +\boldsymbol{1}\left(
a_{i,n}\left\vert \theta _{i}\right\vert >\sigma M\right) P_{n,\theta
,\sigma }\left( \left\vert \hat{\theta}_{LS,i}\right\vert \leq \hat{\sigma}%
\xi _{i,n}\eta _{i,n}\right) \\
&\leq &P_{n,\theta ,\sigma }\left( \left( n^{1/2}/\xi _{i,n}\right)
\left\vert \hat{\theta}_{LS,i}-\theta _{i}\right\vert >\sigma M\right) +%
\boldsymbol{1}\left( a_{i,n}\left\vert \theta _{i}\right\vert >\sigma
M\right) P_{n,\theta ,\sigma }\left( \left\vert \hat{\theta}%
_{LS,i}\right\vert \leq \hat{\sigma}\xi _{i,n}\eta _{i,n}\right) .
\end{eqnarray*}%
This gives%
\begin{eqnarray*}
&&\sup_{n\in \mathbb{N}}\sup_{\theta \in \mathbb{R}^{k}}\sup_{0<\sigma
<\infty }P_{n,\theta ,\sigma }\left( a_{i,n}\left\vert \tilde{\theta}%
_{H,i}-\theta _{i}\right\vert >\sigma M\right) \\
&\leq &\sup_{n\in \mathbb{N}}\sup_{\theta \in \mathbb{R}^{k}}\sup_{0<\sigma
<\infty }P_{n,\theta ,\sigma }\left( \left( n^{1/2}/\xi _{i,n}\right)
\left\vert \hat{\theta}_{LS,i}-\theta _{i}\right\vert >\sigma M\right) \\
&&+\sup_{n\in \mathbb{N}}\sup_{0<\sigma <\infty }\sup_{\theta \in \mathbb{R}%
^{k}:\left\vert \theta _{i}\right\vert >\sigma M/a_{i,n}}P_{n,\theta ,\sigma
}\left( \left\vert \hat{\theta}_{LS,i}\right\vert \leq \hat{\sigma}\xi
_{i,n}\eta _{i,n}\right)
\end{eqnarray*}%
where the first term on the r.h.s.~can be made arbitrarily small in view of
Proposition \ref{ls_consistency}(b) by choosing $M$ large enough. The second
term on the r.h.s.~can be written as (cf.~(\ref{select_prob_unknown})) 
\begin{eqnarray*}
&&\sup_{n\in \mathbb{N}}\sup_{0<\sigma <\infty }\sup_{\theta \in \mathbb{R}%
^{k}:\left\vert \theta _{i}\right\vert >\sigma M/a_{i,n}}\int_{0}^{\infty
}P_{n,\theta ,\sigma }\left( \left\vert \hat{\theta}_{LS,i}\right\vert \leq
s\sigma \xi _{i,n}\eta _{i,n}\right) \rho _{n-k}(s)ds \\
&\leq &\sup_{n\in \mathbb{N}}\sup_{0<\sigma <\infty }\int_{0}^{\infty
}\sup_{\theta \in \mathbb{R}^{k}:\left\vert \theta _{i}\right\vert >\sigma
M/a_{i,n}}P_{n,\theta ,\sigma }\left( \left\vert \hat{\theta}%
_{LS,i}\right\vert \leq s\sigma \xi _{i,n}\eta _{i,n}\right) \rho
_{n-k}(s)ds.
\end{eqnarray*}%
For $\varepsilon >0$ choose $c^{\ast }(\varepsilon /2)$ as in the proof of
Proposition \ref{select_prob_pointwise_unknown}. Using continuity of $\Phi $
and the fact that the probability appearing on the r.h.s.~above is
monotonically increasing as $\left\vert \theta _{i}\right\vert $ approaches $%
\sigma M/a_{i,n}$ from above, this can be further bounded by%
\begin{eqnarray*}
&\leq &\sup_{n\in \mathbb{N}}\int_{0}^{\infty }\Phi \left( sn^{1/2}\eta
_{i,n}-Ma_{i,n}^{-1}n^{1/2}/\xi _{i,n}\right) \rho _{n-k}(s)ds \\
&\leq &\varepsilon /2+\sup_{n\in \mathbb{N}}\int_{0}^{c^{\ast }(\varepsilon
/2)}\Phi \left( sn^{1/2}\eta _{i,n}-Ma_{i,n}^{-1}n^{1/2}/\xi _{i,n}\right)
\rho _{n-k}(s)ds \\
&\leq &\varepsilon /2+\sup_{n\in \mathbb{N}}\Phi \left( n^{1/2}\xi
_{i,n}^{-1}a_{i,n}^{-1}\left( c^{\ast }(\varepsilon /2)\xi _{i,n}\eta
_{i,n}a_{i,n}-M\right) \right) \leq \varepsilon /2+\Phi \left( c^{\ast
}(\varepsilon /2)-M\right) ,
\end{eqnarray*}%
the last inequality holding for $M>c^{\ast }(\varepsilon /2)$ and since $%
n^{1/2}\xi _{i,n}^{-1}a_{i,n}^{-1}\geq 1$ and $\xi _{i,n}\eta
_{i,n}a_{i,n}\leq 1$. Choosing $M$ sufficiently large (depending on $%
\varepsilon $) completes the proof for $\tilde{\theta}_{H,i}$. Next observe
that 
\begin{equation*}
a_{i,n}\left\vert \tilde{\theta}_{H,i}-\tilde{\theta}_{S,i}\right\vert \leq 
\hat{\sigma}\min \left( n^{1/2}\eta _{i,n},1\right) \leq \hat{\sigma}
\end{equation*}%
and similarly $a_{i,n}\left\vert \tilde{\theta}_{H,i}-\tilde{\theta}%
_{AS,i}\right\vert \leq \hat{\sigma}$ hold. Since the set of distributions
of $\hat{\sigma}/\sigma $ (i.e., the set of distributions corresponding to $%
\rho _{n-k}$) is tight as already noted, this proves (b) then also for $\hat{%
\theta}_{S,i}$ and $\hat{\theta}_{AS,i}$.

(c) By a subsequence argument we can reduce the argument to the case where $%
n^{1/2}\eta _{i,n}\rightarrow e_{i}\in \mathbb{\bar{R}}$ and $n-k$ converges
in $\mathbb{N}\cup \{\infty \}$. Suppose first that $e_{i}=\infty $: Observe
that then $a_{i,n}=(\xi _{i,n}\eta _{i,n})^{-1}$ eventually. Choose $\theta
_{i,n}$ and $\sigma _{n}$ such that $\theta _{i,n}/\left( \sigma _{n}\xi
_{i,n}\eta _{i,n}\right) =\zeta _{i}$, where $\zeta _{i}$ does not depend on 
$n$ and $0<\left\vert \zeta _{i}\right\vert <1$ holds, and set the other
coordinates of $\theta ^{(n)}$ to arbitrary values (e.g., equal to zero).
Observe that there exists a constant $\delta >0$ such that 
\begin{equation}
\liminf_{n\rightarrow \infty }P_{n,\theta ^{(n)},\sigma _{n}}\left( \tilde{%
\theta}_{i}=0\right) >\delta  \label{delta}
\end{equation}%
holds: If $n-k$ converges to a finite limit, i.e., is eventually constant,
the claim follows from Theorem \ref{select_prob_moving_par_unknown}(b1); if $%
n-k\rightarrow \infty $, then use Theorem \ref%
{select_prob_moving_par_unknown}(b2). By (\ref{nec}) we have for $%
\varepsilon =\delta $ and a suitable $M$ that%
\begin{eqnarray*}
\delta &>&P_{n,\theta ^{(n)},\sigma _{n}}\left( b_{i,n}\left\vert \tilde{%
\theta}_{i}-\theta _{i,n}\right\vert >\sigma _{n}M\right) \geq P_{n,\theta
^{(n)},\sigma _{n}}\left( b_{i,n}\left\vert \tilde{\theta}_{i}-\theta
_{i,n}\right\vert >\sigma _{n}M,\tilde{\theta}_{i}=0\right) \\
&=&P_{n,\theta ^{(n)},\sigma _{n}}\left( \left\vert b_{i,n}\theta
_{i,n}\right\vert /\sigma _{n}>M,\tilde{\theta}_{i}=0\right) \\
&=&\boldsymbol{1}\left( \left\vert b_{i,n}\theta _{i,n}\right\vert /\sigma
_{n}>M\right) P_{n,\theta ^{(n)},\sigma _{n}}\left( \tilde{\theta}%
_{i}=0\right) >\delta \boldsymbol{1}\left( \left\vert b_{i,n}\theta
_{i,n}\right\vert /\sigma _{n}>M\right)
\end{eqnarray*}%
for all $n$ sufficiently large. But this is only possible if $b_{i,n}\xi
_{i,n}\eta _{i,n}\leq M/\left\vert \zeta _{i}\right\vert <\infty $ holds
eventually, implying that $b_{i,n}=O(a_{i,n})$. Next consider the case where 
$0<e_{i}<\infty $: Observe that then $a_{i,n}$ is of the same order as $%
n^{1/2}/\xi _{i,n}$. Then define $\theta _{i,n}$ and $\sigma _{n}$ such that 
$n^{1/2}\theta _{i,n}/\left( \sigma _{n}\xi _{i,n}\right) =\nu _{i}$, where $%
\nu _{i}$ does not depend on $n$ and $0<\left\vert \nu _{i}\right\vert
<\infty $ holds, and set the other coordinates of $\theta ^{(n)}$ to
arbitrary values (e.g., equal to zero). Observe that then (\ref{delta}) also
holds, in view of Theorem \ref{select_prob_moving_par_unknown}(a1) in case $%
n-k$ is eventually constant, and in view of Theorem \ref%
{select_prob_moving_par_unknown}(a2) in case $n-k\rightarrow \infty $. The
rest of the proof is then similar as before. It remains to consider the case 
$e_{i}=0$: It follows from (\ref{closeness_H_S_AS_LS_unknown}), the
assumptions on $\xi _{i,n}$ and $\eta _{i,n}$, from $e_{i}=0$, and from the
observation that $\hat{\theta}_{LS,i}$ is $N(\theta _{i},\sigma ^{2}\xi
_{i,n}^{2}/n)$-distributed, that $n^{1/2}\xi _{i,n}^{-1}\sigma ^{-1}\left( 
\tilde{\theta}_{i}-\theta _{i}\right) $ converges in distribution to a
standard normal distribution for each fixed $\theta _{i}$ and $\sigma $.
Hence, stochastic boundedness of $\sigma ^{-1}b_{i,n}\left\vert \tilde{\theta%
}_{i}-\theta _{i}\right\vert $ for each $\theta _{i}$ (and a fortiori (\ref%
{nec})) necessarily implies that $b_{i,n}=O(n^{1/2}\xi
_{i,n}^{-1})=O(a_{i,n})$.

(d) The proof for $\hat{\theta}_{i}$ is similar and in fact simpler: note
that now $\left\vert \hat{\theta}_{i}-\hat{\theta}_{LS,i}\right\vert \leq
\sigma \xi _{i,n}\eta _{i,n}$ holds and that in the proof of (b) the
integration over $s$ can simply be replaced by evaluation at $s=1$. For (c)
one uses Proposition \ref{select_prob_moving_par} instead of Theorem \ref%
{select_prob_moving_par_unknown}. $\blacksquare $

\subsection{Proofs for Section \protect\ref{FS}}

\textbf{Proofs of Propositions \ref{1}, \ref{2}, and \ref{3}: }Observe that 
\begin{equation*}
\hat{\theta}_{H,i}/(\sigma \xi _{i,n})=\left( \hat{\theta}_{LS,i}/(\sigma
\xi _{i,n})\right) \boldsymbol{1}\left( \left\vert \hat{\theta}%
_{LS,i}/(\sigma \xi _{i,n})\right\vert >\eta _{i,n}\right)
\end{equation*}%
and that $\hat{\theta}_{LS,i}/(\sigma \xi _{i,n})$ is $N\left( \theta
_{i}/(\sigma \xi _{i,n}),1/n\right) $. Furthermore, we have 
\begin{eqnarray*}
H_{H,n,\theta ,\sigma }^{i}(x) &=&P_{n,\theta ,\sigma }\left( \sigma
^{-1}\alpha _{i,n}(\hat{\theta}_{H,i}-\theta _{i})\leq x\right) \\
&=&P_{n,\theta ,\sigma }\left( n^{1/2}(\hat{\theta}_{H,i}-\theta
_{i})/(\sigma \xi _{i,n})\leq n^{1/2}\alpha _{i,n}^{-1}\xi
_{i,n}^{-1}x\right) .
\end{eqnarray*}%
Identifying $\hat{\theta}_{LS,i}/(\sigma \xi _{i,n})$ and $\theta
_{i}/(\sigma \xi _{i,n})$ with $\bar{y}$ and $\theta $ in P\"{o}tscher and
Leeb (2009) and making use of eq.~(4) in that reference immediately gives
the result for $dH_{H,n,\theta ,\sigma }^{i}$. The result for $H_{H,n,\theta
,\sigma }^{i}$ then follows from elementary calculations.

The result for $dH_{S,n,\theta ,\sigma }^{i}$ follows similarly by making
use of eq.~(5) instead of eq.~(4) in P\"{o}tscher and Leeb (2009). The
result for $H_{S,n,\theta ,\sigma }^{i}$ then follows from elementary
calculations.

The results for $dH_{AS,n,\theta ,\sigma }^{i}$ and $H_{AS,n,\theta ,\sigma
}^{i}$ follow similarly by making use of eqs.~(9)-(11) in P\"{o}tscher and
Schneider (2009). $\blacksquare $

\bigskip

\textbf{Proofs of Propositions \ref{4}, \ref{5}, and \ref{6}: }We have 
\begin{eqnarray*}
H_{H,n,\theta ,\sigma }^{i\maltese }(x) &=&\int_{0}^{\infty }P_{n,\theta
,\sigma }\left( \sigma ^{-1}\alpha _{i,n}(\tilde{\theta}_{H,i}-\theta
_{i})\leq x\mid \hat{\sigma}=s\sigma \right) \rho _{n-k}(s)ds \\
&=&\int_{0}^{\infty }H_{H,s\eta _{i,n},n,\theta ,\sigma }^{i}(x)\rho
_{n-k}(s)ds,
\end{eqnarray*}%
where we have used independence of $\hat{\sigma}$ and $\hat{\theta}_{LS,i}$
allowing us to replace $\hat{\sigma}$ by $s\sigma $ in the relevant
formulae, cf.~Leeb and P\"{o}tscher (2003, p.~110). Substituting (\ref%
{hard_finite_sample}), with $\eta _{i,n}$ replaced by $s\eta _{i,n}$, into
the above equation gives (\ref{hard_finite_sample_unknown}). Representing $%
H_{H,s\eta _{i,n},n,\theta ,\sigma }^{i}(x)$ as an integral of $dH_{H,s\eta
_{i,n},n,\theta ,\sigma }^{i}$ given in (\ref{hard_finite_sample_density})
and applying Fubini's theorem then gives (\ref%
{hard_finite_sample_unknown_density}).

Similarly, we have 
\begin{equation*}
H_{S,n,\theta ,\sigma }^{i\maltese }(x)=\int_{0}^{\infty }H_{S,s\eta
_{i,n},n,\theta ,\sigma }^{i}(x)\rho _{n-k}(s)ds.
\end{equation*}%
Substituting (\ref{soft_finite_sample}), with $\eta _{i,n}$ replaced by $%
s\eta _{i,n}$, into the above equation and noting that $\int_{0}^{\infty
}\Phi (a+bs)\rho _{\nu }(s)\,ds=T_{\nu ,-a}(b)$ gives (\ref%
{soft_finite_sample_unknown}). Elementary calculations then yield (\ref%
{soft_finite_sample_unknown_density}).

Finally, we have 
\begin{equation*}
H_{AS,n,\theta ,\sigma }^{i\maltese }(x)=\int_{0}^{\infty }H_{AS,s\eta
_{i,n},n,\theta ,\sigma }^{i}(x)\rho _{n-k}(s)ds.
\end{equation*}%
Substituting (\ref{adaptive_finite_sample}), with $\eta _{i,n}$ replaced by $%
s\eta _{i,n}$, into the above equation gives (\ref%
{adaptive_finite_sample_unknown}). Elementary calculations then yield (\ref%
{adaptive_finite_sample_unknown_density}). $\blacksquare $

\subsection{Proofs for Section \protect\ref{LS} \label{prfs_LS}}

\textbf{Proof of Proposition \ref{LSDK_H} : }The proof of (a) is completely
analogous to the proof of Theorem 4 in P\"{o}tscher and Leeb (2009), whereas
the proof of (b) is analogous to the proof of Theorem 17 in the same
reference. $\blacksquare $

\bigskip

\textbf{Proof of Proposition \ref{LSDK_S} : }The proof of (a) is completely
analogous to the proof of Theorem 5 in P\"{o}tscher and Leeb (2009), whereas
the proof of (b) is analogous to the proof of Theorem 18 in the same
reference. $\blacksquare $

\bigskip

\textbf{Proof of Proposition \ref{LSDK_AS} : }The proof of (a) is completely
analogous to the proof of Theorem 4 in P\"{o}tscher and Schneider (2009),
whereas the proof of (b) is analogous to the proof of Theorem 6 in the same
reference. $\blacksquare $

\bigskip

\textbf{Proof of Theorem \ref{closeness}: }Observe that the total variation
distance between two cdfs is bounded by the sum of the total variation
distances between the corresponding discrete and continuous parts.
Furthermore, recall that the total variation distance between the absolutely
continuous parts is bounded from above by the $L_{1}$-distance of the
corresponding densities. Hence, from (\ref{hard_finite_sample_density}) and (%
\ref{hard_finite_sample_unknown_density}) we obtain

\begin{equation*}
\left\Vert H_{H,n,\theta ,\sigma }^{i}-H_{H,n,\theta ,\sigma }^{i\maltese
}\right\Vert _{TV}\leq A+B
\end{equation*}%
where%
\begin{equation*}
A=\left\vert P_{n,\theta ,\sigma }\left( \hat{\theta}_{H,i}=0\right)
-P_{n,\theta ,\sigma }\left( \tilde{\theta}_{H,i}=0\right) \right\vert
\end{equation*}%
and%
\begin{eqnarray*}
B &=&\int_{-\infty }^{\infty }\int_{0}^{\infty }\left\vert \boldsymbol{1}%
\left( \left\vert \alpha _{i,n}^{-1}x+\theta _{i}/\sigma \right\vert >\xi
_{i,n}\eta _{i,n}\right) \right. \\
&&\left. -\boldsymbol{1}\left( \left\vert \alpha _{i,n}^{-1}x+\theta
_{i}/\sigma \right\vert >\xi _{i,n}s\eta _{i,n}\right) \right\vert \rho
_{n-k}(s)dsn^{1/2}\alpha _{i,n}^{-1}\xi _{i,n}^{-1}\phi \left(
n^{1/2}x/(\alpha _{i,n}\xi _{i,n})\right) dx \\
&=&\int_{0}^{\infty }\int_{-\infty }^{\infty }\left\vert \boldsymbol{1}%
\left( \left\vert u+n^{1/2}\theta _{i}/\left( \sigma \xi _{i,n}\right)
\right\vert >n^{1/2}\eta _{i,n}\right) \right. \\
&&\left. -\boldsymbol{1}\left( \left\vert u+n^{1/2}\theta _{i}/\left( \sigma
\xi _{i,n}\right) \right\vert >sn^{1/2}\eta _{i,n}\right) \right\vert \phi
(u)du\rho _{n-k}(s)ds \\
&=&\int_{0}^{\infty }\int_{-\infty }^{\infty }\boldsymbol{1}\left(
n^{1/2}\eta _{i,n}(s\wedge 1)<\left\vert u+n^{1/2}\theta _{i}/(\sigma \xi
_{i,n})\right\vert \leq n^{1/2}\eta _{i,n}(s\vee 1)\right) \phi \left(
u\right) du\rho _{n-k}(s)ds \\
&=&\int_{0}^{\infty }\left\{ \left[ \Phi \left( n^{1/2}\left( -\theta
_{i}/(\sigma \xi _{i,n})+\eta _{i,n}(s\vee 1)\right) \right) -\Phi \left(
n^{1/2}\left( -\theta _{i}/(\sigma \xi _{i,n})+\eta _{i,n}(s\wedge 1)\right)
\right) \right] \right. \\
&&+\left. \left[ \Phi \left( n^{1/2}\left( -\theta _{i}/(\sigma \xi
_{i,n})-\eta _{i,n}(s\wedge 1)\right) \right) -\Phi \left( n^{1/2}\left(
-\theta _{i}/(\sigma \xi _{i,n})-\eta _{i,n}(s\vee 1)\right) \right) \right]
\right\} \rho _{n-k}(s)ds,
\end{eqnarray*}%
where we have made use of Fubini's theorem and performed an obvious
substitution. By a trivial modification of Lemma 13 in P\"{o}tscher and
Schneider (2010) we conclude that for every $\varepsilon >0$ there exists a
real number $c=c(\varepsilon )>0$ such that 
\begin{equation}
\int_{\left\vert s-1\right\vert >(n-k)^{-1/2}c}\rho _{n-k}(s)ds<\varepsilon
\label{c}
\end{equation}%
for every $n-k>0$. Using the fact, that $\Phi $ is globally Lipschitz with
constant $(2\pi )^{-1/2}$, this gives%
\begin{eqnarray*}
\sup_{\theta \in \mathbb{R}^{k},0<\sigma <\infty }B &\leq &2\int_{\left\vert
s-1\right\vert >(n-k)^{-1/2}c}\rho _{n-k}(s)ds \\
&&+2(2\pi )^{-1/2}n^{1/2}\eta _{i,n}\int_{\left\vert s-1\right\vert \leq
(n-k)^{-1/2}c}\left\vert (s\vee 1)-(s\wedge 1)\right\vert \rho _{n-k}(s)ds \\
&\leq &2\varepsilon +2(2\pi )^{-1/2}n^{1/2}\eta _{i,n}(n-k)^{-1/2}c.
\end{eqnarray*}%
The r.h.s.~now converges to $2\varepsilon $ because $n^{1/2}\eta
_{i,n}(n-k)^{-1/2}\rightarrow 0$. Since $\varepsilon >0$ was arbitrary, this
shows that $\sup_{\theta \in \mathbb{R}^{k},0<\sigma <\infty }B$ converges
to zero. Note also that $\sup_{\theta \in \mathbb{R}^{k},0<\sigma <\infty }A$
has already been shown to converge to zero in Proposition \ref%
{closeness_prob}. This completes the proof for the hard-thresholding
estimator.

With the same argument as above we obtain%
\begin{equation*}
\left\Vert H_{S,n,\theta ,\sigma }^{i}-H_{S,n,\theta ,\sigma }^{i\maltese
}\right\Vert _{TV}\leq A+B,
\end{equation*}%
where 
\begin{equation*}
A=\left\vert P_{n,\theta ,\sigma }\left( \hat{\theta}_{S,i}=0\right)
-P_{n,\theta ,\sigma }\left( \tilde{\theta}_{S,i}=0\right) \right\vert
\end{equation*}%
and%
\begin{eqnarray*}
B &=&n^{1/2}\alpha _{i,n}^{-1}\xi _{i,n}^{-1}\int_{-\infty }^{\infty
}\int_{0}^{\infty }\left\vert \phi \left( n^{1/2}x/(\alpha _{i,n}\xi
_{i,n})+n^{1/2}\eta _{i,n}\right) \right. \\
&&\left. -\phi \left( n^{1/2}x/(\alpha _{i,n}\xi _{i,n})+n^{1/2}s\eta
_{i,n}\right) \right\vert \rho _{n-k}(s)ds\boldsymbol{1}\left( \alpha
_{i,n}^{-1}x+\theta _{i}/\sigma >0\right) dx \\
&&+n^{1/2}\alpha _{i,n}^{-1}\xi _{i,n}^{-1}\int_{-\infty }^{\infty
}\int_{0}^{\infty }\left\vert \phi \left( n^{1/2}x/(\alpha _{i,n}\xi
_{i,n})-n^{1/2}\eta _{i,n}\right) \right. \\
&&\left. -\phi \left( n^{1/2}x/(\alpha _{i,n}\xi _{i,n})-n^{1/2}s\eta
_{i,n}\right) \right\vert \rho _{n-k}(s)ds\boldsymbol{1}\left( \alpha
_{i,n}^{-1}x+\theta _{i}/\sigma <0\right) dx
\end{eqnarray*}%
where we have used (\ref{soft_finite_sample_density}) and (\ref%
{soft_finite_sample_unknown_density}). Now,%
\begin{equation*}
B\leq \int_{0}^{\infty }\left( B_{1}(s)+B_{2}(s)\right) \rho _{n-k}(s)ds
\end{equation*}%
where%
\begin{eqnarray*}
B_{1}(s) &=&\int_{-\infty }^{\infty }\left\vert \phi \left( u+n^{1/2}\eta
_{i,n}\right) -\phi \left( u+n^{1/2}s\eta _{i,n}\right) \right\vert du, \\
B_{2}(s) &=&\int_{-\infty }^{\infty }\left\vert \phi \left( u-n^{1/2}\eta
_{i,n}\right) -\phi \left( u-n^{1/2}s\eta _{i,n}\right) \right\vert du,
\end{eqnarray*}%
and where we have used Fubini's theorem and an obvious substitution. It is
elementary to verify that%
\begin{equation*}
B_{1}(s)=B_{2}(s)=2\left\vert \Phi (n^{1/2}\eta _{i,n}(s-1)/2)-\Phi
(-n^{1/2}\eta _{i,n}(s-1)/2)\right\vert ,
\end{equation*}%
and that $B_{1}(s)\leq 2$ holds. Consequently, using (\ref{c}) we obtain%
\begin{eqnarray*}
B &\leq &4\int_{\left\vert s-1\right\vert >(n-k)^{-1/2}c}\rho
_{n-k}(s)ds+\int_{\left\vert s-1\right\vert \leq (n-k)^{-1/2}c}\left(
B_{1}(s)+B_{2}(s)\right) \rho _{n-k}(s)ds \\
&\leq &4\varepsilon +4(2\pi )^{-1/2}n^{1/2}\eta _{i,n}\int_{\left\vert
s-1\right\vert \leq (n-k)^{-1/2}c}\left\vert s-1\right\vert \rho _{n-k}(s)ds
\\
&\leq &4\varepsilon +4(2\pi )^{-1/2}n^{1/2}\eta _{i,n}(n-k)^{-1/2}c,
\end{eqnarray*}%
where we have again used the fact that $\Phi $ is globally Lipschitz with
constant $(2\pi )^{-1/2}$. Since $n^{1/2}\eta _{i,n}(n-k)^{-1/2}\rightarrow
0 $ and $\varepsilon >0$ was arbitrary, the proof for soft-thresholding is
complete, because $\sup_{\theta \in \mathbb{R}^{k},0<\sigma <\infty }A$ goes
to zero by Proposition \ref{closeness_prob}.

Finally, from (\ref{adaptive_finite_sample}) and (\ref%
{adaptive_finite_sample_unknown}) we obtain%
\begin{eqnarray*}
\left\Vert H_{AS,n,\theta ,\sigma }^{i}-H_{AS,n,\theta ,\sigma }^{i\maltese
}\right\Vert _{\infty } &\leq &\int_{0}^{\infty }\sup_{x\in \mathbb{R}%
}\left\vert \Phi \left( z_{n,\theta ,\sigma }^{(2)}(x,\eta _{i,n})\right)
-\Phi \left( z_{n,\theta ,\sigma }^{(2)}(x,s\eta _{i,n})\right) \right\vert
\rho _{n-k}(s)ds \\
&&+\int_{0}^{\infty }\sup_{x\in \mathbb{R}}\left\vert \Phi \left(
z_{n,\theta ,\sigma }^{(1)}(x,\eta _{i,n})\right) -\Phi \left( z_{n,\theta
,\sigma }^{(1)}(x,s\eta _{i,n})\right) \right\vert \rho _{n-k}(s)ds \\
&=&:\int_{0}^{\infty }C_{1}(s)\rho _{n-k}(s)ds+\int_{0}^{\infty
}C_{2}(s)\rho _{n-k}(s)ds.
\end{eqnarray*}%
Observe that on the one hand $C_{1}(s)$ and $C_{2}(s)$ are bounded by $1$,
and that on the other hand, using the Lipschitz-property of $\Phi $ and the
mean-value theorem,%
\begin{eqnarray*}
\left\vert C_{1}(s)\right\vert &\leq &(2\pi )^{-1/2}\sup_{x\in \mathbb{R}%
}\left\vert z_{n,\theta ,\sigma }^{(2)}(x,\eta _{i,n})-z_{n,\theta ,\sigma
}^{(2)}(x,s\eta _{i,n})\right\vert \\
&=&(2\pi )^{-1/2}\sup_{x\in \mathbb{R}}\left\vert n^{1/2}\sqrt{\left( 0.5\xi
_{i,n}^{-1}(\alpha _{i,n}^{-1}x+\theta _{i}/\sigma )\right) ^{2}+\eta
_{i,n}^{2}}\right. \\
&&\left. -n^{1/2}\sqrt{\left( 0.5\xi _{i,n}^{-1}(\alpha _{i,n}^{-1}x+\theta
_{i}/\sigma )\right) ^{2}+s^{2}\eta _{i,n}^{2}}\right\vert \\
&\leq &(2\pi )^{-1/2}n^{1/2}\eta _{i,n}^{2}\left\vert s-1\right\vert
\sup_{x\in \mathbb{R}}\left\vert \left( \left( 0.5\xi _{i,n}^{-1}(\alpha
_{i,n}^{-1}x+\theta _{i}/\sigma )\right) ^{2}\bar{s}^{-2}+\eta
_{i,n}^{2}\right) ^{-1/2}\right\vert ,
\end{eqnarray*}%
where $\bar{s}$ is a mean-value between $s$ and $1$ which may depend on $x$.
The supremum over $x$ on the r.h.s.~is now clearly assumed for $x=-\alpha
_{i,n}\theta _{i}/\sigma $, resulting in the bound%
\begin{equation*}
\left\vert C_{1}(s)\right\vert \leq (2\pi )^{-1/2}n^{1/2}\eta
_{i,n}\left\vert s-1\right\vert .
\end{equation*}%
The same bound is obtained for $C_{2}$ in exactly the same way.
Consequently, using (\ref{c}) we obtain%
\begin{eqnarray*}
\sup_{\theta \in \mathbb{R}^{k},0<\sigma <\infty }\left\Vert H_{AS,n,\theta
,\sigma }^{i}-H_{AS,n,\theta ,\sigma }^{i\maltese }\right\Vert _{\infty }
&\leq &2\int_{\left\vert s-1\right\vert >(n-k)^{-1/2}c}\rho _{n-k}(s)ds \\
&&+2(2\pi )^{-1/2}n^{1/2}\eta _{i,n}\int_{\left\vert s-1\right\vert \leq
(n-k)^{-1/2}c}\left\vert s-1\right\vert \rho _{n-k}(s)ds \\
&\leq &2\left[ \varepsilon +(2\pi )^{-1/2}n^{1/2}\eta _{i,n}(n-k)^{-1/2}c%
\right] .
\end{eqnarray*}%
Since $n^{1/2}\eta _{i,n}(n-k)^{-1/2}\rightarrow 0$ and $\varepsilon >0$ was
arbitrary, the proof is complete. $\blacksquare $

\bigskip

\textbf{Proof of Theorem \ref{HTconservative}: }(a) The atomic part of $%
dH_{H,n,\theta ^{(n)},\sigma _{n}}^{i\maltese }$ as given in (\ref%
{hard_finite_sample_unknown_density}) clearly converges weakly to the atomic
part of (\ref{hard_large_sample_unknown_density_A}) in view of Theorem \ref%
{select_prob_moving_par_unknown}(a1) and the fact that $\alpha _{i,n}\theta
_{i,n}/\sigma _{n}=n^{1/2}\theta _{i,n}/(\sigma _{n}\xi _{i,n})\rightarrow
\nu _{i}$ by assumption; also note that the atomic part converges to the
zero measure in case $\left\vert \nu _{i}\right\vert =\infty $ or $e_{i}=0$
as then the total mass of the atomic part converges to zero. We turn to the
absolutely continuous part next. For later use we note that what has been
established so far also implies that the total mass of the absolutely
continuous part converges to the total mass of the absolutely continuous
part of the limit, since it is easy to see that the limiting distribution
given in the theorem has total mass $1$. The density of the absolutely
continuous part of (\ref{hard_finite_sample_unknown_density}) takes the form%
\begin{equation*}
\phi \left( x\right) \int_{0}^{\infty }\boldsymbol{1}\left( \left\vert
x+n^{1/2}\theta _{i,n}/(\sigma _{n}\xi _{i,n})\right\vert >sn^{1/2}\eta
_{i,n}\right) \rho _{n-k}(s)ds.
\end{equation*}%
Observe that for given $x\in \mathbb{R}$, the indicator function in the
above display converges to \linebreak $\boldsymbol{1}\left( \left\vert x+\nu
_{i}\right\vert >se_{i}\right) $ for Lebesgue almost all $s$. [If $e_{i}=0$,
this is necessarily true only for $x\in \mathbb{R}$ with $x\neq -\nu _{i}$.]
Since $n-k=m$ eventually, we get from the dominated convergence theorem that
the above display converges to $\phi \left( x\right) \int_{0}^{\infty }%
\boldsymbol{1}\left( \left\vert x+\nu _{i}\right\vert >se_{i}\right) \rho
_{m}(s)ds$ for every $x\in \mathbb{R}$ (for every $x\in \mathbb{R}$ with $%
x\neq -\nu _{i}$ in case $e_{i}=0$), which is the density of the absolutely
continuous part in (\ref{hard_large_sample_unknown_density_A}). Since the
total mass of the absolutely continuous part is preserved in the limit as
shown above, the proof is completed by Scheff\'{e}'s Lemma.

(b) Follows immediately from Proposition \ref{LSDK_H} and Theorem \ref%
{closeness}. $\blacksquare $

\bigskip

\textbf{Proof of Theorem \ref{STconservative}: }(a) The atomic part of $%
dH_{S,n,\theta ^{(n)},\sigma _{n}}^{i\maltese }$ as given in (\ref%
{soft_finite_sample_unknown_density}) converges weakly to the atomic part of
(\ref{soft_large_sample_unknown_density_A}) in view of Theorem \ref%
{select_prob_moving_par_unknown}(a1) and the fact that $\alpha _{i,n}\theta
_{i,n}/\sigma _{n}=n^{1/2}\theta _{i,n}/(\sigma _{n}\xi _{i,n})\rightarrow
\nu _{i}$ by assumption; also note that the atomic part converges to the
zero measure in case $\left\vert \nu _{i}\right\vert =\infty $ or $e_{i}=0$
as then the total mass of the atomic part converges to zero. We turn to the
absolutely continuous part next. For later use we note that what has been
established so far also implies that the total mass of the absolutely
continuous part converges to the total mass of the absolutely continuous
part of the limit, since it is easy to see that the limiting distribution
given in the theorem has total mass $1$. The density of the absolutely
continuous part of (\ref{soft_finite_sample_unknown_density}) takes the form%
\begin{eqnarray*}
&&\int_{0}^{\infty }\phi \left( x+sn^{1/2}\eta _{i,n}\right) \rho _{n-k}(s)ds%
\boldsymbol{1}\left( x+n^{1/2}\theta _{i,n}/(\sigma _{n}\xi _{i,n})>0\right)
\\
&&+\int_{0}^{\infty }\phi \left( x-sn^{1/2}\eta _{i,n}\right) \rho
_{n-k}(s)ds\boldsymbol{1}\left( x+n^{1/2}\theta _{i,n}/(\sigma _{n}\xi
_{i,n})<0\right) .
\end{eqnarray*}%
Observe that for given $x\in \mathbb{R}$, the functions $\phi \left( x\pm
sn^{1/2}\eta _{i,n}\right) $ converge to $\phi \left( x\pm se_{i}\right) $,
respectively, for all $s$. Since $n-k=m$ eventually, we then get from the
dominated convergence theorem that the above display converges to 
\begin{equation*}
\int_{0}^{\infty }\phi \left( x+se_{i}\right) \rho _{m}(s)ds\boldsymbol{1}%
\left( x+\nu _{i}>0\right) +\int_{0}^{\infty }\phi \left( x-se_{i}\right)
\rho _{m}(s)ds\boldsymbol{1}\left( x+\nu _{i}<0\right)
\end{equation*}%
for every $x\mathbb{\neq -}\nu _{i}$; the last display is precisely the
density of the absolutely continuous part in (\ref%
{soft_large_sample_unknown_density_A}). Since the total mass of the
absolutely continuous part is preserved in the limit as shown above, the
proof is completed by Scheff\'{e}'s Lemma.

(b) Follows immediately from Proposition \ref{LSDK_S} and Theorem \ref%
{closeness}. $\blacksquare $

\bigskip

\textbf{Proof of Theorem \ref{ASTconservative}: }(a) Observe that 
\begin{eqnarray}
H_{AS,n,\theta ^{(n)},\sigma _{n}}^{i\maltese }(x) &=&\int_{0}^{\infty }\Phi
\left( z_{n,\theta ^{(n)},\sigma _{n}}^{(2)}(x,s\eta _{i,n})\right) \rho
_{n-k}(s)ds\boldsymbol{1}\left( x+n^{1/2}\theta _{i,n}/(\sigma _{n}\xi
_{i,n})\geq 0\right)  \label{above} \\
&&+\int_{0}^{\infty }\Phi \left( z_{n,\theta ^{(n)},\sigma
_{n}}^{(1)}(x,s\eta _{i,n})\right) \rho _{n-k}(s)ds\boldsymbol{1}\left(
x+n^{1/2}\theta _{i,n}/(\sigma _{n}\xi _{i,n})<0\right) \,  \notag
\end{eqnarray}%
where $z_{n,\theta ^{(n)},\sigma _{n}}^{(1)}(x,s\eta _{i,n})$ and $%
z_{n,\theta ^{(n)},\sigma _{n}}^{(2)}(x,s\eta _{i,n})$ reduce to 
\begin{equation*}
0.5(x-n^{1/2}\theta _{i,n}/(\sigma _{n}\xi _{i,n}))\pm \sqrt{\left(
0.5(x+n^{1/2}\theta _{i,n}/(\sigma _{n}\xi _{i,n}))\right) ^{2}+s^{2}n\eta
_{i,n}^{2}}.
\end{equation*}%
Clearly, $\Phi \left( z_{n,\theta ^{(n)},\sigma _{n}}^{(1)}(x,s\eta
_{i,n})\right) $ as well as $\Phi \left( z_{n,\theta ^{(n)},\sigma
_{n}}^{(2)}(x,s\eta _{i,n})\right) $ converge for every $s\geq 0$ to 
\begin{equation*}
\Phi \left( 0.5(x-\nu _{i})-\sqrt{\left( 0.5(x+\nu _{i})\right)
^{2}+s^{2}e_{i}^{2}}\right)
\end{equation*}%
and 
\begin{equation*}
\Phi \left( 0.5(x-\nu _{i})+\sqrt{\left( 0.5(x+\nu _{i})\right)
^{2}+s^{2}e_{i}^{2}}\right) ,
\end{equation*}%
respectively, if $\left\vert \nu _{i}\right\vert <\infty $, and the
dominated convergence theorem shows that the weights of the indicator
functions in (\ref{above}) converge to the corresponding weights in (\ref%
{adaptive_soft_large_sample_unknown_cdf_A}). Since $n^{1/2}\theta
_{i,n}/(\sigma _{n}\xi _{i,n})$ converges to $\nu _{i}$ by assumption, it
follows that for every $x\neq -\nu _{i}$ we have convergence of $%
H_{AS,n,\theta ^{(n)},\sigma _{n}}^{i\maltese }$ to the cdf given in (\ref%
{adaptive_soft_large_sample_unknown_cdf_A}). This proves part (a) in case $%
\left\vert \nu _{i}\right\vert <\infty $. In case $\nu _{i}=\infty $, we
have that $z_{n,\theta ^{(n)},\sigma _{n}}^{(2)}(x,s\eta _{i,n})$ converges
to $x$ by an application of Proposition 15 in P\"{o}tscher and Schneider
(2009). Consequently, the limit of $\Phi \left( z_{n,\theta ^{(n)},\sigma
_{n}}^{(2)}(x,s\eta _{i,n})\right) $ is now $\Phi \left( x\right) $. Again
applying the dominated convergence theorem and observing that for each $x\in 
\mathbb{R}$ we have that $\boldsymbol{1}\left( x+n^{1/2}\theta
_{i,n}/(\sigma _{n}\xi _{i,n})<0\right) $ is eventually zero, shows that $%
H_{AS,n,\theta ^{(n)},\sigma _{n}}^{i\maltese }(x)$ converges to $\Phi
\left( x\right) $. The case $\nu _{i}=-\infty $ is proved analogously.

(b) Follows immediately from Proposition \ref{LSDK_AS} and Theorem \ref%
{closeness}. $\blacksquare $

\bigskip

\textbf{Proof of Theorem \ref{HTconsistent}: }Observe that%
\begin{eqnarray*}
\sigma _{n}^{-1}\alpha _{i,n}(\tilde{\theta}_{H,i}-\theta _{i,n}) &=&-\theta
_{i,n}/(\sigma _{n}\xi _{i,n}\eta _{i,n})\boldsymbol{1}\left( \tilde{\theta}%
_{H,i}=0\right) \\
&&+(\sigma _{n}\xi _{i,n}\eta _{i,n})^{-1}\left( \hat{\theta}_{LS,i}-\theta
_{i,n}\right) \boldsymbol{1}\left( \tilde{\theta}_{H,i}\neq 0\right) \\
&=&-\theta _{i,n}/(\sigma _{n}\xi _{i,n}\eta _{i,n})\boldsymbol{1}\left( 
\tilde{\theta}_{H,i}=0\right) +n^{-1/2}\eta _{i,n}^{-1}Z_{n}\boldsymbol{1}%
\left( \tilde{\theta}_{H,i}\neq 0\right)
\end{eqnarray*}%
where $Z_{n}$ is standard normally distributed. The expressions in front of
the indicator functions now converge to $-\zeta _{i}$ and $0$, respectively,
in probability as $n\rightarrow \infty $. Inspection of the cdf of $\sigma
_{n}^{-1}\alpha _{i,n}(\tilde{\theta}_{H,i}-\theta _{i,n})$ then shows that
this cdf converges weakly to%
\begin{equation*}
\left( \lim_{n\rightarrow \infty }P_{n,\theta ^{(n)},\sigma _{n}}\left( 
\tilde{\theta}_{H,i}=0\right) \right) \delta _{-\zeta _{i}}+\left(
1-\lim_{n\rightarrow \infty }P_{n,\theta ^{(n)},\sigma _{n}}\left( \tilde{%
\theta}_{H,i}=0\right) \right) \delta _{0}
\end{equation*}%
if $\left\vert \zeta _{i}\right\vert <\infty $. Part (b) of Theorem \ref%
{select_prob_moving_par_unknown}\ completes the proof of both parts of the
theorem in case $\left\vert \zeta _{i}\right\vert <\infty $. If $\left\vert
\zeta _{i}\right\vert =\infty $ the same theorem shows that the weak limit
is now $\delta _{0}$. $\blacksquare $

\bigskip

\textbf{Proof of Theorem \ref{STconsistent}: }(a) The atomic part of $%
dH_{S,n,\theta ^{(n)},\sigma _{n}}^{i\maltese }$ as given in (\ref%
{soft_finite_sample_unknown_density}) converges weakly to the atomic part
given in (\ref{soft_large_sample_unknown_density_C}) by Theorem \ref%
{select_prob_moving_par_unknown}(b1). The density of the absolutely
continuous part of $dH_{S,n,\theta ^{(n)},\sigma _{n}}^{i\maltese }$ can be
written as%
\begin{eqnarray*}
&&n^{1/2}\eta _{i,n}\int_{-\infty }^{\infty }\phi \left( n^{1/2}\eta
_{i,n}\left( x+s\right) \right) \rho _{m}(s)ds\boldsymbol{1}\left( x+\theta
_{i,n}/(\sigma _{n}\xi _{i,n}\eta _{i,n})>0\right) \\
&&+n^{1/2}\eta _{i,n}\int_{-\infty }^{\infty }\phi \left( n^{1/2}\eta
_{i,n}\left( x-s\right) \right) \rho _{m}(s)ds\boldsymbol{1}\left( x+\theta
_{i,n}/(\sigma _{n}\xi _{i,n}\eta _{i,n})<0\right)
\end{eqnarray*}%
recalling the convention that $\rho _{m}(s)=0$ for $s<0$. Note that with
this convention $\rho _{m}$ is then a bounded continuous function on the
real line. Since $n^{1/2}\eta _{i,n}\phi \left( n^{1/2}\eta _{i,n}\left(
x+\cdot \right) \right) $ and $n^{1/2}\eta _{i,n}\phi \left( n^{1/2}\eta
_{i,n}\left( x-\cdot \right) \right) $ clearly converge weakly to $\delta
_{-x}$ and $\delta _{x}$, respectively, the density of the absolutely
continuous part of $dH_{S,n,\theta ^{(n)},\sigma _{n}}^{i\maltese }$ is seen
to converge to $\rho _{m}(-x)\boldsymbol{1}\left( x+\zeta _{i}>0\right)
+\rho _{m}(x)\boldsymbol{1}\left( x+\zeta _{i}<0\right) $ for every $x\neq
-\zeta _{i}$. An application of Scheff\'{e}'s Lemma then completes the
proof, noting that the total mass of the absolutely continuous part of $%
dH_{S,n,\theta ^{(n)},\sigma _{n}}^{i\maltese }$ converges to the total mass
of the absolutely continuous part of (\ref%
{soft_large_sample_unknown_density_C}) as the same is true for the atomic
part in view of Theorem \ref{select_prob_moving_par_unknown}(b1) (and since
the distributions involved all have total mass $1$).

(b) Rewrite $\sigma _{n}^{-1}\alpha _{i,n}(\tilde{\theta}_{S,i}-\theta
_{i,n})$ as%
\begin{equation*}
-\theta _{i,n}/(\sigma _{n}\xi _{i,n}\eta _{i,n})\boldsymbol{1}\left( \tilde{%
\theta}_{S,i}=0\right) +\left( W_{n}-\left( \hat{\sigma}/\sigma _{n}\right) 
\limfunc{sign}(W_{n}+\theta _{i,n}/(\sigma _{n}\xi _{i,n}\eta
_{i,n}))\right) \boldsymbol{1}\left( \tilde{\theta}_{S,i}\neq 0\right) ,
\end{equation*}%
where $W_{n}$ is a sequence of $N(0,n^{-1}\eta _{i,n}^{-2})$-distributed
random variables. Observe that $\theta _{i,n}/(\sigma _{n}\xi _{i,n}\eta
_{i,n})$ converges to $\zeta _{i}$ and that $W_{n}$ converges to zero in $%
P_{n,\theta ^{(n)},\sigma _{n}}$-probability. Now, if $\left\vert \zeta
_{i}\right\vert <1$, then $P_{n,\theta ^{(n)},\sigma _{n}}\left( \tilde{%
\theta}_{S,i}=0\right) \rightarrow 1$ by Theorem \ref%
{select_prob_moving_par_unknown}(b2), and hence $\sigma _{n}^{-1}\alpha
_{i,n}(\tilde{\theta}_{S,i}-\theta _{i,n})$ converges to $-\zeta _{i}$ in $%
P_{n,\theta ^{(n)},\sigma _{n}}$-probability. This proves the result in case 
$\left\vert \zeta _{i}\right\vert <1$. In case $\left\vert \zeta
_{i}\right\vert >1$ we have that 
\begin{equation*}
P_{n,\theta ^{(n)},\sigma _{n}}\left( \tilde{\theta}_{S,i}\neq 0\right)
\rightarrow 1
\end{equation*}%
and%
\begin{equation}
P_{n,\theta ^{(n)},\sigma _{n}}\left( \limfunc{sign}(W_{n}+\theta
_{i,n}/(\sigma _{n}\xi _{i,n}\eta _{i,n}))=\limfunc{sign}(\zeta _{i})\right)
\rightarrow 1.  \label{sign}
\end{equation}%
Clearly, also $\hat{\sigma}/\sigma _{n}$ converges to $1$ in $P_{n,\theta
^{(n)},\sigma _{n}}$-probability since $n-k\rightarrow \infty $.
Consequently, $\sigma _{n}^{-1}\alpha _{i,n}(\tilde{\theta}_{S,i}-\theta
_{i,n})$ converges to $-\limfunc{sign}(\zeta _{i})$ in $P_{n,\theta
^{(n)},\sigma _{n}}$-probability, which proves the case $\left\vert \zeta
_{i}\right\vert >1$. Finally, if $\left\vert \zeta _{i}\right\vert =1$, then
(\ref{sign}) continues to hold and we can write%
\begin{eqnarray*}
\sigma _{n}^{-1}\alpha _{i,n}(\tilde{\theta}_{S,i}-\theta _{i,n}) &=&\left(
-\zeta _{i}+o(1)\right) \boldsymbol{1}\left( \tilde{\theta}_{S,i}=0\right)
-\left( o_{p}(1)+\left( 1+o_{p}(1)\right) \limfunc{sign}(\zeta _{i})\right) 
\boldsymbol{1}\left( \tilde{\theta}_{S,i}\neq 0\right) \\
&=&-\limfunc{sign}(\zeta _{i})+o_{p}(1),
\end{eqnarray*}%
where $o_{p}(1)$ refers to a term that converges to zero in $P_{n,\theta
^{(n)},\sigma _{n}}$-probability. This then completes the proof of part (b). 
$\blacksquare $

\bigskip

\textbf{Proof of Theorem \ref{ASTconsistent}: }(a) Assume first that $0\leq
\zeta _{i}<\infty $ holds. Note that $z_{n,\theta ^{(n)},\sigma
_{n}}^{(1)}(x,s\eta _{i,n})$ and $z_{n,\theta ^{(n)},\sigma
_{n}}^{(2)}(x,s\eta _{i,n})$ now reduce to 
\begin{equation*}
n^{1/2}\eta _{i,n}\left[ 0.5(x-\theta _{i,n}/(\sigma _{n}\xi _{i,n}\eta
_{i,n}))\pm \sqrt{\left( 0.5(x+\theta _{i,n}/(\sigma _{n}\xi _{i,n}\eta
_{i,n}))\right) ^{2}+s^{2}}\right] .
\end{equation*}%
First, for $x>-\zeta _{i}$ we see that $H_{AS,n,\theta ^{(n)},\sigma
_{n}}^{i\maltese }(x)$ eventually reduces to 
\begin{equation*}
\int_{0}^{\infty }\Phi \left( z_{n,\theta ^{(n)},\sigma _{n}}^{(2)}(x,s\eta
_{i,n})\right) \rho _{m}(s)ds.
\end{equation*}%
Furthermore, for $x\geq 0$ we see that $z_{n,\theta ^{(n)},\sigma
_{n}}^{(2)}(x,s\eta _{i,n})\rightarrow \infty $ for all $s>0$ whereas for $%
-\zeta _{i}<x<0$ we have that $z_{n,\theta ^{(n)},\sigma _{n}}^{(2)}(x,s\eta
_{i,n})\rightarrow \infty $ for $s>\sqrt{-x\zeta _{i}}$ and $z_{n,\theta
^{(n)},\sigma _{n}}^{(2)}(x,s\eta _{i,n})\rightarrow -\infty $ for $s<\sqrt{%
-x\zeta _{i}}$. As a consequence, we obtain from the dominated convergence
theorem that $H_{AS,n,\theta ^{(n)},\sigma _{n}}^{i\maltese }(x)$ converges
to $1$ for $x\geq 0$ and to $\int_{\sqrt{-x\zeta _{i}}}^{\infty }\rho
_{m}(s)ds$ for $-\zeta _{i}<x<0$. Second, for $x<-\zeta _{i}$ note that $%
H_{AS,n,\theta ^{(n)},\sigma _{n}}^{i\maltese }(x)$ eventually reduces to 
\begin{equation*}
\int_{0}^{\infty }\Phi \left( z_{n,\theta ^{(n)},\sigma _{n}}^{(1)}(x,s\eta
_{i,n})\right) \rho _{m}(s)ds
\end{equation*}%
and that $z_{n,\theta ^{(n)},\sigma _{n}}^{(1)}(x,s\eta _{i,n})\rightarrow
-\infty $ for all $s>0$ in this case. This shows that for $x<-\zeta _{i}$ we
have that $H_{AS,n,\theta ^{(n)},\sigma _{n}}^{i\maltese }(x)$ converges to $%
0$. But this proves the result for the case $0\leq \zeta _{i}<\infty $. In
case $\zeta _{i}=\infty $ the same reasoning shows that now $H_{AS,n,\theta
^{(n)},\sigma _{n}}^{i\maltese }(x)$ eventually reduces to 
\begin{equation*}
\int_{0}^{\infty }\Phi \left( z_{n,\theta ^{(n)},\sigma _{n}}^{(2)}(x,s\eta
_{i,n})\right) \rho _{m}(s)ds
\end{equation*}%
for all $x$, and that now for $x>0$ we have $z_{n,\theta ^{(n)},\sigma
_{n}}^{(2)}(x,s\eta _{i,n})\rightarrow \infty $ for all $s>0$ whereas for $%
x<0$ we have that $z_{n,\theta ^{(n)},\sigma _{n}}^{(2)}(x,s\eta
_{i,n})\rightarrow -\infty $ for all $s>0$. This shows that $H_{AS,n,\theta
^{(n)},\sigma _{n}}^{i\maltese }$ converges weakly to $\delta _{0}$ in case $%
\zeta _{i}=\infty $. The proof for the case $\zeta _{i}<0$ is completely
analogous.

(b) Rewrite $\sigma _{n}^{-1}\alpha _{i,n}(\tilde{\theta}_{AS,i}-\theta
_{i,n})$ as%
\begin{eqnarray*}
&&-\theta _{i,n}/(\sigma _{n}\xi _{i,n}\eta _{i,n})\boldsymbol{1}\left( 
\tilde{\theta}_{AS,i}=0\right) \\
&&+\left( \sigma _{n}\xi _{i,n}\eta _{i,n}\right) ^{-1}\left( \hat{\theta}%
_{LS,i}-\theta _{i,n}-\hat{\sigma}^{2}\xi _{i,n}^{2}\eta _{i,n}^{2}/\hat{%
\theta}_{LS,i}\right) \boldsymbol{1}\left( \tilde{\theta}_{AS,i}\neq 0\right)
\\
&=&-\theta _{i,n}/(\sigma _{n}\xi _{i,n}\eta _{i,n})\boldsymbol{1}\left( 
\tilde{\theta}_{AS,i}=0\right) +\left( W_{n}-\left( \hat{\sigma}^{2}/\sigma
_{n}\right) \xi _{i,n}\eta _{i,n}/\hat{\theta}_{LS,i}\right) \boldsymbol{1}%
\left( \tilde{\theta}_{AS,i}\neq 0\right) \\
&=&-\theta _{i,n}/(\sigma _{n}\xi _{i,n}\eta _{i,n})\boldsymbol{1}\left( 
\tilde{\theta}_{AS,i}=0\right) \\
&&+\left( W_{n}-\left( \hat{\sigma}^{2}/\sigma _{n}^{2}\right) \left(
W_{n}+\theta _{i,n}/(\sigma _{n}\xi _{i,n}\eta _{i,n})\right) ^{-1}\right) 
\boldsymbol{1}\left( \tilde{\theta}_{AS,i}\neq 0\right)
\end{eqnarray*}%
where $W_{n}$ is a sequence of $N(0,n^{-1}\eta _{i,n}^{-2})$-distributed
random variables. Note that $\theta _{i,n}/(\sigma _{n}\xi _{i,n}\eta
_{i,n}) $ converges to $\zeta _{i}$ by assumption. Now, if $\left\vert \zeta
_{i}\right\vert <1$, then $P_{n,\theta ^{(n)},\sigma _{n}}\left( \tilde{%
\theta}_{AS,i}=0\right) \rightarrow 1$ by Theorem \ref%
{select_prob_moving_par_unknown}(b2), hence $\sigma _{n}^{-1}\alpha _{i,n}(%
\tilde{\theta}_{AS,i}-\theta _{i,n})$ converges to $-\zeta _{i}$ in $%
P_{n,\theta ^{(n)},\sigma _{n}}$-probability, establishing the result in
this case. Furthermore, for $1\leq \left\vert \zeta _{i}\right\vert \leq
\infty $ rewrite the above display as%
\begin{eqnarray*}
&&\left( -\zeta _{i}+o(1)\right) \boldsymbol{1}\left( \tilde{\theta}%
_{AS,i}=0\right) +\left( o_{p}(1)-\left( 1+o_{p}(1)\right) \left( \zeta
_{i}+o_{p}(1)\right) ^{-1}\right) \boldsymbol{1}\left( \tilde{\theta}%
_{AS,i}\neq 0\right) \\
&=&\left( -\zeta _{i}+o(1)\right) \boldsymbol{1}\left( \tilde{\theta}%
_{AS,i}=0\right) +\left( -\zeta _{i}^{-1}+o_{p}(1)\right) \boldsymbol{1}%
\left( \tilde{\theta}_{AS,i}\neq 0\right) ,
\end{eqnarray*}%
with the convention that $\zeta _{i}^{-1}=0$ in case $\left\vert \zeta
_{i}\right\vert =\infty $. If $\left\vert \zeta _{i}\right\vert >1$
(including the case $\left\vert \zeta _{i}\right\vert =\infty $) then $%
P_{n,\theta ^{(n)},\sigma _{n}}\left( \tilde{\theta}_{AS,i}\neq 0\right)
\rightarrow 1$ by Theorem \ref{select_prob_moving_par_unknown}(b2), and
hence the last display shows that $\sigma _{n}^{-1}\alpha _{i,n}(\tilde{%
\theta}_{AS,i}-\theta _{i,n})$ converges to $-\zeta _{i}^{-1}$ in $%
P_{n,\theta ^{(n)},\sigma _{n}}$-probability, establishing the result in
this case. Finally, if $\left\vert \zeta _{i}\right\vert =1$ holds, then the
last line in the above display reduces to $-\zeta _{i}+o_{p}(1)$, completing
the proof of part (b). $\blacksquare $

\bigskip

\textbf{Proof of Proposition \ref{oracle_1}: }(a) By a subsequence argument
we may assume that $n-k$ converges in $\mathbb{N\cup \{\infty \}}$. Applying
Theorem \ref{select_prob_moving_par_unknown}(b) we obtain that $P_{n,\theta
,\sigma }\left( \tilde{\theta}_{H,i}=0\right) $ converges to $1$ in case $%
\theta _{i}=0$, and to $0$ in case $\theta _{i}\neq 0$. Observe that 
\begin{equation*}
\sigma ^{-1}n^{1/2}\xi _{i,n}^{-1}\left( \tilde{\theta}_{H,i}-\theta
_{i}\right) =-\sigma ^{-1}n^{1/2}\xi _{i,n}^{-1}\theta _{i}
\end{equation*}%
holds on the event $\tilde{\theta}_{H,i}=0$, while 
\begin{equation*}
\sigma ^{-1}n^{1/2}\xi _{i,n}^{-1}\left( \tilde{\theta}_{H,i}-\theta
_{i}\right) =\sigma ^{-1}n^{1/2}\xi _{i,n}^{-1}\left( \hat{\theta}%
_{LS,i}-\theta _{i}\right) =:Z_{n}
\end{equation*}%
holds on the event $\tilde{\theta}_{H,i}\neq 0$. The result then follows in
view of the fact that $Z_{n}$ is standard normally distributed. The proof
for $\hat{\theta}_{H,i}$ is similar using Proposition \ref%
{select_prob_moving_par}(b) instead of Theorem \ref%
{select_prob_moving_par_unknown}(b) (it is in fact simpler as the
subsequence argument is not needed).

(b) Again we may assume that $n-k$ converges in $\mathbb{N\cup \{\infty \}}$%
. By the same reference as in the proof of (a) we obtain that $P_{n,\theta
,\sigma }\left( \tilde{\theta}_{AS,i}=0\right) $ converges to $1$ in case $%
\theta _{i}=0$, and to $0$ in case $\theta _{i}\neq 0$. Now 
\begin{equation*}
\sigma ^{-1}n^{1/2}\xi _{i,n}^{-1}\left( \tilde{\theta}_{AS,i}-\theta
_{i}\right) =-\sigma ^{-1}n^{1/2}\xi _{i,n}^{-1}\theta _{i}
\end{equation*}%
holds on the event $\tilde{\theta}_{AS,i}=0$ and the claim for $\theta
_{i}=0 $ follows immediately. On the event $\tilde{\theta}_{AS,i}\neq 0$ we
have from the definition of the estimator 
\begin{eqnarray*}
\sigma ^{-1}n^{1/2}\xi _{i,n}^{-1}\left( \tilde{\theta}_{AS,i}-\theta
_{i}\right) &=&\sigma ^{-1}n^{1/2}\xi _{i,n}^{-1}\left( \hat{\theta}%
_{LS,i}-\theta _{i}-\hat{\sigma}^{2}\xi _{i,n}^{2}\eta _{i,n}^{2}/\hat{\theta%
}_{LS,i}\right) \\
&=&Z_{n}-\left( \hat{\sigma}/\sigma \right) ^{2}\left( \left( n\eta
_{i,n}^{2}\right) ^{-1}Z_{n}+\sigma ^{-1}\xi _{i,n}^{-1}n^{-1/2}\eta
_{i,n}^{-2}\theta _{i}\right) ^{-1}.
\end{eqnarray*}%
Now, if $\theta _{i}\neq 0$, then the event $\tilde{\theta}_{AS,i}\neq 0$
has probability approaching $1$ as shown above. Hence, we have on events
that have probability tending to $1$%
\begin{eqnarray*}
\sigma ^{-1}n^{1/2}\xi _{i,n}^{-1}\left( \tilde{\theta}_{AS,i}-\theta
_{i}\right) &=&Z_{n}-\left( \hat{\sigma}/\sigma \right) ^{2}\left(
o_{p}(1)+\sigma ^{-1}\xi _{i,n}^{-1}n^{-1/2}\eta _{i,n}^{-2}\theta
_{i}\right) ^{-1} \\
&=&Z_{n}-o_{p}(1),
\end{eqnarray*}%
since $n\eta _{i,n}^{2}\rightarrow \infty $ and $\xi _{i,n}^{-1}n^{-1/2}\eta
_{i,n}^{-2}\rightarrow \infty $ by the assumption and since $\theta _{i}\neq
0$; also note that $\hat{\sigma}/\sigma $ is stochastically bounded since
the collection of distributions corresponding to $\rho _{m}$ with $m\in 
\mathbb{N}$ is tight on $(0,\infty )$ as was noted earlier. The proof for $%
\hat{\theta}_{AS,i}$ is again similar (and simpler) by using Proposition \ref%
{select_prob_moving_par}(b) instead of Theorem \ref%
{select_prob_moving_par_unknown}(b). $\blacksquare $

\section{References}

\quad\ Alliney, S. \& S.~A. Ruzinsky (1994): An algorithm for the
minimization of mixed $l_{1}$ and $l_{2}$ norms with applications to
Bayesian estimation. \emph{IEEE Transactions on Signal Processing \ }42,
618-627.

Bauer, P., P\"{o}tscher, B.~M. \& P.~Hackl (1988): Model selection by
multiple test procedures. \emph{Statistics \ }19, 39--44.

Donoho, D.~L., Johnstone, I.~M., Kerkyacharian, G., D. Picard (1995):
Wavelet shrinkage: asymptopia? With discussion and a reply by the authors. 
\emph{Journal of the Royal Statistical Society Series B} \ 57, 301--369.

Fan, J. \& R. Li (2001): Variable selection via nonconcave penalized
likelihood and its oracle properties. \emph{Journal of the American
Statistical Association} \ 96, 1348-1360.

Fan, J. \& H. Peng (2004): Nonconcave penalized likelihood with a diverging
number of parameters. \emph{Annals of Statistics} \ 32, 928--961.

Feller, W. (1957): \emph{An Introduction to Probability Theory and Its
Applications, Volume 1. }2nd ed., Wiley, New York.

Frank, I.~E. \& J.~H. Friedman (1993): A statistical view of some
chemometrics regression tools (with discussion). \emph{Technometrics \ }35,
109-148.

Ibragimov, I.~A. (1956): On the composition of unimodal distributions. \emph{%
Theory of Probability and its Applications \ }1, 255-260.

Knight, K. \& W. Fu (2000): Asymptotics for lasso-type estimators. \emph{%
Annals of Statistics \ }28, 1356-1378.

Leeb, H. \& B.~M. P\"{o}tscher (2003): The finite-sample distribution of
post-model-selection estimators and uniform versus nonuniform
approximations. \emph{Econometric Theory} {\ 19}, 100--142.

Leeb, H. \& B.~M. P\"{o}tscher (2005): Model selection and inference: facts
and fiction. \emph{Econometric Theory} {\ 21}, 21--59.

Leeb, H. \& B.~M. P\"{o}tscher (2008): Sparse estimators and the oracle
property, or the return of Hodges' estimator. \emph{Journal of Econometrics
\ }142, 201-211.

P\"{o}tscher, B.~M. (1991): Effects of model selection on inference. \emph{%
Econometric Theory} {\ 7}, 163--185.

P\"{o}tscher, B.~M. (2006): The distribution of model averaging estimators
and an impossibility result regarding its estimation. \emph{IMS Lecture
Notes-Monograph Series} \ 52, 113--129.

P\"{o}tscher, B.~M. \& H. Leeb (2009): On the distribution of penalized
maximum likelihood estimators: the LASSO, SCAD, and thresholding. \emph{%
Journal of Multivariate Analysis \ }100, 2065-2082.

P\"{o}tscher, B.~M. \& U. Schneider (2009): On the distribution of the
adaptive LASSO estimator. \emph{Journal of Statistical Planning and
Inference \ }139, 2775-2790.

P\"{o}tscher, B.~M. \& U. Schneider (2010): Confidence sets based on
penalized maximum likelihood estimators in Gaussian regression. \emph{%
Electronic Journal of Statistics} \ 10, 334-360.

Sen, P.~K. (1979): Asymptotic properties of maximum likelihood estimators
based on conditional specification. \emph{Annals of Statistics \ }7,
1019-1033.

Tibshirani, R. (1996): Regression shrinkage and selection via the lasso.\ 
\emph{Journal of the Royal Statistical Society Series B} \ 58, 267-288.

Zhang, C.-H. (2010): Nearly unbiased variable selection under minimax
concave penalty. \emph{Annals of Statistics \ }38, 894-942.

Zou, H. (2006): The adaptive lasso and its oracle properties. \emph{Journal
of the American Statistical Association \ }101, 1418-1429.

\appendix{}

\section{Appendix}

Recall that $\rho _{m}(x)=0$ for $x<0$.

\begin{lemma}
$(2m)^{-1/2}\rho _{m}((2m)^{-1/2}t+1)$ converges to $\phi (t)$ in the $L_{1}$%
-sense as $m\rightarrow \infty $.
\end{lemma}

\begin{proof}
Observe that $(2m)^{-1/2}\rho _{m}((2m)^{-1/2}t+1)$ is the density of $%
U_{m}=(2m)^{1/2}\left( \sqrt{\chi _{m}^{2}/m}-1\right) $ where $\chi
_{m}^{2} $ denotes a chi-square distributed random variable with $m$ degrees
of freedom. By the central limit theorem and the delta-method $U_{m}$
converges in distribution to a standard normal random variable. With 
\begin{equation*}
g_{m}(x)=2^{-m/2}\left( \Gamma (m/2)\right) ^{-1}x^{(m/2)-1}\exp
(-x/2)\qquad \text{for }x>0
\end{equation*}%
being the density of $\chi _{m}^{2}$ we have for $x>0$ 
\begin{eqnarray*}
\rho _{m}(x) &=&2mxg_{m}(mx^{2})=2^{1-m/2}\left( \Gamma (m/2)\right)
^{-1}m^{1/2}\left( mx^{2}\right) ^{(m/2)-1/2}\exp \left( -mx^{2}/2\right) \\
&=&(8m)^{1/2}\Gamma (\left( m+1\right) /2)\left( \Gamma (m/2)\right)
^{-1}g_{m+1}\left( mx^{2}\right) .
\end{eqnarray*}%
and we have $\rho _{m}(x)=0$ for $x\leq 0$. Since the cdf associated with $%
g_{m+1}$ is unimodal, this shows that the same is true for the cdf
associated with $\rho _{m}$. But then convergence in distribution of $U_{m}$
implies convergence of $m^{-1/2}\rho _{m}(m^{-1/2}t+1)$ to $\phi (t)$ in the 
$L_{1}$-sense by a result of Ibragimov (1956), Scheff\'{e}'s Lemma, and a
standard subsequence argument.
\end{proof}

\end{document}